\documentclass{amsart}

\title{A higher Gross-Zagier type formula for moduli of shtukas at deeper level}
\author{Patrick Bieker}

\address{Fakultät für Mathematik, Universität Bielefeld, Postfach 100 131, 33501 Bielefeld, Germany}
\email{pbieker@math.uni-bielefeld.de}

\usepackage[T1]{fontenc} 
\usepackage[utf8]{inputenc} 

\usepackage[all,cmtip]{xy}

\usepackage{enumerate}
\usepackage{enumitem}
\usepackage[colorlinks=true,hyperindex,linktocpage=true]{hyperref}
\hypersetup{
	colorlinks,
	linkcolor={red!50!black},
	citecolor={red!50!black}, 
	urlcolor={red!50!black}, 
	filecolor={red!50!black} 
}
\usepackage[top=0.9in, left=0.9in, right=0.9in, bottom=0.9in]{geometry}
\usepackage{tikz-cd}

\usepackage{amsmath, amssymb, amsthm}
\usepackage{mathtools}
\usepackage{mathabx}
\usepackage{mathrsfs}
\usepackage{stmaryrd}

\usepackage{todonotes}
\usepackage{comment}
\usepackage{theoremref}

\usepackage{lscape}

\makeatletter
\@addtoreset{equation}{section}
\makeatother

\numberwithin{equation}{section}

\newcommand{\dsq}[1]{\left\llbracket #1 \right\rrbracket}

\newcommand{\se}{\subseteq}

\newcommand{\ph}{\varphi}

\newcommand{\defined}{\hspace{0.1cm}\stackrel{\text{\tiny \rm def}}{=}\hspace{0.1cm}}

\newcommand{\Ac}{\mathcal{A}}

\newcommand{\Dc}{\mathcal{D}}
\newcommand{\Ec}{\mathcal{E}}
\newcommand{\Fc}{\mathcal{F}}
\newcommand{\Gc}{\mathcal{G}}
\newcommand{\Hc}{\mathcal{H}}
\newcommand{\Ic}{\mathcal{I}}
\newcommand{\Jc}{\mathcal{J}}
\newcommand{\Kc}{\mathcal{K}}
\newcommand{\Lc}{\mathcal{L}}

\newcommand{\Oc}{\mathcal{O}}
\newcommand{\Pc}{\mathcal{P}}

\newcommand{\Uc}{\mathcal{U}}

\newcommand{\Wc}{\mathcal{W}}

\newcommand{\Yc}{\mathcal{Y}}
\newcommand{\Zc}{\mathcal{Z}}

\newcommand{\A}{\mathbb{A}}
\newcommand{\C}{\mathbb{C}}
\newcommand{\F}{\mathbb{F}}
\newcommand{\G}{\mathbb{G}}

\newcommand{\Q}{\mathbb{Q}}

\newcommand{\Z}{\mathbb{Z}}

\renewcommand{\P}{\mathbb{P}}

\newcommand{\Fq}{{\F_q}}
\newcommand{\Ga}{\mathbb{G}_a}

\newcommand{\pf}{\mathfrak{p}}

\newcommand{\mr}[1]{\mathrm{#1}}

\DeclareMathOperator{\Inv}{Inv}

\DeclareMathOperator{\Hecke}{Hecke}

\DeclareMathOperator*{\colim}{colim}

\usetikzlibrary{arrows}

\newcommand{\GammaN}{\Gamma_0(\Sigma)}

\newcommand{\BDGrdNnaive}{\Gr^{\underline{\mu}}_{2}(\Sig)}
\newcommand{\BunGnaive}{\Bun_{G}(\Sigma)}
\newcommand{\BunGLnaive}{\Bun_{2}(\Sigma)}
\newcommand{\HeckeGGa}{\Hecke^{\underline{\mu}}_{G}(\Sigma)}
\newcommand{\HeckenGa}{\Hecke^{\underline{\mu}}_{2}(\Sigma)}
\newcommand{\ShtGnaive}{\Sht^{\underline{\mu}}_{G}(\Sigma)}
\newcommand{\ShtSiD}{\Sht_2^{\underline{\mu}}(\Sigma, D_\infty)}
\newcommand{\ShtSiDp}{{\Sht'}_2^{\underline{\mu}}(\Sigma, D_\infty)}
\newcommand{\ShtGSiD}{\Sht_G^{\underline{\mu}}(\Sigma, D_\infty)}
\newcommand{\ShtGSiDp}{{\Sht'}_G^{\underline{\mu}}(\Sigma, D_\infty)}
\newcommand{\ShtGSii}{\Sht^r_G(\Sii)}

\newcommand{\DivSSi}{\sD\mathrm{iv}(X; \Sf; \Si)}
\usepackage{amsthm,amsfonts,amssymb,amsmath,amsxtra}
\usepackage[all]{xy}
\SelectTips{cm}{}
\usepackage{xr-hyper}
\usepackage{verbatim}

\usepackage{tikz}

\usepackage{mathrsfs}

\RequirePackage{xspace}
\RequirePackage{etoolbox}
\RequirePackage{varwidth}
\RequirePackage{enumitem}
\RequirePackage{tensor}
\RequirePackage{mathtools}
\RequirePackage{longtable}
\RequirePackage{multirow}

\newcommand{\sC}{\ensuremath{\mathscr{C}}\xspace}
\newcommand{\sD}{\ensuremath{\mathscr{D}}\xspace}

\newcommand{\sH}{\ensuremath{\mathscr{H}}\xspace}

\newcommand{\sL}{\ensuremath{\mathscr{L}}\xspace}

\newcommand{\fkm}{\ensuremath{\mathfrak{m}}\xspace}

\newcommand{\nat}{{\natural}}

\newcommand{\BA}{\ensuremath{\mathbb {A}}\xspace}

\newcommand{\BC}{\ensuremath{\mathbb {C}}\xspace}

\newcommand{\BJ}{\ensuremath{\mathbb {J}}\xspace}

\newcommand{\CC}{\ensuremath{\mathcal {C}}\xspace}

\newcommand{\CO}{\ensuremath{\mathcal {O}}\xspace}

\newcommand{\CW}{\ensuremath{\mathcal {W}}\xspace}

\newcommand{\cA}{\ensuremath{\mathcal {A}}\xspace}

\newcommand{\cE}{\ensuremath{\mathcal {E}}\xspace}
\newcommand{\cF}{\ensuremath{\mathcal {F}}\xspace}
\newcommand{\cG}{\ensuremath{\mathcal {G}}\xspace}
\newcommand{\cH}{\ensuremath{\mathcal {H}}\xspace}
\newcommand{\cI}{\ensuremath{\mathcal {I}}\xspace}

\newcommand{\cL}{\ensuremath{\mathcal {L}}\xspace}
\newcommand{\cM}{\ensuremath{\mathcal {M}}\xspace}
\newcommand{\cN}{\ensuremath{\mathcal {N}}\xspace}
\newcommand{\cO}{\ensuremath{\mathcal {O}}\xspace}

\newcommand{\cQ}{\ensuremath{\mathcal {Q}}\xspace}

\newcommand{\cW}{\ensuremath{\mathcal {W}}\xspace}

\newcommand{\cY}{\ensuremath{\mathcal {Y}}\xspace}

\newcommand{\Ad}{{\mathrm{Ad}}}

\newcommand{\Ch}{{\mathrm{Ch}}}

\DeclareMathOperator{\coker}{coker}

\newcommand{\cl}{{\mathrm{cl}}}

\DeclareMathOperator{\cusp}{cusp}

\newcommand{\Div}{{\mathrm{Div}}}

\DeclareMathOperator{\Eis}{Eis}
\DeclareMathOperator{\End}{End}

\DeclareMathOperator{\Fr}{Fr}

\newcommand{\GL}{\mathrm{GL}}

\DeclareMathOperator{\Hom}{Hom}

\newcommand{\id}{\ensuremath{\mathrm{id}}\xspace}
\let\Im\relax
\DeclareMathOperator{\Im}{Im}

\newcommand{\inv}{{\mathrm{inv}}}

\DeclareMathOperator{\Jac}{Jac}

\DeclareMathOperator{\Nm}{Nm}

\newcommand{\PGL}{{\mathrm{PGL}}}
\DeclareMathOperator{\Pic}{Pic}

\newcommand{\red}{\ensuremath{\mathrm{red}}\xspace}
\newcommand{\reg}{{\mathrm{reg}}}

\DeclareMathOperator{\Res}{Res}

	\DeclareMathOperator{\Spec}{Spec}

	\DeclareMathOperator{\Supp}{Supp}

	\DeclareMathOperator{\tr}{tr}
	\DeclareMathOperator{\Tr}{Tr}

	\DeclareMathOperator{\vol}{vol}

	\newcommand{\Bun}{{\mathrm{Bun}}}

	\newcommand{\Sht}{{\mathrm{Sht}}}
	
	\newcommand{\Hk}{{\mathrm{Hk}}}

	\newcommand{\Gr}{\mathrm{Gr}}

	\newcommand{\Iw}{\mathrm{Iw}}
	
	\newcommand{\Mat}{\mathrm{Mat}}

	\newcommand{\wt}{\widetilde}

	\newcommand{\ov}{\overline}

	\newcommand{\incl}{\hookrightarrow}
	\newcommand{\lra}{\longrightarrow}

	\renewcommand\CC{\mathbb{C}}

	\newcommand\GG{\mathbb{G}}
	
	\newcommand\II{\mathbb{I}}
	\newcommand\JJ{\mathbb{J}}
	\newcommand\KK{\mathbb{K}}

	\newcommand\OO{\mathbb{O}}
	\newcommand\PP{\mathbb{P}}
	\newcommand\QQ{\mathbb{Q}}

	\newcommand\ZZ{\mathbb{Z}}

	\newcommand\bR{\mathbf{R}}

	\newcommand\frD{\mathfrak{D}}

	\newcommand\frK{\mathfrak{K}}

	\newcommand\frS{\mathfrak{S}}
	\newcommand\frT{\mathfrak{T}}

	\newcommand\frZ{\mathfrak{Z}}

	\renewcommand\a\alpha
	\renewcommand\b\beta
	\newcommand\g\gamma
	\renewcommand\d\delta
	\newcommand\D\Delta
	
	\renewcommand{\k}{\kappa}
	\renewcommand{\th}{\theta}

	\newcommand{\s}{\sigma}
	\newcommand{\Sig}{\Sigma}
	
	\renewcommand\r{\rho}
	\newcommand\io{\iota}

	\newcommand{\z}{\zeta}
	
	\renewcommand{\l}{\lambda}
	
	\newcommand{\om}{\omega}

	\renewcommand\j{\jmath}
	
	\newcommand\xr{\xrightarrow}
	\newcommand{\isom}{\stackrel{\sim}{\to}}
	\newcommand{\surj}{\twoheadrightarrow}

	\newcommand\ha{\frac{1}{2}}

	\newcommand{\Ql}{\QQ_{\ell}}
	\newcommand{\Qlbar}{\overline{\QQ}_\ell}
	\newcommand{\kbar}{\overline{k}}

	\newcommand\un{\underline}
	
	\newcommand\op{\oplus}
	\newcommand\ot{\otimes}
	
	\newcommand{\sss}{\subsubsection}
	
	\newcommand{\coho}[2]{\mathbf{H}^{#1}({#2})}    
	\newcommand{\cohog}[2]{\textup{H}^{#1}({#2})}     
	\newcommand{\cohoc}[2]{\textup{H}_{c}^{#1}({#2})}     
	
	\newcommand{\oll}[1]{\overleftarrow{#1}}
	\newcommand{\orr}[1]{\overrightarrow{#1}}
	\newcommand{\olr}[1]{\overleftrightarrow{#1}}

	\newcommand\AL{\textup{AL}}
	
	\newcommand{\Gm}{\GG_m}

	\newcommand\inst{\textup{inst}}
	
	\newcommand\Sect{\textup{Sect}}

	\newcommand\dm{\diamondsuit}
	
	\newcommand\da{\dagger}
	\newcommand\sh{\sharp}
	
	\newcommand\fl{\flat}

	\newcommand{\sqR}{\sqrt{R}}

	\newcommand{\Si}{\Sigma_{\infty}}
	\newcommand{\Sf}{\Sigma_{f}}
	\newcommand{\Sii}{\Sig; \Sigma_{\infty}}
	\newcommand{\SIw}{\Sig_{\Iw}}
	\newcommand{\frSi}{\frS_{\infty}}
	\newcommand{\Di}{D_{\infty}}
	\newcommand{\mi}{\mu_{\infty}}
	\newcommand{\mf}{\mu_{f}}
	\newcommand{\Sigha}{\lfloor \Sig/2\rfloor}
	\newcommand{\SigHa}{\left\lfloor \frac{\Sig}{2} \right\rfloor}
	
	\newcommand{\bsq}{\square}
	\newcommand\bx{\mathbf{x}}

	\newtheorem{theorem}{Theorem}
	\newtheorem{prop}[theorem]{Proposition}
	\newtheorem{lem}[theorem]{Lemma}
	\newtheorem{lemma}[theorem]{Lemma}
	
	\newtheorem{cor}[theorem]{Corollary}
	\newtheorem{thm}[theorem]{Theorem}

	\theoremstyle{definition}
	\newtheorem{defn}[theorem]{Definition}

	\newtheorem{remark}[theorem]{Remark}

	\newtheorem{proposition}[theorem]{Proposition}
	\newtheorem{corollary}[theorem]{Corollary}

	\newtheorem{definition}[theorem]{Definition}

	\newcommand{\matrixx}[4]
	{\left[ \begin{array}{cc}
			#1 &  #2  \\
			#3 &  #4  \\
		\end{array}\right]}

	\numberwithin{equation}{section}
	\numberwithin{theorem}{section}

	\renewcommand{\to}{%
		\ifbool{@display}{\longrightarrow}{\rightarrow}%
	}
	\let\shortmapsto\mapsto
	\renewcommand{\mapsto}{%
		\ifbool{@display}{\longmapsto}{\shortmapsto}%
	}
	\newlength{\olen}
	\newlength{\ulen}
	\newlength{\xlen}
	\newcommand{\xra}[2][]{%
		\ifbool{@display}%
		{\settowidth{\olen}{$\overset{#2}{\longrightarrow}$}%
			\settowidth{\ulen}{$\underset{#1}{\longrightarrow}$}%
			\settowidth{\xlen}{$\xrightarrow[#1]{#2}$}%
			\ifdimgreater{\olen}{\xlen}%
			{\underset{#1}{\overset{#2}{\longrightarrow}}}%
			{\ifdimgreater{\ulen}{\xlen}%
				{\underset{#1}{\overset{#2}{\longrightarrow}}}
				{\xrightarrow[#1]{#2}}}}%
		{\xrightarrow[#1]{#2}}
	}
	\makeatother
	\newcommand{\xyra}[2][]{%
		\settowidth{\xlen}{$\xrightarrow[#1]{#2}$}%
		\ifbool{@display}%
		{\settowidth{\olen}{$\overset{#2}{\longrightarrow}$}%
			\settowidth{\ulen}{$\underset{#1}{\longrightarrow}$}%
			\ifdimgreater{\olen}{\xlen}%
			{\mathrel{\xymatrix@M=.12ex@C=3.2ex{\ar[r]^-{#2}_-{#1} &}}}%
			{\ifdimgreater{\ulen}{\xlen}%
				{\mathrel{\xymatrix@M=.12ex@C=3.2ex{\ar[r]^-{#2}_-{#1} &}}}
				{\mathrel{\xymatrix@M=.12ex@C=\the\xlen{\ar[r]^-{#2}_-{#1} &}}}}}%
		{\mathrel{\xymatrix@M=.12ex@C=\the\xlen{\ar[r]^-{#2}_-{#1} &}}}%
	}
	\makeatletter
	\newcommand{\xla}[2][]{%
		\ifbool{@display}%
		{\settowidth{\olen}{$\overset{#2}{\longleftarrow}$}%
			\settowidth{\ulen}{$\underset{#1}{\longleftarrow}$}%
			\settowidth{\xlen}{$\xleftarrow[#1]{#2}$}%
			\ifdimgreater{\olen}{\xlen}%
			{\underset{#1}{\overset{#2}{\longleftarrow}}}%
			{\ifdimgreater{\ulen}{\xlen}%
				{\underset{#1}{\overset{#2}{\longleftarrow}}}
				{\xleftarrow[#1]{#2}}}}%
		{\xleftarrow[#1]{#2}}
	}
	\newcommand{\isoarrow}{%
		\ifbool{@display}{\overset{\sim}{\longrightarrow}}{\xrightarrow\sim}%
	}
	\renewcommand{\lra}{%
		\ifbool{@display}{\longleftrightarrow}{\leftrightarrow}%
	}

\thanks{
	The author was supported by the Deutsche Forschungsgemeinschaft (DFG, German Research Foundation) --
	Project-ID 520675682 as well as by the Deutsche Forschungsgemeinschaft (DFG, German Research Foundation) --
	Project-ID 491392403 -- TRR 358.
}

\begin{document}

\begin{abstract}
	We generalize the function field analogue analogue of the Gross-Zagier and Waldspurger formulae due to Yun and Zhang relating the (higher) derivative of an automorphic (base change) $L$-function to intersection numbers of Heegner-Drinfeld cycles on moduli spaces of shtukas to not necessarily square-free level.
	The main new input is the study of the geometry and cohomology of the appropriate ``integral models'' of moduli spaces of shtukas at deeper level.
\end{abstract}

\maketitle

\section{Introduction}

In their seminal works \cite{Yun2017, Yun2019}, Yun and Zhang show (higher order) analogues of the Waldspurger \cite{Waldspurger1985} and Gross-Zagier \cite{Gross1986} formulae in the function field setting.
More precisely, for $G = \PGL_2$ defined over the function field $F$ of a smooth projective and geometrically connected curve $X$ over a finite field $\Fq$ of characteristic $p \neq 2$ the Yun-Zhang formula relates the higher order central derivatives of the base change $L$-function $\sL_{F'/F}(\pi, s)$ of a cuspical automorphic representation $\pi$ of $G(\A_F)$ with square-free level $\Sigma$ on the one hand with the self-intersection number of a Heegner-Drinfeld cycle (an analogue of Heegner points on modular curves) on a moduli space of $G$-shtukas with level $\Gamma_0(\Sig)$ on the other hand. 
There is hope to use these higher Gross-Zagier type formulas to gain new insight into the Birch and Swinnerton-Dyer conjecture in the function field setting even beyond the rank 1 case, compare the discussion in \cite[Section 1.3]{Yun2019} for details.

The goal of the present work is to extend their results to  not necessarily square-free level. The main new input is to introduce adequate moduli space of $G$-shtukas at these deeper levels that supports the Heegner-Drinfeld cycles in this setting, and to study their geometry and cohomology.

\subsection{Statement of the main results}
Let $X$ be a smooth projective and geometrically connected curve of genus $g$ over a finite field $k = \Fq$ of characteristic $p \neq 2$ and let $F = F(X)$ be its function field.
Let $\Sigma$ be a (not necessarily reduced) effective divisor of degree $N$ on $X$ and let $|\Sigma| \se |X|$ be its set of closed points.
We fix a decomposition 
$|\Sigma| = \Sf \amalg \Si$
such that $\Sigma$ is reduced at all points in $\Sigma_\infty$.
We call $\Sf$ (respectively $\Si$) the set of finite (respectively infinite) places. 
We say that a quadratic cover $\nu \colon X'\to X$, where $X'$ is a second smooth projective and geometrically connected curve over $\Fq$, satisfies the \emph{generalized Heegner hypothesis} if
\begin{align}
	\label{eq:Heegner-hypothesis}
	\nu \text{ splits at every point in } \Sigma_f \text{ and is inert at all points in } \Sigma_\infty.
\end{align}
In particular, the ramification locus of $\nu$ is disjoint from $\Sigma$. We denote the ramification degree by $\r$.
We define a moduli space $\Sht^r_G(\Sii)$ of $G$-shtukas with $\Gamma_0(\Sig)$-level structure, $r$ modifications and additional supersingular legs at the infinite places $\Si$. 
It comes with a natural projection map $\Pi_G^r \colon \Sht^r_G(\Sii) \to X^r \times \frSi$ and can thus be thought of as an analogue of an integral model of the modular curve for not necessarily square-free level structure. Here, $\frSi = \prod_{x \in \Si} \Spec(k(x))$ is the product of the residue fields of the infinite places.
Our main result is the following generalization of \cite[Theorem 1.6]{Yun2017} (in the case $\Sig = \emptyset$) and \cite[Theorem 1.2]{Yun2019} (in the case $\Sig$ square-free) to deeper level.

\begin{theorem} 
	\thlabel{thm:int-period}
	Let $\pi = \otimes_{x \in |X|} \pi_x$ be a cuspidal automorphic representation of $\PGL_2(\A_F)$ of conductor $\Sig$. 
	Let $\nu \colon X' \to X$ be a quadratic cover such that the generalized Heegner hypothesis \eqref{eq:Heegner-hypothesis} is satisfied. 
	Moreover, let $r$ be a non-negative integer with $r \equiv |\Si| \ \text{ mod } 2$. 
	Then
	\begin{equation}\label{eq:main-thm}
		\frac{|\om_{X}| q^{\rho/2 - N} \cdot \wt{L}_{\Sig}(\pi)}{2(- \log q)^{r}\cdot L(\pi, \Ad, 1) } \cdot \sL_{F'/F}^{(r)}(\pi, 0) = \left( Z_\pi^{\mu}(\xi), Z_\pi^{\mu}(\xi) \right)_{\Sht'^r_G(\Sig;\Si)}.
	\end{equation}
\end{theorem} 
Let us briefly explain the terms appearing.
Here, $Z_\pi^{\mu}(\xi) \in H^{2r}_c(\Sht'^r_G(\Sig;\xi), \Qlbar)$ is the $\pi$-isotypic component of a Heegner-Drinfeld cycle defined on the base change $\Sht'^r_G(\Sig;\xi) \defined \Sht^r_G(\Sig;\Si) \times_{X^r \times \frSi} X'^r \times \Spec(\bar k)$ where $\xi \in \frSi(\bar k)$ for an algebraic closure $\bar k$ of $k$. 
Then $\left( \cdot \, , \,  \cdot \right)_{\Sht'^r_G(\Sig;\Si)}$ is the cup product pairing on compactly supported cohomology of $\Sht'^r_G(\Sig;\xi)$.

On the other side, $\sL_{F'/F}(\pi, s)$ is the base change $L$-function of $\pi$ normalized such that 0 is the center of its functional equation.
The value $\wt L_\Sig(\pi)$ is essentially a normalized product of local Rankin-Selberg $L$-values of the product of the local factor $\pi_x$ of $\pi$ at $x$ with its contragradient at the places $x \in \Sig$ that occur with multiplicity bigger than 1. 
Moreover, $|\omega_X| = q^{-(2g-2)}$ corresponds to 
the degree of the canonical bundle of $X$.

\subsubsection{Moduli spaces of shtukas}

The classical Gross-Zagier formula relates the value of the  first derivative of the (base-change) $L$-function of an elliptic curve over $\mathbb Q$ to the N\'eron-Tate height of a Heegner point on the elliptic curve and is used to establish cases of the Birch and Swinnerton-Dyer conjecture in rank 1. 
Equivqalently, it relates the value of the first derivative of the $L$-function of a modular form of level $\Gamma_0(N)$ to an arithmetic self-intersection number of a certain Heegner divisor on the integral model of the modular curve of level $\Gamma_0(N)$, which is defined using Drinfeld level structures when $N$ is not square-free.

A generalization of the Yun-Zhang formula to not necessarily square-free level, as for the modular curve, seemed to be out of reach due to the lack of an appriopriate ``integral model'' that supports the special cycles. 
In particular, the approach using Drinfeld level structures for shtukas as defined in \cite{Bieker2023} does not seem to have a good analogue in the situations where multiple legs collide.
Instead, we use a more naive approach and define the moduli space of shtukas $\Sht^r_G(\Sii)$ as a moduli space of shtukas for a smooth affine group scheme on $X$ encoding the $\Gamma_0(\Sigma)$-level structure (e.g. coming from Bruhat-Tits theory).
In the case with Iwahori level structure this recovers the stack of the same name used by Yun and Zhang.

When the level is not square-free, in the situation with both a single finite and infinite leg (which is the most direct analogue of the modular curve in this setting) we also do not recover the integral model defined using Drinfeld level structures, but rather an open substack therein. 
In particular, in the fibers over points with higher multiplicity our $\Sht^1_G(\Sii)$ is missing several components that appear when using Drinfeld level structure, including all supersingular points.
More precisely, it is the smooth locus of the special fiber of the moduli space of shtukas with Drinfeld level structure. 
As the Heegner-Drinfeld cycles are supported only on the remaining (open) components (compare \cite[III, Proposition 3.1]{Gross1986} for the corresponding statement for the modular curve), this difference is not material for us.  
For a more detailed discussion of the different "integral models" at deeper level compare also \cite[Remark 2.20]{Bieker2023}.

\begin{theorem}
	The moduli space of shtukas  $\Sht^r_G(\Sii)$ is smooth over $k$ of dimension $2r$. 
	Moreover, the projection $\Sht^r_G(\Sii) \to X^r \times \frSi$ is flat of relative dimension $r$ and smooth over $(X \setminus \Sigma_{\Iw})^r \times \frSi$. 
	\thlabel{thm:geometry-intro}
\end{theorem}
Here $\Sigma_{\Iw} \se X$ is the set of points in $\Sig$ that appear with multiplicity 1. In particular, the fibers above points with higher multiplicity are smooth.

The spherical Hecke algebra $\sH^\Sig_G$ away from the level $\Sigma$ acts on the compactly supported cohomology $H^\bullet_c(\Sht^r_G(\Sii), \Qlbar)$. 
As in the unramified or square-free level situations (compare \cite[Theorem 1.1]{Yun2019}) we have a spectral decomposition of the middle degree cohomology.
We define a deeper level analogue  $\wt I^{\Sig}_{\rm Eis} \se \sH^\Sig_G$ of the Eisenstein ideal of \emph{loc.~cit.~}in the square-free situation. For a maximal ideal $\mathfrak{ m}$ of $\sH^\Sig_G$ not contained in the Eisenstein ideal we denote by $V(\xi)_{\mathfrak{ m}} \se H^{2r}_c(\Sht^r_G(\Sig, \xi), \Qlbar)$ the generalized eigenspace for $\mathfrak{ m}$ under the Hecke action.
\begin{theorem}\thlabel{thm:spectral-dec}
	The middle degree cohomology group $H^{2r}_c(\Sht^r_G(\Sig, \xi), \Qlbar)$ decomposes canonically as $\sH^\Sig_G$-module as
	\[
		H^{2r}_c(\Sht^r_G(\Sig, \xi), \Qlbar) = \left( \bigoplus_{\mathfrak{ m} \not \supset \cI_{\Eis}} V(\xi)_{\mathfrak{ m}} \right) \oplus V(\xi)_{\rm Eis}
	\]
	such that
	\begin{itemize}
		\item  the space $V(\xi)_{\mathfrak{ m}}$ is finite dimensional and non-zero only for finitely many maximal ideal $\mathfrak{ m}$ of $\sH^\Sig_G$ (assumed not to be contained in the Eisenstein ideal $\wt I^{\Sig}_{\rm Eis}$)
		\item $V(\xi)_{\Eis}$ is a finitely generated module over $\sH^\Sig_G$ on which the Hecke action of some power of $\wt I^{\Sig}_{\rm Eis}$ is trivial.  
	\end{itemize}
	
\end{theorem}

To define the Heegner-Drinfeld cycles we consider the moduli stack of shtukas $\Sht_T^{\un \mu}(\mu_\infty \Si)$ of rank 1 on the double cover $X'$ of $X$ as defined in \cite{Yun2019} associated to $\un \mu \in \{\pm 1\}^r$ and $\mi \in \{\pm 1\}^{\Si}$.
It is a finite \'etale cover of $X'^r \times \frSi'$, and in particular proper and smooth of dimension $r$ over $k$. 

To define the Heegner-Drinfeld cycles we additionally choose a section $\mf$ of $\nu \colon \nu^{-1}(\Sf) \to \Sf$, recall that $\nu$ splits along $\Sf$. For $\mu = (\un \mu, \mf , \mi)$ we then construct a map
\[
	\theta'^{\mu} \colon \Sht_T^{\un \mu}(\mu_\infty \Si) \to \Sht_G^{r}(\Sii) \times_{(X^r \times \frSi)} (X'^r \times \frSi'), 
\] 
which can also be interpreted as an instance of functoriality for moduli spaces of shtukas of \cite{Breutmann2019, Yun2022} under the (pseudo-)closed immersion $T = \Res_{X'/X} \Gm \hookrightarrow \cG_{\GammaN}$ of group schemes over $X$.

The associated \emph{Heegner-Drinfeld cycle} is $\Zc^{\mu} \defined \theta'^\mu_*[\Sht_{T}^{\un \mu}(\mu_\infty \Si)] \in \Ch_{c,r}(\Sht'^r_G(\Sii))$, and its cycle class is denoted by $Z^\mu \defined \cl(\Zc^{\mu}) \in H^{2r}_c(\Sht'^r_G(\Sii))_\Q$. Moreover, for a maximal ideal $\mathfrak{m}$ in the Hecke algebra away from the Eisenstein ideal we denote by $Z^{\mu}_{\mathfrak{m}}$ its $\mathfrak{m}$-isotypical component under the spectral decomposition.
In particular, a cuspidal automorphic representation $\pi$ of $G(\A)$ of conductor $\Sig$ gives rise to such a maximal ideal $\mathfrak{m}_\pi$, and we use the notation $Z^\mu_\pi = Z^\mu_{\mathfrak{m}_\pi}$ in this case. 
This is the class appearing in the main theorem above. 

%

\subsection{Overview and proof strategy}
Let us give an overview over the parts of the arguments that are new in our setting, compared to the unramified or square-free level cases treated in the literature. 

In Section \ref{sect:analysis}, we introduce the analytic setup and carry out the necessary computations.
The main novelty here is the introduction of the deeper level analogue of the Eisenstein ideal $\wt I^{\Sig}_{\rm Eis}$ in the Hecke algebra, which together with some explicit calculation of certain $L$-values in the deeper level case is used to study the spectral decomposition of the global spherical character. 
In particular, we define a local test function that serves at places with deeper level structure and compute the corresponding local spherical characters.

In Section \ref{sect:shtukas}, we study the moduli space of shtukas with not necessarily square-free level structure, prove the relevant geometric properties as announced in \thref{thm:geometry-intro}. We define the Heegner-Drinfeld cycles at deeper level and sketch how to translate the intersection problem into computing traces of Frobenius-Hecke operators on a certain Hitchin-type fibration.

In Section \ref{sect:cohomology}, we generalize the finiteness results and spectral decomposition on cohomology to our setting with not necessarily square-free level and use them to finish the proof of our main theorem.

Generally speaking, for the remaining parts of the argument we closely follow the strategy of \cite{Yun2017, Yun2019} and only give proofs in detail that are new in our setting and at most sketch the parts of the arguments that work essentially verbatim as in \emph{loc.~cit}.
In particular, the treatment of the two Hitchin-type fibrations and their comparison does not need to be modified in an essential way in the deeper level situation and we refer to \emph{loc.~cit} for details.

\subsection*{Acknowledgements}
The author thanks Zhiwei Yun for his suggestion to consider intersection problems on moduli spaces of shtukas at deeper level structure as well as his support. 
It is pleasure to also thank Lennart Gehrmann, Urs Hartl, Timo Richarz and Wei Zhang for helpful discussions surrounding the topic of this work.
The author thanks all participants of the workshop ``Shtukas and Higher Gross-Zagier Formulas'' in Darmstadt in March 2023.


\section{Automorphic representations, L-functions and the relative trace formula}
\label{sect:analysis}

We introduce the relevant notions surrounding automorphic representations, Hecke algebras and their relations to $L$-functions.
In particular, we introduce the modification of the test functions of \cite{Yun2019} at points with deeper level and compute the resulting local factors of the global spherical character in \thref{prop:J-pi-deeper-level} and \thref{thm:global-char}. 

We continue to use the notation from the introduction and denote by $X$ a smooth projective and geometrically connected curve over the finite field $k = \Fq$ of characteristic $p \neq 2$. 
Let $F$ be its function field and $|X|$ its set of closed points.
For a closed point $x \in |X|$ let $\Oc_x$ be the completion of its local ring $\Oc_{X,x}$ and let $\mathfrak{m}_x$ be its maximal ideal.
Let $\A = \A_F$ be the ring of adeles of $F$ and let $\OO = \OO_F \defined \prod_{x \in |X|} \Oc_x \se \A$ be the ring of integral adeles.

\subsection{Group theoretic setup}
We consider the reductive $F$-groups $\tilde{G} = \GL_{2,F}$ and $G = \PGL_{2, F}$ with their  diagonal tori $\tilde A \se \tilde G$ and $A \se G$.
For a quadratic extension $F'/F$ we have the corresponding non-split tori $\tilde T \defined \Res_{F'/F} \Gm \se \tilde G$ and $T \defined \tilde T / \Gm \se G$.

\sss{Haar measures}
For any smooth affine group $H$ over $F$ let $[H] \defined H(F) \setminus H(\A)$.
When $H$ is split reductive, we equip $H(\A)$ with the Haar measure normalized (unless otherwise stated) such that $H(\OO)$ has measure 1, and let $[H]$ have the quotient measure.

\sss{Level $\Sigma$}
\label{sec:level-sig}
We fix a (not necessarily reduced) divisor $\Sigma$ in $X$ as in the introduction and denote by $|\Sigma| \subset |X|$ its set of closed points.
Let $\Sigma_{\Iw} \se |\Sigma|$ be the set of points at which $\Sigma$ is reduced.
Let 
\begin{equation*}
	K_0(\Sigma) \defined \ker(G(\OO) \to \prod_{x \in |\Sigma|} B(\Oc_x/ \mathfrak{p}_x^{n_x} \Oc_x)),	
	\label{eq:definition-gamma0}
\end{equation*}
where $B \se \PGL_2$ is the subgroup of upper triangular matrices.
Then $K_0(\Sigma)$ is a compact open subgroup of $G(\A)$. We can write $K_0(\Sigma) = \prod_{x \in |X| \setminus \Sigma} G(\Oc_x) \times \prod_{x \in \Sigma} \Gamma_0(\pf_x^{n_x})$, where $\Gamma_0(\pf_x^{n_x}) \defined \ker(G(\Oc_x) \to G(\Oc_x/\pf_x^{n_x}))$.
We define $\tilde K_0(\Sigma) \se \tilde{G}(\A_F)$ similarly.

\subsection{Automorphic representations and Hecke algebras}
%
Let 
$
	\sH_x \defined \sC_c^\infty(G(\Oc_x) \setminus G(F_x)/ G(\Oc_x), \Q)
$
be the spherical local Hecke algebra at $x$.
Moreover, for a finite set of places $S \se |X|$ let $\sH^{S} \defined \otimes_{x \in |X| \setminus S}' \sH_x$.

Let $\pi$ be an automorphic representation of $G(\A)$.
Then $\pi$ factors as $\pi = \bigotimes_{x \in |X|}' \pi_x$ into its local factors. 
For every $x \in |X|$ let $n_x$ be its conductor exponent, i.e., the smallest non-negative number such that $\pi_x^{\Gamma_0(\pf_x^{n_x})} \neq 0$. 
Moreover, $\pi_x^{\Gamma_0(\pf_x^{n_x})}$ is one-dimensional in this case.
Then $n_x = 0$ for almost all $x \in |X|$ and we set $\Sigma \defined \sum_{x \in |X|} n_x [x]$ and $\dim(\pi^{K_0(\Sigma)})$ is one-dimensional. 
In particular, the Hecke algebra $\sH^{|\Sigma|}$ acts on $\pi^{K_0(\Sigma)}$ by a character
\[
	\lambda_\pi \colon \sH^{|\Sig|} \to \C.
\]

From the discussion in \cite[§3.1]{Yun2017} we get that $\sH^S$ has a basis by functions $h_D$ indexed by divisors $D = \sum_{x \in |X|} m_x [x] \in \Div(X - S)$ given by $h_D = \otimes_{x \in |X|} h_{m_x [x]}$ where $h_{m_x [x]}$ is the characteristic function of 
\[
	M_{m_x, x} \defined \Mat_{2}(\Oc_x)_{v_x(\det) = n_x} \se G(F_x).
\]
%

\sss{Local test functions}
\label{sec:test-function}
Recall that \cite[§2.4]{Yun2019} define local test functions for places $x \in |R \cup \Sig_\Iw|$:
\begin{itemize}
	\item For $x \in R$ they define two function $h_x^\bsq$ and $f_x^\bsq$ with matching orbital integrals.
	We do not recall the definition of these functions but refer to \emph{loc.~cit.} for more details. 
	\item 
	For $x \in \Sig_\Iw$ the test function is given by the characteristic function of the Iwahori  subgroup $\Iw_x = \Gamma_0(\pf_x)$ at $x$.
\end{itemize}  
At places $x \in |\Sigma| - | \Sigma_\Iw|$ we define the local test function to be 
\[
	f_x \defined \mathbf{1}_{\Gamma_0(\pf_x^{n_x})}.
\]
To summarize, for a Hecke function  $f \in \sH^{\Sigma \cup R}$ let 
$$ f^{\Sig} \defined f \otimes \left( \bigotimes_{x \in R} h_x^\bsq \right) \otimes \left( \bigotimes_{x \in \Sigma} \mathbf{1}_{\Gamma_0(\pf_x^{n_x})} \right) \in \sC_c^\infty(G(\A)).$$

\subsection{A modified Eisenstein ideal}

Let $S \se |X|$ be a finite set of places of $X$. 
Recall from \cite[§4.1]{Yun2017} and \cite[§2.2]{Yun2019} that the Eisenstein ideal $\Ic_{\Eis}^S$ is defined as the kernel of the map
$$ a_{\Eis}^S \colon \sH^S \to \Q[\Div(X \setminus S)] \twoheadrightarrow \Q[\Pic_X(k)]$$
induced by the Satake transform. 
We use a modified version of the Eisenstein ideal to account for points with deeper level structures. 
For the (effective) divisor $\Sigma = \sum_{x \in |X|} n_x [x]$ we set  
\[
	\SigHa \defined \sum_{x \in |X|} \left \lfloor \frac{n_x}{2} \right\rfloor [x] \in \Div(X).
\]
In particular, when $\Sigma$ is square-free, $\Sigha = 0$. 
We consider the Picard stack $\Pic_X^{\Sigha}$ with $\Sigha$-level structure that parametrizes line bundles $\cL$ on $X$ together with a trivialization $\theta \colon \Lc|_{\Sigha} \cong \Oc_{\Sigha}$ over $\Sigha$. 
By Weil's uniformization theorem, 
\[
	\Pic_X^{\Sigha}(k) \cong F^\times \backslash \A^\times /K(\Sigha),
\]
where 
\[
 	K(\Sigha) = \prod_{x \in |X| - |\Sig|} \Oc_x^\times \times \prod_{x \in |\Sig|} K_x\left(\pf_x^{\lfloor n_x/2 \rfloor}\right) 
\] 
with 
\[
	K_x\left(\pf_x^{\lfloor n_x/2 \rfloor}\right) = \begin{cases}
		\Oc_x^\times & 0 \leq n_x \leq 1, \\
		1 + \pf_x^{\lfloor n_x/2 \rfloor} \Oc_x & \text{else}
	\end{cases}.
\]   

The Picard stack with level structure has a forgetful map $\Pic^{\Sigha}_X \to \Pic_X$ and the involution $\iota_{\Pic}$ on $\Q[\Pic_X(k)]$ defined in \cite[§4.1.2]{Yun2017} lifts to an involution $\iota_{\Pic}^{\Sigha}$ on $\Q[\Pic^{\Sigha}_X(k)]$ defined by 
\[
	\mathbf{1}_{(\Lc, \theta)} \mapsto q^{\deg \Lc} \cdot \mathbf{1}_{(\Lc^\vee, \theta')}
\]  
where $\theta'$ is defined as the trivialisation of $\Lc^\vee|_\Sigha$ induced by $\theta$. Moreover, by composing the forgetful map with the degree map we get a short exact sequence
\[
	0 \rightarrow \Pic_X^{\Sigha, 0}(k) \rightarrow \Pic_X^{\Sigha}(k) \rightarrow \Z \to 0.
\]
Note that $\Pic_X^{\Sigha, 0}(k)$ is a finite abelian group as an extension of the finite group $\Pic_X^0(k) \cong \Jac_X(k)$ by $\Oc_{\Sigha}^\times$. 

Moreover, the map $a^\Sig_{\Eis}$ factors through
\begin{align*}
	\wt a^\Sig_{\Eis} \colon \sH^{|\Sig|} \xrightarrow{\rm{Sat}} \Q[\Div(X - \Sig)] & \to \Q\left[ \Pic_X^{\Sigha}(k)\right]\\
	D & \mapsto (\Oc_X(D), 1)
\end{align*}
using that $\Oc_X(D)$ has a canonical trivialization along $\Sig$ as the two divisors $\Sig$ and $D$ are disjoint by assumption.
\begin{definition}
	The modified Eisenstein ideal $\wt \Ic^{\Sig}_{\Eis}$ is the kernel of the map
	\[
		\wt a^{\Sig}_{\Eis} \colon \sH^{|\Sig|} \to \Q\left[ \Pic^{\Sigha}_X(k) \right].
	\]
\end{definition}
Note that in contrast to the square-free case, the modified Eisenstein ideal depends on the multiplicities in the divisor $\Sigma$ and not just its set of places.
In general, $\wt \Ic^{\Sig}_{\Eis} \se \Ic^{|\Sig|}_{\Eis}$ with equality in the square-free case.
The definition is done in such a way that for an automorphic representation $\pi$ with level at most $\Sig$ its character $\lambda_\pi$ does not factor through $\wt a_{\Eis}^\Sig$ if and only if $\pi$ is cuspidal.

We extend the finiteness results from the square-free case in \cite[Lemma 4.2]{Yun2017} and \cite[Lemma 2.1]{Yun2019}.

\begin{lemma}
	Let $\Sig$ as before and fix $x \in |X|- |\Sig|$.
	\begin{enumerate}
		\item The algebra $\Q\left[ \Pic^{\Sigha}_X(k) \right]^{\iota_{\Pic}^{\Sigha}}$ is finitely generated as $\sH_x$-module via $\wt a^\Sig_{\Eis}$.
		\item The map $\wt a^{\Sig}_{\Eis}$ is surjective. In other words, the induced map 
		\[
			\Spec\left(\wt a^\Sig_{\Eis}\right) \colon \Spec\left(\Q\left[ \Pic^{\Sigha}_X(k) \right]^{\iota_{\Pic}^{\Sigha}} \right) \to \Spec\left(\sH^\Sig\right)
		\] 
		is a closed immersion.
	\end{enumerate}
\end{lemma}
\begin{proof}
	We adapt the proofs of \cite[Lemma 4.2]{Yun2017}. The first part follows analogously to \cite[Lemma 4.2 (1)]{Yun2017} by using the finiteness of $\Pic_X^{\Sigha, 0}(k)$. 
	
	For the second part (using the finiteness of the map) it suffices to show that the map is injective on $\Qlbar$-points and induces injections on all tangent spaces at $\Qlbar$-points.
	The injectivity on $\Qlbar$-points works as in the proof of \cite[Lemma 4.2 (2)]{Yun2017} by using that 
	$
		\Pic_X^{\Sigha}(k) \cong F^\times \backslash \A^\times / K(\Sigha)
	$
	is by global class field theory identified with the Weil group $W_F(\Sigha)$ of the ray class field for the divisor $\Sigha$. 
	Hence, closed points of $\frZ_{\Eis}^\Sig = \Spec\left(\Qlbar\left[ \Pic^{\Sigha}_X(k) \right]^{\iota}\right)$ can be identified with $\Qlbar^\times$-valued characters of $W_F(\Sigha)$ up to $\chi \sim \chi^{-1}(-1)$, where $(-1)$ is the Tate twist.
	
	To show that the induced map on tangent spaces is injective, we note that 
	$\wt \frZ_{\Eis} = \Spec\left(\Qlbar\left[ \Pic^{\Sigha}_X(k) \right]\right)$ decomposes into a disjoint union of components isomorphic to $\G_{m, \Qlbar}$ indexed by characters of $\Pic_X^{\Sigha}(k)$.
	The proof then proceeds as in \emph{loc.\@ cit.}
\end{proof}

\subsection{On (normalized) $L$-functions}
We introduce the (local and global) $L$-functions and their normalizations as used later.
Let $\pi = \bigotimes_{x \in |X|}  \pi_x$ be an automorphic representation of $G(\A)$. 
We denote by $\wt \pi$ and $\wt \pi_x$ their contragradient representations.
When convenient we interpret $\pi$ as an automorphic representation of $\GL_{2}(\A)$ with trivial central character.

\sss{The normalized global $L$-function}
The global $L$-function for the quadratic twist of $\pi$ by a quadratic character $\chi$ is by definition given as the Euler product $L(\pi \otimes \chi, s) \defined \prod_{x \in |X|} L(\pi_x \otimes \chi_x, s)$.
For $\chi = 1$ and $\chi = \eta$, the quadratic character associated to the extension $F'/F$, the degrees of the $L$-functions as polynomials in $q^{-s}$ are given by
\[
\deg L(\pi,s)= 4g-4+N, \qquad \text{and} \qquad \deg  L(\pi\otimes\eta,s)= 4g-4+2\rho+N.
\]
We use the normalization of the base change $L$-function given by
\begin{eqnarray}
	\label{eq:def-L-function}
	\sL_{F'/F}(\pi,s)& \defined & |\om_{X}|^{-2s}q^{(\r + N) s } \frac{L\left(\pi,s + \ha\right) L\left(\pi\otimes\eta,s+\ha \right)} {L(\pi,\Ad,1)},
\end{eqnarray}
where $L(\pi, \Ad, s)$ is the adjoint $L$-function.
Our normalization is done in such a way that $\sL_{F'/F}$ satisfies the functional equation
$
\sL_{F'/F}(\pi,s)=(-1)^{\#\Si}\sL_{F'/F}(\pi,-s)
$
as in \cite[§2.6]{Yun2019}.

\sss{On certain Rankin-Selberg $L$-functions}

We also need values of certain Rankin-Selberg $L$-factors 
\[
	L(\pi_x \times \wt \pi_x, s)
\] below. 
By \cite[Theorem 15.1]{Jacquet1972} and \cite[Corollary 1.3, Proposition 1.4]{Gelbart1978} the local Rankin-Selberg $L$-functions can be determined depending on the type of the local representation $\pi_x$. 
When $\pi_x$ is supercuspidal, then the possible local $L$-factors are either 
\[
(1 - q_x^{-s}) \qquad \text{ or } \qquad (1 - q_x^{-s}) \cdot(1 + q_x^{-s}) = (1 - q_x^{-2s}). 
\]
When $\pi_x$ is a ramified twist of the Steinberg representation, the local Rankin-Selberg $L$-factor is given by
\[
(1 - q_x^{-s}) \cdot (1 - q_x^{-1 - s}).
\]
Finally, when $\pi_x \cong \pi(\chi, \chi^{-1})$ is a ramified principal series representation, the $L$-factor is given by 
\[
(1 - q_x^{-s})^2
\]
when $\chi^2$ is ramified or 
\[
(1 - q_x^{-s})^2 \cdot (1-\chi^2(\varpi_x) q_{x}^{-s}) \cdot (1-\chi^{-2}(\varpi_x) q_{x}^{-s})
\]
when $\chi^2$ is unramified.
In either case, the value at 1 of the local Rankin-Selberg $L$-factor is of the form 
\begin{equation}
	L(\pi_x \times \tilde \pi_x, 1) = (1 - q_x^{-1}) \cdot P_{\pi_x}(q)
	\label{eq:L-value-RS}
\end{equation}
where $P_{\pi_x}$ is a rational function of $q$ depending on the type of $\pi_x$ as above.

\sss{Local $L$-factors and Whittaker functions at points with non-square-free conductor}
\label{sect:whittaker}
Let $x \in |X|$ and let us fix a nontrivial unramified additive character $\psi_x \colon F_x\to \CC^{\times}$.
Let us assume that the local factor $\pi_x$ has conductor $\pf_x^{n_x}$ where $n_x \geq 2$. 
We continue to follow the normalizations of \cite{Yun2019} and consider the measure $d^\times t_x=\zeta_x(1)\frac{dt_x}{|t_x|}$ on $F_x^\times$, where $dt_x$ is the self-dual measure with respect to $\psi_{x}$. In particular, we have $\vol_{d^\times t_x}(\cO_{x}^{\times})=1$.

Let $\cW = \cW_{\psi_x}(\pi_x)$ be the Whittaker model of $\pi_x$ with respect to the choice of $\psi_x$. 
For a section $\phi_x \in \cW$ and a character $\chi_x$ on $F_x^\times$ we consider the linear functional
\[
	\lambda(W_x, \chi_x, s) \defined \int_{F_x^\times} W_x\left(\matrixx a {}{} 1 \right) \chi_x(a) |a|^s d^\times a.
\]
By assumption on the conductor, there is a unique invariant Whittaker function $W_{x,0} \in \Wc^{\Gamma_0(\pf_x^{n_x})}$ such that $W_{x,0}(1) = 1$, which is explicitly given by 
\[
	W_{x, 0}\left(\matrixx a {}{} 1 \right) = \mathbf{1}_{\Oc_x^\times}(a)
\]
in our case, compare for example \cite[§2.4]{Schmidt2002}.
As in \cite[§3.3]{ZhangWei2014}, compare also \cite[§4.3]{Yun2017} and \cite[§2.3.2]{Yun2019}, we consider the normalized linear functional on $\cW$ defined by 
\[
	\lambda^\nat(W_x, \chi_x, s) \defined \frac{\lambda(W_x, \chi_x, s)}{L(\pi_x \otimes \chi_x, s + 1/2)}.
\]
In particular, $\lambda^\nat(W_{x,0}, \chi_x, s) = 1$ for unramified characters $\chi_x$ as in this case also $L(\pi_x \otimes \chi_x, s) = 1$.

%

%
%

\subsection{Local spherical characters}

Following \cite[§2.3.2]{Yun2019} we consider the inner product on $\Wc = \Wc_{\psi_x}(\pi_x)$ defined as  
\begin{align*}
	\theta^\nat_x(W_x,W_x') \defined \frac{1}{L(\pi_x\times\wt\pi_x,1)} \int_{F_x^\times} W_x \left(\matrixx{a}{}{}{1}\right) \ov {W'_{x}} \left( \matrixx{a}{}{}{1} \right) d^\times a,
\end{align*}
where we continue to use the (unnormalized, local) Tamagawa measure $d^\times a$.
The local spherical character for $\pi_x$ as is then defined as 
\begin{align}\label{eqn J loc}
	\BJ_{\pi_x}(f_x,s) \defined 
	\sum_{\{W_{i}\}} \frac {\lambda_x^\nat (\pi_{x}(f_{x})W_{i}, {\bf 1}_x, s)\lambda_x^\nat(\ov{W_{i}}, \eta_x, s)} {\theta_x^\nat(W_{i},W_{i})}.
\end{align}
where the sum runs over an orthogonal basis $\{W_{i}\}$ of $\CW_{\psi_x}(\pi_{x})$. 

The local spherical characters at points $x \in X \setminus (\Sig \cup R)$ for spherical Hecke functions, as well as for the local test functions at points $x \in R$ and square-free points $x \in \Sig$ is done in \cite[Propositions 2.7, 2.8 and 2.9]{Yun2019}.
It remains to do the calculation at points with deeper level.
\begin{prop}
	\thlabel{prop:J-pi-deeper-level}
	Assume that $\pi_{x}$ has conductor $\pf_x^{n}$ with $n \geq 2$. Then we have
	\begin{equation*}
		\BJ_{\pi_{x}}(1_{\Gamma_0(\pf_x^n)},s_1,s_2) \vol(G(\cO_{x}))\z_{x}(2)\cdot q_{x}^{-n} \cdot q_xP_x(q).
	\end{equation*}
\end{prop}
\begin{proof}
	We adapt the proof of \cite[Proposition 2.9]{Yun2019}.
	We extend $W_{x,0}$ to an orthogonal basis of $\Wc$ and note that $\pi(\mathbf{1}_{K_0(\pf_x^{n_x})})$ acts by multiplication by $\vol(\Gamma_0(\pf_x^n))$ on $W_{x, 0}$ and by 0 on its orthogonal complement.
	Hence, the local spherical character at $f={\mathbf{1}_{\Gamma_0(\pf_x^n)}}$ collapses to 
	\begin{equation}\label{pre J+}
		\BJ_{\pi_{x}}(1_{1_{\Gamma_0(\pf_x^n)}},s_1,s_2)=\vol({\Gamma_0(\pf_x^n)})\,
		\frac{\lambda^\nat_{x}( W_{x, 0}, {\bf1}_x, s) \lambda^\nat_{x} \left(  W_{x, 0}, \eta_{x}, s\right)} {\theta_{x}^\nat(W_{x, 0},W_{x,0})}.
	\end{equation}
	Following the discussion in Section \ref{sect:whittaker}
	we have for any unramified character $\chi': F^\times_{x} \to\BC^\times$ that
	\begin{align}\label{eq lambda W0}
		\lambda^\nat_{x}(W_0,\chi',s)=1.
	\end{align}
	As $W_{x,0} = \mathbf{1}_{\Oc_x^\times}$ it follows that the normalized inner product 
	\begin{align*}
		\theta^\nat_x(W_{x, 0}, W_{x, 0})= L(\pi_x \times \wt \pi_x, 1)^{-1}.
	\end{align*}
	As the index of $\Gamma_0(\pf_x^n)$ inside $G(\Oc_x)$ is $q_x^n(1 + \frac{1}{q_x})$ we get
	\begin{equation*}
		\vol(\Gamma_0(\pf_x^n))= q_x^{1-n}(1+q_{x})^{-1}\vol (G(\CO_{x}))
	\end{equation*}
	and hence, using \eqref{eq:L-value-RS},
	\begin{equation}\label{theta vol}
		\vol(\Gamma_0(\pf_x^n)) \theta^\nat_{x}(W_{0,x},W_{0,x})^{-1} = \vol(G(\cO_{x}))\z_{x}(2)q_{x}^{-n} \cdot q_x P_{\pi_x} (q),
	\end{equation}
	where $P_\pi$ is the rational function of $q$ as in \eqref{eq:L-value-RS}.
	The claim now follows by plugging \eqref{eq lambda W0} and \eqref{theta vol} into \eqref{pre J+}.	
%
%
\end{proof}

\subsection{Period integrals and $L$-functions}
	For a section $\phi \in \pi$ and a character $\chi \colon F^\times \setminus \A^\times \to \C^\times$ we define its $(A, \chi)$-period integral as
	\begin{equation}
		\Pc_{\chi}(\phi, s) \defined \int_{[A]} \phi(a)\chi(a)|a|^{s} da.
	\end{equation}
	\begin{definition}
		For $f \in \sC_c^\infty(G(\A))$ the global spherical character $\JJ_{\pi}(f, s)$ is defined as
		\[
			\JJ_\pi(f, s) \defined \sum_{\phi} \frac{\Pc(\pi(f)\phi,s)\Pc_{\eta}(\overline{\phi}, s)}{\langle \phi, \phi\rangle_{\mr{Pet}}},
		\]
		where the sum runs over an orthogonal basis of $\pi$.
	\end{definition}
	To compare the global spherical characters to the local ones defined above, we choose a non-trivial global character $\psi \colon F \backslash \A \to \C^\times$ which is unramified at all places $R \cup \Sig$ and use its local components to define the local spherical characters.
	As in \cite{Yun2017} and \cite[(2.9)]{Yun2019} the global spherical character has a decomposition into its local factors as 
	\[
		\JJ_{\pi}\left(f^\Sig,s\right) = |\omega_X|^{-1} \frac{L\left(\pi, s + \frac{1}{2} \right) \cdot L\left(\pi \otimes \eta, s + \frac{1}{2} \right)}{2 L(\pi, \mr{Ad}, 1)} \prod_{x \in |X|} \JJ_{\pi_x}(f_x, s),
	\]
	where the factor $|\omega_X|$ is due to the fact that we used the unnormalized Tamagawa measure instead of the normalized Haar measure to define the local factors. 
	We set 
	\[
		P(q) \defined |\om_{X}| q^{\rho/2 - N} \cdot \prod_{x \in |\Sigma| - |\Sigma_\Iw|} q_x P_{\pi_x}(q) = |\om_{X}| q^{\rho/2 - N} \cdot \prod_{x \in |\Sigma| - |\Sigma_\Iw|} \frac{L(\pi_x \times \wt \pi_x, 1)}{q_x -1},
	\]
	which is a rational function of $q$ that only depends on $\pi$ and the quadratic extension $F'/F$. 
	In particular, when $\pi$ has square-free level $P(q) = |\om_{X}| q^{\rho/2 - N}$ is the power of $q$ that appears in \cite[Proposition 2.10]{Yun2019}.
	
	\begin{thm}[{\cite[Proposition 2.10]{Yun2019}}]\thlabel{thm:global-char} 
		
		Let $\pi$ be a cuspidal automorphic representation of $G(\BA)$ satisfying the generalized Heegner hypothesis \eqref{eq:Heegner-hypothesis} with respect to the quadratic extension $F'/F$. 
		Then for $f\in \sH^{\Sig\cup R}_{G}$, we have
		\[
			q^{N s} \JJ_\pi \left(f^\Sig, s \right)=\ha\lambda_\pi(f) \cdot P(q) \cdot \sL_{F'/F}(\pi,s).
		\] 
	\end{thm}
	
	\begin{proof}
		This follows precisely as in \cite[Proposition 2.10]{Yun2019} using the result of \thref{prop:J-pi-deeper-level} at places with deeper level.
	\end{proof}

 \subsection{Jacquet's relative trace formula}
 
 We introduce the setup and statements of the version of Jacquet's relative trace formula as used in \cite{Yun2017, Yun2019}.
 As essentially all the arguments either apply directly in our setting or go through with only little modifications, we do not give full arguments in this section but refer to \emph{loc.\@ cit.\@} for the details.
  
 \sss{The kernel function}
 
 For $f \in \sC_c^\infty(G(\A))$ 
 we define the \emph{kernel function} on $G(\A) \times G(\A)$ as 
 $$ \KK_f(g,g') \defined \sum_{\gamma \in J} f(g^{-1} \gamma g').$$
 For a function $\chi$ on $\A^\times$ we define a function of the same name on $A(\A) \se G(\A)$ by $\begin{psmallmatrix}
 	a &  \\ & d
 \end{psmallmatrix} \mapsto \chi(a d^{-1})$.
 
 Using $\KK_f$ as an integral kernel we define a functional on $\sH$ by 
 \begin{equation}
 	\JJ(f, s) \defined 
 	\int_{[A] \times [A]}^{\reg} \KK_f(h,h')|h h'|^{s} \eta(h') dh \, dh'
 	\label{eq:definition-J}
 \end{equation}
 as in \cite[§2]{Yun2017}, which is also $\JJ(f, s, 0)$ in the two variable version defined in \cite[§2.1]{Yun2019}.
 The integral is regularized as explained in \cite[§2.2, Proposition 2.1]{Yun2017} by taking the limit of integrals  over certain increasing bounded subsets of $[A] \times [A]$. 
 
 \sss{Orbital integrals}
 On the geometric side of the relative trace formula we decompose $\JJ$ into orbital integrals along $A \times A$-orbits in $G$. 
 Recall the invariant on $G(F)$ from \cite[§2.1]{Yun2017} defined by 
 \begin{align*}
 	\mr{inv}\left( 
 	\begin{pmatrix}
 		a & b \\ c & d
 	\end{pmatrix}  \right) \defined  \frac{bc}{ad} \in \PP^1(F) \setminus \{1\},
 \end{align*}
 which is constant on $A \times A$-orbits. 
 The regular semisimple orbits are precisely those with invariant contained in $\P^1_F \setminus \{0,1,\infty\}$ and $\mr{inv}$ descends to a bijection on the set of regular semisimple orbits with $\P^1_F \setminus \{0,1,\infty\}$.
 Moreover, the two invariants $0$ and $\infty$ both correspond to three orbits, representatives for the three orbits with invariant 0 are given by 
 $$ 1 = \begin{psmallmatrix}
 	1 & 0 \\ 0 & 1
 \end{psmallmatrix}, \quad n_+ = \begin{psmallmatrix}
 	1 & 1 \\ 0 & 1 
 \end{psmallmatrix}, \quad \text{ and } \quad n_- = \begin{psmallmatrix}
 	1 & 0 \\ 1 & 1 
 \end{psmallmatrix},$$
 while representatives for the orbits with invariant $\infty$ are 
 $$ w = \begin{psmallmatrix}
 	0 & 1 \\ 1 & 0
 \end{psmallmatrix},  \quad wn_+ = \begin{psmallmatrix}
 	0 & 1 \\ 1 & 1 
 \end{psmallmatrix}, \quad \text{ and } \quad wn_- = \begin{psmallmatrix}
 	1 & 1 \\ 1 & 0 
 \end{psmallmatrix}.$$

 For $\g \in G(F)$ we define the orbital integral as 
 \begin{equation}
 	\JJ_\gamma(f, s) \defined  \int_{A(\A) \times A(\A)}^{\reg} f(h^{-1} \gamma h') | hh'|^{s} \,  \eta(h') dh \, dh'.
 	\label{eq:definition-J-stabilizer}
 \end{equation} 
 The integral is absolutely convergent in the regular semisimple case, and has to be regularized as the integral before otherwise.
 It is convenient to group orbital integrals according to the invariant $u \in \PP^1(F) \setminus \{1\}$ by
 $$ \JJ_u(f, s) \defined  \sum_{\inv \gamma = u} \JJ_\g(f, s). $$
 As in \cite[§2.6]{Yun2017} we obtain the orbital integral expansion of $\JJ$ by unfolding the integrals:
 \begin{equation}
 	\JJ(f, s) = \sum_{\g \in J(F)} \JJ_\g(f, s) = \sum_{u \in \PP^1(F) \setminus \{1\}} \JJ_u(f, s).
 	\label{eq:orbital-integral-decomposition}
 \end{equation}

 \sss{Spectral decomposition and period integrals}
 
 Recall from \cite[(4.8)]{Yun2017} that the kernel function $\KK_f$ decomposes as 
 $$ \KK_{f}(h, h') = \KK_{f, \cusp}(h,h') + \KK_{f, \mathrm{sp}}(h,h') + \KK_{f,\Eis}(h,h')$$
 into a cuspidal, residual and Eisenstein part.
 For $* \in \{ \cusp, \mathrm{sp}, \Eis\}$ define $\JJ_*(f, s)$ as in \eqref{eq:definition-J} using $\KK_*$ instead of $\KK$ as the kernel function.
  We have the following spectral decomposition of $\JJ$.
 
 \begin{proposition}[{\cite[Theorem 4.3, Lemma 4.4]{Yun2017}, \cite[Theorem 2.2]{Yun2019}}]
 	\thlabel{prop:spectral-dec-J}
 	Let $\Sig = \sum_{x} n_x [x]$ and let $f = f^{|\Sig|} \otimes f_{|\Sig|} \in \sH$ with $f^{|\Sig|} \in \wt \cI_{\Eis}^\Sig$ and $f_{|\Sig|} \in \sC_c^\infty(G(\A_{|\Sig|}))$ such that $f_{|\Sig|}$ is invariant under $\prod_{x \in S} \Gamma_0(\pf_x^{n_x})$. Then
 	\begin{equation}
 		\JJ(f, s) = \sum_{\pi} \JJ_\pi(f, s)
 	\end{equation}
 	where the sum runs over all cuspidal automorphic representations $\pi$ of $G(\A)$ and only finitely many summands are non-zero.
 \end{proposition}
 \begin{proof}
 	We have to show that the special and Eisenstein parts both vanish.
 	For the special part this is \cite[Lemma 4.4]{Yun2017}.
 	For the Eisenstein part the modification in the proof of \cite[Theorem 2.2]{Yun2019} of the argument of \cite[Theorem 4.3]{Yun2017} generalizes to our setting. 
 	We use their notation and in particular define the induced representations $(\rho_{\chi,u}, V_{\chi, u})$, $V_\chi$ and the Eisenstein series $E(-)$ as in \emph{loc.\@ cit.}
 	We fix a compact open $K = K^{|\Sig|} K_{|\Sig|}$ such that $f$ is bi-$K$-invariant.
 	We can write 
 	\[
 		\KK_{f, \Eis}(h,h') = \sum_{\chi} \KK_{f, \Eis, \chi} (h,h')
 	\]
 	where the sum runs over all $\chi \colon \F^\times \backslash \A^1 \to \C$ and 
 	\[
 		\KK_{f, \Eis, \chi}(h,h') = \frac{\log q}{2 \pi i} \sum_{i,j} \int_{0}^{\frac{2 \pi i}{\log q}} (\rho_{\chi, u}(f) \phi_i, \phi_j) E(h, \phi_i, u, \chi) \overline{E(h', \phi_j, u, \chi)} du,
 	\]
 	where the sum runs over an orthonormal basis of $V^K_{\chi}$.
 	As in the proof of \cite[Theorem 2.2]{Yun2019} for the inner product not to vanish it is necessary that $V_\chi^{K_0(\Sig)} \neq 0$. This implies that $\chi$ has conductor at most $\Sigha$ and hence factors through $\Pic_X^{\Sigha}(k)$. 
 	However, using that $f \in \wt \cI^\Sig_{\Eis}$ this implies that the integral vanishes as in \emph{loc.\@ cit.
 	} 
 \end{proof}

 \sss{The distribution $\JJ^r$}
 Let $r$ be a non-negative integer.
 For $f \in \sH^{|\Sig| \cup R}$ we set 
 $$ \JJ^{r}(f) \defined \left(\frac{\partial}{\partial s}\right)^{r}   \left( q^{N s} \JJ(f^{\Sig}, s)\right)|_{s =0}. $$
 In a similar fashion we denote by $\JJ^{r}_{\pi} (f)$ the derivatives of the global spherical characters and by $\JJ^{r}_{u}(f)$ the corresponding derivative of the orbital integral.

 In particular, as a consequence of \thref{thm:global-char} for $f \in \sH^{|\Sig| \cup R}$ we get 
 \begin{equation}
 	\JJ^r_\pi\left( f^\Sig \right) = \ha\lambda_\pi(f) \cdot P(q) \cdot \sL^{(r)}_{F'/F}\left(\pi,\ha \right).
 	\label{eq:spectral-dec-J}
 \end{equation}
 
 \subsection{Geometrization of orbital integrals}
 We state the geometrization of orbital integrals in terms of traces of Frobenius on a certain local system on a Hitchin-type fibration. 
 We closely follow the treatment in \cite[§6]{Yun2019}.
 As at no point in \emph{loc.~cit.} the reducedness of $\Sig$ is used in an essential way, all arguments carry over essentially verbatim.
 Instead of repeating the arguments we only give references to the corresponding statements in \emph{loc.~cit.}
 
 We introduce some more notation. 
 Recall from \cite[Appendix A]{Yun2019} the Picard stack $\Pic_X^{\sqR, d}$ with square roots along the ramification divisor parametrising triples 
 $\Delta^\nat = (\Delta, \Theta_R, \iota)$ consisting of a degree $d$ line bundle $\Delta$ on $X \times S$ together with a line bundle $\Theta_R$ on $R \times S$ and an isomorphism $\iota \colon \Theta_R^{\otimes 2} \cong \Delta|_{R \times S}$. 
 
 Moreover, let $U \defined X - (\Sig \cup R)$ and denote by $X_d$ the $d$-th symmetric power of $X$ parametrising degree $d$ effective divisors. 
 Similarly define $U_d$. 
 
 \subsubsection{The Hitchin base}
 We introduce the relevant Hitchin fibration $g_d \colon \cN_{\un d} \to \cA_d$
 
 \begin{definition}
 	\begin{enumerate}
 		\item 
 		Let $\cA_d$ be the stack over $\Fq$ whose $S$-points are given by the groupoid of tuples 
 		$$ (\Delta^\nat,  a,b, \vartheta_{R})$$
 		consisting of a point $\Delta^\nat \in \Pic_X^{\sqR, d + \rho }(S)$, two sections $a,b \in H^0(X \times S, \Delta)$, and a section $\vartheta_{R} \in H^0(R \times S, \Theta_R)$
 		such that 
 		\begin{enumerate}
 			\item $a|_{\Sigma \times S}$ is nowhere vanishing while $b|_{\Sigma \times S} = 0$,
 			\label{it:defn-Ad-Sig}
 			\item $a|_{R \times S} = b|_{R \times S} = \iota(\vartheta_{R}^{\otimes 2})$ as sections of $\Delta|_{R \times S}$, and the section $a - b$ of $\Delta$ vanishes precisely to the first order along $R \times S$, and
 			\label{it:defn-Ad-R}
 			\item $(a - b)|_{X \times s}$ is non-zero for all geometric points $s$ of $S$.\
 			\label{it:trace-non-zero}
 		\end{enumerate}
 		\thlabel{def:Ad}
 		\item 
 		Let $\Ac^\Diamond_d \se \Ac_d$ be the open substack where neither $a, b \neq 0$ for every geometric point on $S$.
 		\item 
 		Let $\Ac^\flat$ be the stack over $\Fq$ whose $S$-points are given by the groupoid of pairs $(\Delta, \xi)$ as in the definition of $\cA_d$ above, in particular they are supposed to satisfy conditions \eqref{it:defn-Ad-Sig}, \eqref{it:defn-Ad-R} and \eqref{it:trace-non-zero}.
 	\end{enumerate}
 \end{definition}
 As in \cite{Yun2019} the last condition \eqref{it:trace-non-zero} is non-void only when both $\Sig$ and $R$ are empty. 
 For later computations the following version of $\cA_d^\flat$ without the extra data of square roots along the ramification locus $R$ will be convenient. Let $\Omega \colon \cA_d \to \cA_d^\flat$ denote the forgetful map.
 Moreover, we have a map $\tr \colon \cA_d^\flat \to U_d$ associating to $(\Delta, a, b) \in \Ac^\flat$ the Cartier divisor given by $(\Delta(-R), a-b)$. Note that  conditions  \eqref{it:defn-Ad-R} and \eqref{it:defn-Ad-Sig} guarantee that the divisor is supported away from $\Sig \cup R$.

 An effective divisor $D$ of degree $d$ (assumed to be away from $\Sigma$ and $R$) of $X$ defines the point $(\Oc(D), 1) \in U_d(k)$.
 Let $\cA_{D}$ and $\cA_D^\flat$ denote the base changes of $\cA_d$ and $\cA_d^\flat$ along this map.
 In other words, $\cA_{D}^\flat$ parametrizes pairs $(\Delta, \xi)$ together with an isomorphism $\Delta \cong \Oc_X(D + R)$ that identifies $a - b =1$, compare \cite[§3.3.2]{Yun2017}.
 On $k$-points we then have an injective map 
 $$\inv \colon \cA_D^\flat(k) \to \Inv$$
 given by 
 $(\Delta, a,b) \mapsto \frac{b}{a} \in \P^1(F)$.
 
 We recall some of the geometric properties of $\cA_d$, $\cA_d^\Diamond$ and $\cA_d^\flat$.
 \begin{lemma}[{\cite[§3.2.3]{Yun2017}, \cite[§5.1]{Yun2019}}]
 	\thlabel{lem:Ad-properties}
 	\begin{enumerate}
 		\item $\cA_d^\flat$ is representable by a scheme and the map $\Omega$ is proper.
 		\item Then $\cA_{d}$ is representable by a geometrically connected Deligne-Mumford stack over $\Fq$.
 	\end{enumerate}
 \end{lemma}
 
\subsubsection{A Hitchin fibration}
 We generalize the stack $\cN_{\underline{d}}$ of \cite[§6.1]{Yun2019}. For a non-negative integer $d$ let $Q_d$ be the set of quadruples $\un d = (d_{11}, d_{12}, d_{21}, d_{22})$ of non-negative integers such that $d_{11} + d_{22} = d+ \r = d_{12} + d_{21}$.
 \begin{definition}
 	\thlabel{def:Nd}
 	The stack $\wt \cN_{\un d}$ over $\Fq$ is defined to parametrize tuples 
 	$$(\Lc_{01}^\nat, \Lc_{02}^\nat, \Lc_{31}^\nat, \Lc_{32}^\nat, \ph = (\ph_{ij})_{1 \leq i,j \leq 2}, \psi_{R} = (\psi_{ij, R})_{1 \leq i,j \leq 2})$$ 
 	consisting of
 	\begin{itemize}
 		\item four line bundles $\Lc_{01}^\nat, \Lc_{02}^\nat, \Lc_{31}^\nat, \Lc_{32}^\nat, \in \Pic_{X}^{\sqR}(S)$ with $\Lc^{\nat}_{ij} = (\Lc_{ij}, \Kc_{ij, R}, \iota_{ij})$, 
 		\item maps $\ph_{ij} \colon \Lc_{3j} \to \Lc_{0i}$, and		
 		\item maps $\psi_{ij, R} \colon \Kc_{3j, R} \to \Kc_{0i, R}$,
 	\end{itemize} 
 	such that the following conditions are satisfied:
 	\begin{enumerate}
 		\item 
 		$\deg(\Lc_{0i}) - \deg(\Lc_{3j}) = d_{ij}$ fiberwise over every geometric point of $S$.
 		\label{it:def-Nd-degree}
 		\item 
 		The diagram
 		\begin{equation}
 			\begin{tikzcd}
 				\Kc_{3j, R}^{\otimes 2} \arrow[r, "\psi_{ij, R}^{\otimes 2}"] \arrow[d, "\iota_{3j}"] & \Kc_{0i, R}^{\otimes 2} \arrow[d, "\iota_{0i}"] \\
 				\Lc_{3j}|_{R} \arrow[r, "\ph_{ij}|_{R}"] & \Lc_{0i}|_{R}
 			\end{tikzcd} 
 		\end{equation}
 		is commutative for every $1 \leq i,j \leq 2$.
 		\label{it:def-Nd-root-ramification}
 		\item \label{it:def-Nd-R-vanish}
 		$\det(\psi) = 0$ and $\det(\ph)$ vanishes precisely to the first order along $R\times S$. 
 		\item 
 		$\ph_{21}|_{\Sig \times S} = 0$ and both $\ph_{11}|_{\Sig \times S}$ and $\ph_{22}|_{\Sig \times S}$ are fiberwise on $S$ nowhere vanishing
 		\label{def:Nd-vanish-sigma}
 		\item Moreover, $\det (\ph)$ is fiberwise on $S$ nonzero (which is only non-void when $\Sig = R = R_3 = \emptyset$).
 		\label{it:def-Nd-det-nonzero}
 		\item If $d_{11} < d_{22}$, then $\ph_{11}$ is fiberwise nonzero; if $d_{11} \geq d_{22}$, then $\ph_{22}$ is fiberwise nonzero (when $\Sig \neq \emptyset$ this condition is void). If $d_{12} < d_{21} - N$, then $\ph_{12}$ is fiberwise nonzero; and if $d_{12} \geq d_{21}-N$, then $\ph_{21}$ is fiberwise nonzero.
 	\end{enumerate}	
 	The ramified Picard stack $\Pic^{\sqR}_X$ acts on $\wt \cN_{\un d}$ by simultaneously twisting both $\Lc_{0i}^\nat$ and $\Lc_{3j}^\nat$. Then define $\cN_{\un d} \defined \wt \cN_{\un d}/\Pic^{\sqR}_{X}$.
	Moreover, let $\widetilde{\cN}^\Diamond_{\un d} \se \wt{\cN}_{\un d}$ be the substack such that all $\ph_{ij}$ are fiberwise on $S$ non-zero and set $\cN^\Diamond_{\un d} \defined \wt{\cN}^\Diamond_{\un d}/ \Pic_X^{\sqR}$.
\end{definition}
 For $(\Lc_{01}^\nat, \Lc_{02}^\nat, \Lc_{31}^\nat, \Lc_{32}^\nat, \ph, \psi_{R}) \in \cN_{\un d}(S)$ set
 \(
 	\Delta^\nat \defined \underline{\Hom}(\Lc_{31}^\nat \otimes \Lc_{32}^\nat,  \Lc_{01}^\nat \otimes \Lc_{02}^\nat).
 \)
 Its degree is given by
 \[
 	\deg(\Delta) = \deg(\Lc_{01}) + \deg(\Lc_{02}) - \deg(\Lc_{31}) - \deg(\Lc_{32})  =  d + \r,
 \]
 i.e. $\Delta^\nat  \in \Pic_X^{\sqR, d+\r}(S)$.
 Moreover, we define the sections $a = \ph_{11} \ph_{22}$ and $b = \ph_{21} \ph_{12}$ of $\Delta$ as well as $\theta_R = \psi_{11} \psi_{22} = \psi_{12} \psi_{21}$. 
 One checks that the conditions in the definition of $\cN_{\un d}$ guarantee that 
 $g_{\un d}( \Lc_{0i}^\nat, \Lc_{3j}^\nat, \ph, \psi_R) \defined (\Delta^\nat, a, b, \vartheta_{R}) \in  \cA_d(S)$. 
 	Moreover, denote by $g_{\un d}^\Diamond$ the restriction of $g_{\un d}^\Diamond$ to $\cN^\Diamond_{\un d}$. 
 	By construction, the restriction $g_{\un d}^\Diamond$ factors through $\cA_d^\Diamond$.
 	
 	We recall properties of $\cN_{\un d}$ generalizing \cite[Proposition 3.1]{Yun2017} and \cite[Proposition 6.2]{Yun2019}.
 	\begin{proposition}
 		\thlabel{prop:geometry-Nd}
 		Let $\un d \in Q_d$.
 		\begin{enumerate}
 			\item Assume that $d \geq 4g - 3 + \rho + N$. Then $\cN_{\underline{d}}$ (and hence its open substack $\cN_{\un d}^\Diamond$) is a smooth and geoemtrically connected Deligne-Mumford stack of dimension $2d + \rho + \r_3/2 - g - N+1$ over $k$.
 			\item 
 			The morphism $g_{\un d}$ is proper and its restrictions $g_{\underline{d}}^\Diamond$ is finite.
 		\end{enumerate}
 	\end{proposition}

 	\label{ss:geom-orb-int}
 	\subsubsection{A rank one local system on $\cN_d$ and traces of Frobenius}
 	Recall that \cite[Appendix A, §6.2.1]{Yun2019} construct a local system on $\Pic_X^{\sqR}$ whose trace function is given by the quadratic character associated to the extension $F'/F$ and use it to construct an associated local system $L_{\un d}$ on $\cN_{\un d}$. 
 	The construction generalizes to our setting.

 	We can now state the main result of this section that relates the orbital integrals with traces of Frobenius on the local system $L_{\un d}$ on the Hitchin fibration.
 	Recall that for an effective divisor $D$ of degree $d$ as above we have the Hecke function $h_D^{\Sigma} \in \sH^\Sig $ as introduced in Section \ref{sec:test-function} above. 
 	\begin{theorem}[{\cite[Theorem 6.3]{Yun2019}}]
 		Let $D \se X \setminus (R \cup \Sigma)$ an effective divisor of degree $d$ and $u \in \P^1(F) - \{1\}$.  
 		If $u$ lies in the image of $\inv_D \colon \cA^\flat_D(k) \to \P^1(F)$ denote by $a \in \cA_D^\flat(k)$ its (unique) preimage and fix a geometric point $\bar a$ over $a$.
 		\begin{enumerate}
 			\item
 			\label{it:J-tr-zero} 
 			If $u \not \in \cA^\flat_D(k)$, then $\JJ(\gamma, h_D^{\Sigma}, s_1, s_2) = 0$ for every $\gamma$ with $\inv(\gamma) = u$. 
 			\item 
 			\label{it:J-tr-rss}
 			If $u$ is regular semisimple, i.e. $u \not \in \{0,1,\infty\}$, then 
 			\begin{equation}
 				\label{eq:orbital-integrals-trace} 
 				\JJ( u , h_D^{\Sig}, s_1, s_2) = \sum_{\un d \in Q_d}
 				q^{  (2d_{12} - (d+\r))s_1 + (2d_{11} - (d+\r))s_2}  \cdot
 				\Tr (\Fr_a  ,    ( Rg^{\flat}_{\un d, !} L_{ \un d} )_{\bar a} ).
 			\end{equation}
 			\item 
 			If $u \in \{0,\infty\}$ and $d \geq 4g - 3 + \rho + N$, then \eqref{eq:orbital-integrals-trace} holds.	
 		\end{enumerate}
 	\end{theorem}
 	The proof of \emph{loc.~cit.~}carries over verbatim.

\section{Moduli spaces of shtukas with naive $\GammaN$-level structures and special cycles}
\label{sect:shtukas}

We introduce analogues of the moduli spaces of $\PGL_2$-shtukas with supersingular legs at $\infty$ of \cite{Yun2019} for deeper than Iwahori level structures and construct Heegner-Drinfeld cycles on them.
Such integral models for deeper level structures were studied more generally in \cite{Bieker2023, Bieker2022a}.
It turns out that the Heegner-Drinfeld cycles are already supported on an open 
substack that can be described more naively.
We study the geometry of this moduli space of shtukas with naive $\GammaN$-level structure.

\subsection{Bundles with naive $\GammaN$-level structure}
As before, let $\Sigma \se X$ be an effective divisor. 
We introduce (naive) $\GammaN$-level structures on rank 2 vector bundles. 
This moduli problem, for not necessarily reduced $\Sig$, is also considered in \cite{Logan2025} over (copies of) $X - \Sig$.
The novelty here is that we study its geometry in particular above points in $\Sig$.
When $\Sigma$ is reduced we recover the corresponding definitions in \cite{Yun2019}.
We follow the notation and presentation of \emph{loc.~cit.}.

\begin{definition}
	\thlabel{def:naive-level}
	Let $S$ be an $\Fq$-scheme.
	A \emph{naive $\GammaN$-level structure} on a rank $2$ vector bundle $\Ec$ on $X \times S$ is a quotient of $\Oc_{\Sigma \times S}$-modules 
	$$ \Ec|_{\Sigma \times S} \twoheadrightarrow \Lc $$
	where $\Lc$ is a line bundle on $\Sigma \times S$. 
	The moduli stack of rank $2$ vector bundles with naive $\GammaN$-level structure is denoted by $\BunGLnaive$.
\end{definition}

\begin{remark}
	\thlabel{rem:def-Bunnaive}
	\begin{enumerate}
		\item 
		Equivalently, a naive $\GammaN$-level structure on $\Ec$ is given by a rank 2 vector bundle 
		$$\Ec\left(-\frac{\Sigma}{2}\right) \defined \ker\left(\Ec \twoheadrightarrow \Lc \right) \se \Ec$$
		such that the quotient $\Ec/\Ec(-\Sig/2)$ is supported on $\Sig \times S$ and finite locally free of rank 1 as $\Oc_{\Sig \times S}$-module.
		Here $\Ec\left(-\frac{\Sigma}{2}\right)$  is just notation, in particular, when $\Sigma = 2[x_0]$ for some $x_0 \in |X|$ the vector bundle $\Ec\left(-\frac{2x_0}{2}\right)$ will not be the twist $\Ec(-x_0)$ of $\Ec$.
		\item 
		By Bruhat-Tits theory there is a smooth affine group scheme $\Gc \to X$ with connected fibers that is generically given by $\GL_2 \times (X \setminus \Sig)$ and at points $x \in \Sig$ we have that
		\[
			\Gc(\Oc_x) = \Gamma_0(\varpi_x^{n_x}) = \ker(\GL_2(\Oc_x) \twoheadrightarrow \GL_2(\Oc_x/\varpi_x^n)/ B(\Oc_x/\varpi_x^n))		
		\] 
		is the subgroup of $\GL_2(\Oc_x)$ of matrices that become upper triangular after reduction mod $\varpi_x^n$.
		For example by \cite[Theorem 4.8]{Mayeux2020} a rank 2 vector bundle with naive $\GammaN$-level structure is precisely a $\Gc$-torsor, and $\BunGLnaive = \Bun_{\Gc}$.
		We do not need this perspective other than to appeal to the general representability statements for $\Bun_\Gc$ as well as for Hecke stacks, Beilinson-Drinfeld affine Grassmannians and moduli spaces of shtukas for smooth affine group schemes.
	\end{enumerate}
\end{remark}

The Picard stack $\Pic_X$ acts on $\BunGLnaive$ by simulatenously twisting $\Ec$ and $\Ec\left(-\frac{\Sigma}{2}\right)$. 
Let $$\BunGnaive \defined \BunGLnaive/\Pic_X.$$
We need the following well-known properties of $\BunGLnaive$ and its (forgetful) level maps.
The precise statement does not seem to appear in the literature yet so we give a proof.
\begin{lemma}
	\thlabel{lem:BunG-level-smooth}
	Let $\Sig' \se \Sig$ be a second effective divisor on $X$ and let $N'$ be the degree of $\Sig'$. 
	The forgetful map $\Bun_2(\Sig) \to \Bun_2(\Sig')$ representable, quasi-projective and smooth of relative dimension $N - N'$.
	In particular, $\Bun_2(\Sig)$ is a smooth algebraic stack over $\Fq$ of dimension $4(g - 1) + N$.
	
	The analogous statements hold also for $\BunGnaive$, which consequently is a smooth algebraic stack of dimension $3(g-1) + N$ over $k$. 
\end{lemma}
The smoothness of the stacks $\BunGLnaive$ follows from \cite[Proposition 1]{Heinloth2010}. 
Most of the statements (except for the smoothness assertion of the level map) are contained in \cite[Proposition 3.18]{Breutmann2019} for general generic isomorphisms of smooth affine group schemes. 
We give a more explicit argument in our setting.
\begin{proof}
	It suffices to prove the case where $\Sig = n[x]$ is supported on a single closed point $x \in X$, the general case follows inductively. Then $\Sig' = n'[x]$ for some $n' \leq n$. 
	Let $\Ec^\dagger =(\Ec \twoheadrightarrow \Lc') \in \Bun_G(\Sig')(R)$ for some $\Fq$-algebra $R$. 
	Then the base change $\Spec(R) \times_{\Bun_2(n'[x])} \Bun_2(n[x])$ is visibly smooth of relative dimension $(n-n') \deg(x)$ over $R$.
\end{proof}

We use the $G$-bundles with naive $\GammaN$-level structure to define a Hecke stack for $\GammaN$-level.
We fix a non-negative integer $r \geq 0$ and for every $1 \leq i \leq r$  a dominant cocharacter $\mu_i$ of $\GL_2$ given by either
$
	\mu_+ = (1, 0) $ or $ \mu_- = (0, -1).
$
We set $\underline{\mu} = (\mu_i)_{1 \leq i \leq r}$. 

\begin{definition}
	\thlabel{def:naive-level-moduli-spaces}
	\begin{enumerate}
		\item 
	The Hecke stack $\HeckenGa$ is the stack over $\Fq$ whose $S$-points parametrize the data
	\[ 
		\underline{\Ec}^\dagger = \left( \Ec_0^\dagger \overset{f_1}{\underset{x_1}{\dashrightarrow}} \Ec_{1}^\dagger \overset{f_{2}}{\underset{x_2}{\dashrightarrow}} \ldots \overset{f_{r}}{\underset{x_r}{\dashrightarrow}} \Ec_r^\dagger \right),
	\]
	where
	\begin{itemize}
		\item the $\Ec_i^\dagger = (\Ec_i \twoheadrightarrow \Lc_i)$ are rank $2$ vector bundles on $X \times S$ with a $\GammaN$-level structure for $0 \leq i \leq r$,
		\item the $x_i \in X(S)$ are points on $X$ for $1 \leq i \leq n$, and
		\item the $f_i \colon \Ec^\dagger_{i-1}|_{X \times S \setminus \Gamma_{x_{i}}} \xrightarrow{\cong} \Ec^\dagger_{i}|_{X \times S \setminus \Gamma_{x_i}}$ are isomorphisms of vector bundles with $\GammaN$-level structure outside the graph $\Gamma_{x_i} \se X \times S$ of $x_i$ such that 
		\begin{itemize}
			\item when $\mu_i = \mu_+$ the map $f_i$ extends to an injective map $\Ec_{i-1} \hookrightarrow \Ec_{i}$ that is compatible with the level structure such that $\coker(f_i)$ is supported on $\Gamma_{x_i}$ and its restriction to $\Gamma_{x_i}$ is an invertible sheaf, and 
			\item when $\mu_i = \mu_-$ the inverse map $f_i^{-1}$ extends to an injection $\Ec_{i} \hookrightarrow \Ec_{i-1}$ compatible with the level structures such that $\coker(f_i^{-1})$ is supported on $\Gamma_{x_i}$ and its restriction to $\Gamma_{x_i}$ is an invertible sheaf.
		\end{itemize}  
	\end{itemize}

	\item 
	The Beilinson-Drinfeld affine Grassmannian $\BDGrdNnaive$ parametrizes the same data as $\HeckenGa$ together with a trivialisation  
	$\alpha$ of $\Ec_r^\dagger$, i.e. a trivialisation $\alpha \colon \Ec_r \xrightarrow{\cong} \Oc_{X \times S}^2$ that identifies $\Ec_r \twoheadrightarrow \Lc_r$ with the projection to the first component $\Oc_{X \times S}^2 \twoheadrightarrow \Oc_{\Sig \times S}^2 \twoheadrightarrow \Oc_{\Sig \times S}
	$. 
	\end{enumerate}
\end{definition}
Both $\HeckenGa$ and $\BDGrdNnaive$ have natural forgetful maps $\pi^{\un \mu}_{\mr{Hk}} = (\pi^{\un \mu}_{\mr{Hk}, i})_{1 \leq i \leq r}$ (respectively $\pi^{\un \mu}_{\mr{Gr}} = (\pi^{\un \mu}_{\mr{Gr}, i})_{i}$) to $X^r$ by projection to the $(x_i)_i$, called the legs.

\begin{remark}
	\thlabel{rem:BL}
	Using Beauville-Laszlo descent, compare \cite{Beauville1995},  we see that in the case $r=1$ and $\mu_1 = \mu^+$ the $R$-points in the fiber over $x \in \Sigma_{\red}$ for some $k(x)$-algebra $R$ parametrize pairs of lattices $(\Lambda' \subset \Lambda)$ with
	\begin{center}
		\begin{tikzcd}
			\varpi_x^{n_x} R\dsq{\varpi_x} \oplus R \dsq{\varpi_x}  \arrow[right hook-latex]{r} & R\dsq{\varpi_x}^{\oplus 2} \\
			\Lambda' \arrow[right hook-latex]{r} \arrow[right hook-latex]{u} & \Lambda \arrow[right hook-latex]{u}
		\end{tikzcd}
	\end{center}
	such that $\left(\varpi_x^{n_x} R\dsq{\varpi_x} \oplus R \dsq{\varpi_x} \right) / \Lambda'$ and $R\dsq{\varpi_x}^{\oplus 2} /\Lambda$ are locally free of rank 1 as $R$ modules and $\Lambda/\Lambda'$ is locally free of rank 1 as $R\dsq{\varpi_x}/\varpi_x^n$-module.
\end{remark}

The following is the central computation to extend the results on the geometry of Hecke stacks and moduli stacks of shtukas from Iwahori to deeper level.
\begin{proposition}
	\thlabel{prop:schubert-deep-smooth}
	\begin{enumerate}
		\item 
		The global Schubert variety $\Gr_{2}^{\mu^+}(\Sig)$ is representable by a separated scheme that is of finite type and flat of relative dimension 1 over $X$.
		\item 
		Let $x \in \Sig$ such that $n_x \geq 2$. Then its base change $\Gr_{2}^{\mu^+}(\Sig) \times_X \Spec(\Oc_{X,x})$ is smooth over $\Spec(\Oc_{X,x})$.
		Moreover, its special fiber is isomorphic to a disjoint union $\A^1_{k(x)} \amalg \A^1_{k(x)}.$
		
	\end{enumerate}
\end{proposition}
\begin{proof}
	\begin{enumerate}
		\item 
		The representability statement is a standard fact.
		\item 
	As the global Schubert variety is flat and its generic fiber is smooth the second statement implies the first.
	We compute the special fiber. To simplify notation we set $\kappa = k(x)$ for the residue field and $\Oc_R = R\dsq{\varpi}$ for a $\kappa$-algebra $R$.  
	
	Let $(\Lambda' \subset \Lambda) \in \Gr_{2}^{\mu^+}(\Sig)(R)$ a point in the special fiber, using the description of the fiber in \thref{rem:BL}.
	We claim that there is a decomposition into clopen subschemes $\Spec(R) = U_1 \amalg U_2$ where $U_1$ is the locus where the induced map 
	\[
		u_1 \colon \Lambda/\Lambda' \to  \Oc_R^{\oplus 2} / \left(  \varpi^n \Oc_R \oplus \Oc_R \right)  \cong     \Oc_R/\varpi^{n}
	\] 
	is an isomorphism, and where $U_2 \se \Spec(R)$ is the locus where 
	\[
		u_2 \colon \varpi^{-n} \Lambda'/\Lambda \to  \left(   \Oc_R \oplus \varpi^{-n}\Oc_R \right) /\Oc_R^{\oplus 2} \cong     \Oc_R/\varpi^{n}
	\] 
	is an isomorphism. 
	Clearly, both $U_1$ and $U_2$ are representable by an open subscheme of $\Spec(R)$ and the decomposition is functorial in $R$.
	
	It remains to check that the two are disjoint and cover $\Spec(R)$, which can be done on points:  
	When $R$ is a field, $u_1$ is an isomorphism if and only if $\Lambda$ is different from $\varpi \Oc_R \oplus \Oc_R$ as a sublattice of $\Oc_R^{2}$. Similarly, $u_2$ is an isomorphism if and only if $\Lambda' \neq \varpi^n \Oc_R \oplus \varpi \Oc_R$. 
	This shows that for $n \geq 2$ the opens $U_1$ and $U_2$ cover $\Spec(R)$ as 
	\[
		(\varpi^n \Oc_R \oplus \varpi \Oc_R) \se (\varpi \Oc_R \oplus \Oc_R)
	\]
	visibly does not define a point in $\Gr_{2}^{\mu^+}(\Sig)$. On the other hand, let now $(\Lambda' \se \Lambda) \in U_1$. Then by construction $\Lambda$ admits a basis of the form $(1, r)$, $(0, \varpi)$ for some $r \in R$. Hence, $\Lambda'$ admits a basis $\varpi^n(1,r)$, $(0,\varpi)$. But then $\Lambda' = (\varpi^n \Oc_R \oplus \varpi \Oc_R)$. 
	
	We completed the proof that the special fiber decomposes as $\Gr_{2}^{\mu^+}(\Sig) = \Uc_1 \amalg \Uc_2$. It remains to show that the $\Uc_i$ are isomorphic to the affine line for $i \in \{1,2\}$. We give the argument for $\Uc_1$, the proof for $\Uc_2$ is analogous. 
	By construction we get a map $\Uc_1 \to \P^1_\kappa \cong \P((\Oc_\kappa/\varpi)^2)$ given by $\Lambda \mapsto \Oc_R^2/\Lambda$, which factors through $\A_\kappa^1 = \P^1_\kappa \setminus \{e_1\} \se \P^1_\kappa$, where $e_1$ is given by the projection $\kappa^2 \to \kappa$ to the first component.
	We get a map in the opposite direction by $(R^2 \twoheadrightarrow M) \mapsto (\Lambda' \se \Lambda)$ with 
	\[
		\Lambda \defined \ker(\Oc_R^2 \twoheadrightarrow R^2 \twoheadrightarrow M)
		\qquad \text{and} \qquad \Lambda' \defined \Lambda \cap  \left(  \varpi^n \Oc_R \oplus \Oc_R \right) .
	\]
	But for an arbitrary $(\Lambda' \se \Lambda) \in \Uc_1(R)$ we always have $\Lambda' =  \Lambda \cap \left(  \varpi^n \Oc_R \oplus \Oc_R \right)$ inside $\Oc_R^2$ (as $\Lambda'$ has to be contained in both by definition and cannot be bigger than their intersection due to the restriction on the quotients).
	This shows that the two maps are indeed inverse to each other.
\end{enumerate}
\end{proof}

As a consequence we obtain the corresponding statement for the Beilinson-Drinfeld Schubert variety with arbitrary many legs.
\begin{corollary}
	The Beilinson-Drinfeld Schubert variety $\BDGrdNnaive$ is representable by a separated scheme of finite type over $\Fq$ which is smooth of dimension $2r$. 
	The projection $\BDGrdNnaive \to X^r$ is flat of relative dimension $r$ and smooth over $(X \setminus \SIw)^r$ where $ \SIw = \sum_{n_x = 1} [x] \subseteq \Sigma$ is the locus of Iwahori level structure.
	\thlabel{cor:smoothness-BDnaive}
\end{corollary}
\begin{proof}
	As $\BDGrdNnaive$ is an iterated fibration of global Schubert varieties with one leg it suffices to show the statements for $r =1$. Moreover, by duality it is enough to treat the case $\mu = \mu^+$. 
	At points with deeper than Iwahori level the statement is \thref{prop:schubert-deep-smooth} while at all other points the assertion is contained for example in \cite[Proposition 5.17]{Yun2019}.
\end{proof}

There is an action of $\Pic_X$ on $\HeckenGa$ by simultaneously tensoring all $\Ec_i^\dagger$. We denote the quotient by
$$ \HeckeGGa = \HeckenGa/\Pic_X.$$

We collect some (well-known) geometric properties of the Hecke stacks, compare \cite[Propositions 3.4 and 5.17]{Yun2019} for the corresponding statements in the Iwahori-case.
\begin{proposition}
	\begin{enumerate}
		\item 
		For every $0 \leq i \leq r$ the projection $\tilde p_i \colon \HeckenGa \to \BunGLnaive$, $\underline{\Ec}^\dagger \mapsto \Ec_i^\dagger$
		is representable by schemes and smooth of relative dimension $2r$.
		In particular, $\HeckenGa$ is representable by an algebraic stack smooth of dimension $4(g-1)+N+2r$ over $\Fq$.
		\item For every $0 \leq i \leq r$ the projection $p_i \times \pi \colon \HeckenGa \to \BunGLnaive \times X^r$ is flat of relative dimension $r(n-1)$ and smooth over $(X \setminus \SIw)^r$.
		\item The corresponding statements also hold for $\HeckeGGa$. In particular, $\HeckeGGa$ is smooth of pure dimension $3(g-1) + N + 2r$.
\end{enumerate}		
\thlabel{prop:geometry-Hecke}
\end{proposition}
\begin{proof}
	Either one notes that the arguments in the proof of \cite[Proposition 3.4]{Yun2019} carry over essentially verbatim to our situation (using \thref{cor:smoothness-BDnaive} as an additional input), or one appeals to the general theory, compare \cite[Proposition 3.9]{Rad2019a} or \cite[Lemma 3.4.1]{Bieker2022a}, and uses \thref{cor:smoothness-BDnaive}.
\end{proof}

\subsection{Fractional twists and the Atkin-Lehner action}\label{sss:frac-twist} 
We generalize the notion of fractional twists of \cite[§3.1.1]{Yun2019} to the deeper level setting where the divisor $\Sigma$ is not necessarily reduced.
In particular, for a rank 2 bundle with naive $\Gamma_0(\Sigma)$-level structure  $\Ec^\dagger = (\Ec \twoheadrightarrow\Lc)$ we define rank 2 vector bundles $\Ec\left( -\frac{D}{2}\right)$ for any effective divisor $0 \leq D \leq 2 \Sigma$. 
Note that here $\frac{D}{2}$ really should be understood as a formal expression and for points $x \in \Sigma_{\red}$ with $n_x \geq 2$ the ``fractional'' twist $\Ec(-\frac{2x}{2})$ will differ from the usual twist $\Ec(-x) = \Ec \otimes \Oc(-x)$.

Let us first assume $D = \sum_{x \in \Sigma_{\red}} c_x [x]$ with  $0 \leq c_x \leq n_x$ for all $x$. Then 
$$ \Ec\left( - \frac{D}{2} \right) \defined \ker \left( \Ec \twoheadrightarrow \bigoplus_{x \in \Sigma_{\red}} \Lc|_{c_x [x]} \right).$$
In general for $D = \sum_{x \in \Sigma_{\red}} c_x [x]$ we set $D' =  \sum_{x \in \Sigma_{\red}} c'_x [x]$ with $c'_x = c_x$ if $c_x \leq n_x$ and $c'_x = n_x$ otherwise. Then
$$ \Ec\left(- \frac{D}{2} \right) \defined \ker \left(\Ec\left(-\frac{D'}{2}\right) \twoheadrightarrow \bigoplus_{x \colon c_x > n_x} \Ec\left(-\frac{n_x[x]}{2}\right) / \Ec(-n_x[x])\big|_{(c_x - n_x)[x]}\right).$$
In particular, when $D = 2\Sigma$ we obtain $\Ec\left(-\frac{D}{2}\right) = \Ec(-\Sigma)$. When $\Sigma$ is reduced we recover the corresponding fractional twists of \cite{Yun2019}.

Let us also recall the finer fractional twists at ``infinite'' places from \cite[§3.1.2]{Yun2019}.
We suppose we are given a decomposition
\begin{equation}
	|\Sig|= \Si \coprod \Sf
\end{equation} 
into a disjoint union of two subsets such that $\Sig$ is reduced at all points $x \in \Si$. We call $\Si$ the set of infinite places.
Let  $\frSi=\prod_{x\in \Si}\Spec k(x),$ where the product is taken over $\Spec(k)$.
A map $S \to \frSi$ is then given by a tuple $\{x^{(1)}\}_{x\in \Si}$ where $x^{(1)} \colon \frSi \to \Spec(k(x)) \hookrightarrow X$.
For a positive integer $i$ let $x^{(i)} \defined \Fr_{S}^{(i-1)} \circ x^{(1)}$ be the precomposition of $x^{(1)}$ with a power of the absolute Frobenius $\Fr_{S} \colon S \to S$. Then $x^{(i+d_x)} = x^{(i)}$ where $d_x = [k(x) \colon k]$ is the degree of $x$.
In particular, for every $x \in \Si$ projection to the $x$ factor  induces a map
	$\bx^{(1)} \colon \frS_{\infty}\to\Spec k(x)\to X$.
As in \cite{Yun2019} by a slight abuse of notation we identify $\bx^{(1)}$ with its graph  $\Gamma_{\bx^{(i)}}$, a divisor in $X\times\frS_{\infty}$. 
Then we have a decomposition
\begin{equation*}
	\{x\}\times \frS_{\infty}=\Spec k(x)\times \frS _{\infty}=\coprod_{i=1}^{d_{x}} \bx^{(i)}.
\end{equation*}

For $\cE^{\da}\in (\BunGLnaive \times \frSi)(S)$ and every $x \in \Si$, the quotient $\cE/\cE(-\frac{x}{2}) \cong \Lc|_{\Spec(k(x)) \times S}$ then splits as a direct sum $\oplus_{j=1}^{d_{x}}\cQ^{(j)}$ where $\cQ^{(j)}$ is a line bundle supported on $\Gamma_{x^{(j)}}$. We define $\cE(-\frac{x^{(j)}}{2})$ to be the kernel
\begin{equation*}
	\cE\left(-\frac{x^{(j)}}{2}\right):=\ker\left(\cE\to \cE/\cE(-\frac{x}{2})\surj \cQ^{(j)}\right).
\end{equation*}

We put together the constructions above. We use the notation $\Div_{\Sig}(X) = \bigoplus_{x \in |\Sig|} \Z$ for the group of divisors on $X$ supported on $\Sig$.  
We use the shorthands $\sD_f \defined \Div_{\Sf}(X)$ and $\sD_\infty \defined \Div_{\Si \times \frSi}(X \times \frSi)$.
In other words, an element of $\sD_\infty$ is of the form $\sum_{x \in \Si}\sum_{1\leq j \leq d_x} c_x^{(j)} \bx^{(j)}$ for integers $c_x^{(j)}$. 
Define

\[
	\DivSSi \defined \left( \Div(X \times \frSi) \oplus \frac{1}{2} \sD_f \oplus \frac{1}{2} \sD_\infty  \right) / \left\langle \begin{array}{c}
		 \left( n_{x_f} [x_f], 0 ,0 \right)  = \left( 0,  \frac{2n_{x_f}[x_f]}{2}, 0 \right) \\ 
		 \left(\left[ \bx^{(j)} \right], 0,0 \right)  = \left( 0, 0, \frac{2 \left[\bx^{(j)} \right]}{2} \right) 
	\end{array}  
	\colon 
	\begin{array}{c}
		x_f \in |\Sf|,\\
		x \in |\Si|	
	\end{array}
  \right\rangle.
\]
An element $D \in \DivSSi$ thus has a unique representative 
\begin{equation}\label{D for twist}
	D= D_0 - \frac{D_f}{2} - D_\infty
\end{equation}
where $D_0\in \Div(X\times \frSi)$, $0 \leq D_f < 2\Sf$ and $D_\infty =  \sum_{x\in \Si, 1\le j\le d_{x}}c^{(j)}_{x}\bx^{(j)} \in \frac{1}{2} \sD_\infty$ with $c_x^{(j)} \in \{0, 1/2\}$.

Using the discussion above we define the twist $\cE(D)$ for $(\cE^{\da},  \{x^{(1)}\}_{x\in \Si})\in (\BunGLnaive \times \frSi)(S)$ by 
\[
	\cE(D) \defined \left( \left( \Ec(-D_f/2) \right) (-D_\infty) \right) (D_0)
\]
One checks that for $D_1, D_2 \in \DivSSi$ we have $\Ec(D_1 + D_2) = \Ec(D_1)(D_2)$.
We use the fractional twists to extend the construction of Atkin-Lehner operators of \cite[Definition 3.2]{Yun2019} to not necessarily square-free level.
\begin{defn} Let $D \in \DivSSi$. The {\em Atkin--Lehner automorphisms} for $\BunGLnaive$ and $\BunGnaive$ are given by 
	\[	
	\wt\AL(D): \BunGLnaive \times\frSi\to \BunGLnaive, \qquad \text{and} \qquad
		\AL(D):\BunGnaive \times \frSi\to \BunGnaive
	\]
	with
	\begin{equation*}
		\cE^{\da}=\left(\cE; \cE\left(-\frac{\Sig}{2}\right); \{x^{(1)}\}_{x\in \Si}\right) \mapsto \cE^{\da}(D) \defined \left(\cE(D); \cE\left(D-\frac{\Sig}{2}\right)\right).
	\end{equation*}
	\thlabel{def:al-operators-Bun}
\end{defn}
By construction $\AL(D)$ depends only on the class $D \in \DivSSi / \Div(S \times \frSi)$ and thus induces an involution when $2D = 0$, e.g. when $D = D_\infty$.

\subsection{Shtukas with (naive) $\Gamma_0(\Sigma)$-level structure}
We introduce the moduli space of shtukas with (naive) $\GammaN$-level structure and supersingular legs at infinity, generalising the Iwahori case discussed in \cite[§3.2]{Yun2019}.
As before, we fix $\underline{\mu} = (\mu_i)_{1 \leq i \leq r} \in \{\pm 1\}^r$ and by a slight abuse of notation identify $\mu_i$ with either of the cocharacters $\mu_+ = (1,0)$ or $\mu_-= (0,-1)$ of $\GL_2$ depending on the sign of $\mu_i$.
Morevoer, let $D_\infty = \sum_{x \in \Si, 1 \leq i \leq d_x} c_x^{(i)} \bx^{(i)} \in \frac{1}{2}\sD_\infty$ such that 
\begin{equation}
	\sum_{i= 1}^r \mu_i = \sum_{x \in \Si, 1 \leq i \leq d_x} 2 c_x^{(i)}.
	\label{eq:cond-non-empty}
\end{equation}

\begin{definition}
	\thlabel{def:sht}
	The \emph{moduli stack of rank $2$ shtukas with (naive) $\GammaN$-level structure and supersingular legs along $D_\infty$}, denoted by $\ShtSiD$, is the stack parametrizing tuples 
	\[
		\underline{\Ec}^\dagger = \left( \Ec_0^\dagger \overset{f_1}{\underset{x_1}{\dashrightarrow}} \Ec_{1}^\dagger \overset{f_{2}}{\underset{x_2}{\dashrightarrow}} \ldots \overset{f_{r}}{\underset{x_r}{\dashrightarrow}} \Ec_r^\dagger \xrightarrow[\theta]{\cong} ({}^\tau \Ec_0^\dagger)(D_\infty)\right),	
	\]
	where 
	\begin{itemize}
		\item $\left( \Ec_0^\dagger \overset{f_1}{\underset{x_1}{\dashrightarrow}} \Ec_{1}^\dagger \overset{f_{2}}{\underset{x_2}{\dashrightarrow}} \ldots \overset{f_{r}}{\underset{x_r}{\dashrightarrow}} \Ec_r^\dagger \right) \in \HeckenGa(S)$,
		\item $(x^{(1)})_{x \in \Si} \in \frSi (S)$, 
		\item and $\theta \colon \Ec_r \cong (^{\tau}\Ec_0)(D_\infty)$ is an an isomorphism compatible with the level structures. 
	\end{itemize}
	Let $\ShtGSiD = \ShtSiD/\Pic_{X}(\Fq)$ where the discrete groupoid $\Pic_{X}(\Fq)$ acts by simultaneous twisting. 
\end{definition}

Then the natural diagram
\begin{equation}
	\begin{tikzcd}
		\ShtSiD \arrow[r] \arrow[d] & \HeckenGa \times \frSi \arrow[d, "{(\wt{p}_0 , \wt{\AL}(-D_\infty) \circ (\wt{p}_r \times \mathrm{id}_{\frSi}))}"] \\
		\Bun_{2}(\Sig) \arrow[r, "{(\mathrm{id}, \Fr)}"] & \Bun_{2}(\Sig)  \times \Bun_{2}(\Sig)
	\end{tikzcd}
	\label{diag:def-ShtSiD}
\end{equation}
is Cartesian, and the analogous diagram for $\ShtGnaive$ is Cartesian as well.

\subsubsection{Comparison to other constructions of integral models at deeper level}
As mentioned in \thref{rem:def-Bunnaive}, our construction of $\ShtSiD$ recovers the moduli space of shtukas for the Bruhat-Tits group scheme corresponding to $\GammaN$-level. 
The supersingular legs can be encoded into a boundedness condition of ``finite type'' as in \cite[Definition 2.6.1]{Hartl2023} and \cite[Definition 2.2.1]{Bieker2022a}, compare \cite[Lemma 3.7]{Yun2019} and \cite[§2.8]{Hartl2023}.

The case where $r = 1$ and $\Si = \{\infty\}$ consists of a single point is closely related to the moduli stack of Drinfeld $A = \Gamma(X - \{\infty\}, \Oc_X)$-modules, compare \cite[3.2.3]{Yun2019} or \cite[Theorem 3.1.4]{Blum1997}, and in this sense is a close analogue of the modular curve in the function field setting.
In both the modular curve setting as well as in the Drinfeld module/Drinfeld shtuka setting one can define $\GammaN$-level structures in the not necessarily square-free case using Drinfeld level structures, compare \cite{Katz1985} (for the modular curve) and \cite{Bieker2023}. 
We do not recover the full moduli stack of shtukas with Drinfeld level structure in this setting, instead $\ShtSiD$ can be identified with an open dense substack of the former, compare \cite[Proposition 6.5]{Bieker2023}.
More precisely, in the fiber over a point $x \in \Sf$ that occurs with multiplicity $n_x \geq 2$ we only recover 2 of the $(n_x +1)$-many components that appear in the moduli stack defined using Drinfeld level structures, and also the supersingular points are missing, compare \cite[Remark 2.20]{Bieker2023}.

When $r \geq 1$ there does not seem to be a good definition of Drinfeld level structures, in particular for points over the diagonals of $X^r$.
Using input from Bruhat-Tits theory and functoriality properties of moduli spaces of shtukas \cite{Bieker2022a} gives a construction of deeper level integral models that on the one hand works for general groups $G$ and arbitrary many legs $r$ and contains the ``naive'' analogue of $\ShtSiD$ as an open dense substack, and on the other hand recovers the moduli stack of Drinfeld shtukas with Drinfeld level structures in the setup discussed above (even for higher rank).

We do not require these more intricate constructions of integral models as it turns out that the relevant Heegner-Drinfeld cycles in the fibers above points $ x \in \Sf$ with higher multiplicities are already supported on $\ShtSiD$. 
For the corresponding statement for the modular curve compare \cite[III, Proposition 3.1]{Gross1986}.

\subsubsection{Geometric properties}
We recall some of the geometric properties of the stacks $\ShtSiD$. Note that condition \eqref{eq:cond-non-empty} precisely guarantees the non-emptiness of $\ShtSiD$.
\begin{proposition}[{cf. \cite[Proposition 3.9]{Yun2019} in the Iwahori case}]
	\thlabel{lem:geometry-YZ}
	\begin{enumerate}
		\item 
		The stack $\ShtSiD$ is separated and smooth Deligne-Mumford stack of dimension 2r over $\Fq$. 
		\label{lem:geometry-YZ-sepft}
		\item 
		The projection $\Pi_{2, D_\infty}^{\un \mu} \colon \ShtSiD \to X^r \times \frSi $ is flat of relative dimension $r$, 
		and its restriction 
		\[
			\ShtSiD|_{(X \setminus \SIw) \times \frSi} \to (X \setminus \SIw)^r \times \frSi
		\] 
		is smooth.
		\label{lem:geometry-YZ-flat}
		\item 
		The analogous assertions hold for $\ShtGSiD$ and $\Pi_{G, D_\infty}^{\un \mu} \colon \ShtGSiD \to X^r \times \frSi $.
	\end{enumerate}
\end{proposition}
\begin{proof}
	This follows as in the proof of \cite[Proposition 3.9]{Yun2019} using \thref{cor:smoothness-BDnaive} and \thref{prop:geometry-Hecke}.
	 
	Alternatively, the representability follows from the general theory, compare \cite[Theorem 3.15]{Rad2019a} and \cite[Theorem 2.7.9]{Hartl2023}. The remaining assertions follow from the local model theorem (in the form of \cite[Theorem 2.7.12]{Hartl2023} and \cite[Proposition 3.4.3]{Bieker2022a}) together with  \thref{cor:smoothness-BDnaive}. 
\end{proof}

\subsubsection{Base change}
We also consider the base change situation as in \cite[§3.2.6]{Yun2019}.
Let $X'$ be a second smooth projective and geometrically connected curve over $\Fq$ and let $\nu \colon X' \to X$ be a double cover that is unramified along $\Sigma$.
We set $\frSi' \defined \prod_{x' \in \nu^{-1}(\Si)} k(x')$. 
Let us consider the base change
$$\ShtSiDp \defined \ShtSiD \times_{X^r \times \frSi} (X')^r \times \frSi',$$
and define $\ShtGSiDp$ analogously.
The geometric properties of the corresponding stacks over $X^r$, only the smoothness over $\Fq$ requires an argument:
\begin{proposition}[{cf. \cite[Proposition 3.10]{Yun2019} in the Iwahori case}]
	\thlabel{p:bc smooth}
	The stack $\ShtGSiDp$ is a smooth Deligne-Mumford stack of dimension $rn$ over $\Fq$. 
\end{proposition}
\begin{proof}
	The proof of \cite[Proposition 3.10]{Yun2019} carries over verbatim.
\end{proof}

\subsubsection{Independence of data and $\ShtGSii$}
As in \cite[Lemma 3.6]{Yun2019} the stack $\ShtSiD$ up to canonical isomorphism only depends on $\sum_i c_x^{(i)}$ for every $x \in \Si$, and thus $\ShtGSiD$ only depends on the fractional part of $\sum_i c_x^{(i)}$ for every $x \in \Si$, or in other words, on the image of $D_\infty$ in $ \frac{1}{2}\Z/ \Z \otimes \Div_{\Si}(X)$, compare \cite[Lemma 3.8]{Yun2019}.

In the following we work mostly with a particular choice of $D_\infty$, compare \cite[§3.2.8]{Yun2019}.
Let $D_\infty^{(1)} \defined \sum_{x \in \Si} \frac{\bx^{(1)}}{2}$.
We assume that 
\begin{equation}
	\label{eq:parity-r}
	r \equiv \# \Si \mod 2	
\end{equation}
and let $D_\infty \equiv D_\infty^{(1)} \mod \Dc_\infty$ such that condition \eqref{eq:cond-non-empty} is satisfied.
We set 
$$\ShtGSii \defined \ShtGSiD.$$
By the discussion above this stack is canonically independent of the choices of $\underline{\mu}$ and $D_\infty$.
We denote the projection to the legs by $\Pi^r_G \colon \ShtGSii \to X \times \frSi$. 

\subsubsection{Atkin-Lehner involution}
\label{sect:al-involution-sht}
As in \cite[§3.2.7]{Yun2019} the Atkin-Lehner automorphisms on $\Bun_2(\Sigma)$ and $\HeckenGa$ induce Atkin-Lehner automorphisms $\wt{\AL}_\Sht(- \frac D 2)$ and $\AL_\Sht(- \frac D 2)$ for divisors $D$ supported on $|\Sig|$ on $\ShtSiD$ and $\ShtGSiD$, respectively.
In particular, for every $x \in \Sig$ we obtain involutions $\AL_{\Sht, x} \defined \AL(-\frac{n_x x}{2})$ on $\BunGnaive$ as well as $\ShtGSii$.

We use the shorthand notation $\AL_{\infty} \defined \AL(-D^{(1)}_{\infty}) \colon \BunGnaive \times \Si \to \BunGnaive$. Then $\ShtGSii$ sits in a Cartesian diagram
\begin{equation}
	\begin{tikzcd}
		\ShtGSii \arrow[r] \arrow[d, "p_0"] & \Hecke^r_G(\Sig) \times \frSi \arrow[d, "{({p}_0 , {\AL_\infty} \circ ({p}_r \times \mathrm{id}_{\frSi}))}"] \\
		\BunGnaive \arrow[r, "{(\mathrm{id}, \Fr)}"] & \BunGnaive  \times \BunGnaive
	\end{tikzcd}
	\label{diag:def-ShtGSii}
\end{equation}
as in \cite[(3.21)]{Yun2019}.

\subsection{Hecke action}

We extend the Hecke action on Chow groups and cohomology of $\ShtSiD$ from the Iwahori case in \cite[§3.3]{Yun2019} to our setting of not necessarily square-free level.

\subsubsection{Hecke correspondence}\label{sss:Hk action on Sht}
Let $D$ be an effective divisor on $X-\Sig$. 
We define a self-correspondence $\Sht^{r}_{G}(\Sii;h_{D})$ of $\Sht^r_G(\Sii)$ over $X^r \times \frSi$.

\begin{definition}
	Assume $\un\mu\in\{\pm1\}^{r}$ and $\Di=\sum_{x\in\Si}c_{x}\bx^{(1)}$ satisfy \eqref{eq:cond-non-empty}. Let $\Sht^{\un\mu}_{2}(\Sig;\Di;h_{D})$ be the stack whose  $S$-points classify tuples $(\underline{\cE}^\da, \underline{\cE}'^\da, (\ph_i)_{0 \leq i \leq r})$ where
		\begin{itemize}
			\item the two points $\un{\cE}^{\da} = ((\cE^{\da}_{i})_i, (f_{i})_i, \theta, (x_i)_i, (x^{(1)})_{x \in \Si})$ and $\un{\cE}'^{\da} = (({\cE'}^{\da}_{i})_i, (f'_{i})_i, \theta', (x_i)_i, (x^{(1)})_{x \in \Si})$ of the stack $\Sht^{\un\mu}_{2}(\Sig;\Di)(S)$ map to the same point $(x_{1},\dotsc, x_{r}, \{x^{(1)}\})\in (X^{r}\times \frSi)(S)$, 
			\item for each $i = 0,1,\dotsc, r$ the map $\ph_{i}: \cE_{i}\to \cE'_{i}$ is an injective map of coherent sheaves compatible with the level structures, such that $\det(\ph_{i}) \colon \det(\cE_{i})\to\det(\cE'_{i})$ has divisor $D\times S\subset X\times S$, and such that the diagram
			\begin{equation}
				\label{Sht hD diag}
				\xymatrix{\cE_{0}\ar@{-->}[r]^{f_{1}}\ar[d]^{\ph_{0}} & \cE_{1}\ar@{-->}[r]^{f_{2}}\ar[d]^{\ph_{1}} & \cdots\ar@{-->}[r]^{f_{r}} & \cE_{r}\ar[d]^{\ph_{r}} \ar[r]^-{\io} & ({}^{\tau}\cE_{0})(\Di)\ar[d]^{{}^{\tau}\ph_{0}}       \\
					\cE'_{0}\ar@{-->}[r]^{f'_{1}} & \cE'_{1}\ar@{-->}[r]^{f'_{2}} & \cdots\ar@{-->}[r]^{f'_{r}} & \cE'_{r} \ar[r]^-{\io'} & ({}^{\tau}\cE'_{0})(\Di)}
			\end{equation}
			is commutative.
		\end{itemize} 
		Let $\Sht^{r}_{G}(\Sii;h_{D})=\Sht^{\un\mu}_{2}(\Sig;\Di;h_{D})/\Pic_{X}(k)$, which is independent of the choices as in the definition of $\Sht^{r}_{G}(\Sii)$ above. The stack $\Sht^{r}_{G}(\Sii;h_{D})$ has two maps $\oll{p}$ and $\orr{p}$ to $\Sht^r_G(\Sii;h_D)$ given by projection to $\underline{\cE}^\da$ and $\underline{\cE}'^\da$, respectively.
\end{definition}

We have the following analogue of \cite[Lemma 3.13]{Yun2019} which can be shown by verbatim the same proof (as we only consider Hecke correspondences for $D$ away from the level $\Sigma$).
\begin{lemma}\thlabel{lem:geometryShthD} Let $D$ be an effective divisor on $X-\Sig$.
	\begin{enumerate}
		\item The two maps $\oll{p},\orr{p}: \Sht^{r}_{G}(\Sii;h_{D})\to \Sht^{r}_{G}(\Sii)$ are representable and proper.
		\item The restrictions of  $\oll{p}$ and $\orr{p}$ over $(X-D)^{r}$ are finite \'etale.
		\item The fibers of $\Pi^{r}_{G}(h_{D}): \Sht^{r}_{G}(\Sii;h_{D})\to X^{r}\times\frSi$ all have dimension $r$.
	\end{enumerate}
\end{lemma}

\subsubsection{Hecke action on the Chow group and on cohomology}
\label{sss:Hk Chow}
Recall from \cite[Appendix A.1.6]{Yun2017} that the (rational) ring of correspondences for a Deligne-Mumford stack locally of finite type over $k$ is defined to be 
\[
	{}_{c}\Ch_{n}(X \times X)_{\Q} \defined \colim\limits_{\substack{Z \se X \times X \\ \mathrm{pr}_1 \colon Z \to X \mathrm{ proper}}} \Ch_n(Z)_\Q,
\]
which carries a convolution product as well as an action on the Chow groups of proper cycles $\Ch_{c,*}(X)$ induced by the intersection product. 
Then by \thref{lem:geometryShthD} the fundamental cycle on the Hecke correspondence for $D$ gives rise to a class
\[
	(\oll{p} \times \orr{p})_{*}[\Sht^{r}_{G}(\Sii;h_{D})] \in {}_{c}\Ch_{2r}(\ShtGSii \times \ShtGSii)_{\Q}.
\] 
To get a Hecke action on cohomology we apply the construction of comological correspondences, compare for example \cite[Appendix A.4]{Yun2017}.
We present $\ShtGSii$ as an increasing union of open substacks of finite type as in \cite[§7.1.4]{Yun2017} (compare also Section \ref{sec:truncation}) by truncating the degree of instability.
While the Hecke correspondence defined above does not restrict to a correspondence on the bounded pieces (the degree of instability increases), after taking colimits we obtain a well-defined map
\[
	C(h_D) \colon \cohoc{i}{\Sht^{r}_{G}(\Sii)\ot\kbar,\Ql} \to \cohoc{i}{\Sht^{r}_{G}(\Sii)\ot\kbar,\Ql}.
\]
for every $i \geq 0$.

Using that $h_D$ for effective divisors $D \se X - \Sig$ form a basis of $\sH^{|\Sig|}_{G}$, we can extend 
\[
	h_D \mapsto (\oll{p} \times \orr{p})_{*}[\Sht^{r}_{G}(\Sii;h_{D})] \qquad \text{and} \qquad h_D \mapsto C(h_D)
\]
to $\Q$-linear (respectively $\Q_\ell$-linear) maps
\begin{equation}
	\label{eqn:Hk-action-Ch}
	\sH^{|\Sig|}_{G} \to {}_{c}\Ch_{2r}(\ShtGSii \times\ShtGSii))_{\Q}
\end{equation} 
and
\begin{equation}
	\label{eqn:Hk-action}
	\sH^{|\Sig|}_{G} \otimes_\Q \Q_\ell \to \End_{\Ql}\left(\cohoc{i}{\Sht^{r}_{G}(\Sii)\ot\kbar,\Ql}\right).
\end{equation}
Using the geometric properties of the Hecke correspondence in \thref{lem:geometryShthD} together with the smoothness of $\ShtGSii$, the same arguments as in \cite[Proposition 5.10, Proposition 7.1, Lemma 5.12, Lemma 7.2, and Lemma 7.3]{Yun2017} prove the following results, compare also \cite[§3.3.2, §3.3.3]{Yun2019} for the Iwahori-case.

\begin{prop} 
	\thlabel{prop:Hecke-action-Chow-coh}
	\begin{enumerate}
		\item 
		\label{p:Hk Chow action} 
		The two maps \eqref{eqn:Hk-action-Ch} and \eqref{eqn:Hk-action} are ring homomorphisms.
		\item \label{l:Chow coho equiv}
		The cycle class map
		\begin{equation*}
			\cl: \Ch_{c,i}(\Sht^{r}_{G}(\Sii))_{\QQ}\to \cohoc{4r-2i}{\Sht^{r}_{G}(\Sii)\ot\kbar,\Ql}(2r-i)
		\end{equation*}
		is equivariant under the $\sH^{|\Sig|}_{G}$-actions for all $i$.
		\item \label{l:self adjoint}
		Let $f\in \sH^{|\Sig|}_{G}$. Then the action of $f$ on the Chow group $\Ch_{c,*}(\Sht^{r}_{G}(\Sii))_{\QQ}$ (resp. on the cohomology $\cohoc{2r}{\Sht^{r}_{G}(\Sii)\ot\kbar,\Ql}(r)$) is self-adjoint with respect to the intersection pairing (resp. cup product pairing).
	\end{enumerate}
\end{prop}

For a quadratic extension $X'$ of $X$ as in \thref{p:bc smooth} let the Hecke correspondence $\Sht'^{r}_{G}(\Sii;h_{D})$ for  $\Sht'^{r}_{G}(\Sii)$ be defined as the base change
\begin{equation*}
	\Sht'^{r}_{G}(\Sii; h_D)=\Sht^{r}_{G}(\Sii; h_D)\times_{(X^{r}\times \frSi)}(X'^{r}\times \frSi').
\end{equation*}
Using the smoothness of $\Sht'^{r}_{G}(\Sii)$ (compare \thref{p:bc smooth}) the analogous statements of \thref{prop:Hecke-action-Chow-coh} also hold for the base change $\Sht'^{r}_G(\Sii)$, compare \cite[§3.3.4]{Yun2019}.

\begin{remark}\label{r:AL Sht} 
	As in the Iwahori-case in \cite[Remark 3.18]{Yun2019} the Atkin--Lehner involutions $\AL_{\Sht,x}$ for $x \in |\Sig|$ induce involutions on the Chow groups and cohomology groups of $\Sht^{r}_{G}(\Sii)$ and $\Sht'^{r}_{G}(\Sii)$, which we still denote by $\AL_{\Sht, D}$. These involutions commute with the action of $\sH^{|\Sig|}_{G}$.
\end{remark}

\subsection{Heegner-Drinfeld cycles at deeper level}

We generalize the construction of Heegner-Drinfeld cycles of \cite[§5.5]{Yun2017} and \cite[§4]{Yun2019} to our setting for not necessarily square-free level.
As above, we fix a second smooth projective and geoemtrically connected curve $X'$ over $k$ with function field $F'$ together with a double cover $\nu \colon X' \to X$  satisfying the generalized Heegner hypothesis \eqref{eq:Heegner-hypothesis}.
In other words, $\nu$ is split at all finite places $\Sf$ and inert at all infinite places $\Si$.
In particular, $\nu$ is unramified along $\Sig$. 
We denote by $\Sf' = \nu^{-1}(\Sf)$ and $\Si' = \nu^{-1}(\Si)$. Then the restriction of $\nu$ to $\Sf'$ (respectively $\Si'$) is 2-to-1 (respectively a bijection). 

We use the analogous notation for divisors at infinity for $X'$ as in Section \ref{sss:frac-twist} above for $X$:
For $x\in\Si$ we denote its preimage in $\Si'$ by $x'$ (whose residue field $k(x')$ has degree $d_x' = 2d_x$ over $k$) and set
\begin{equation*}
	\frSi'=\prod_{x'\in\Si'}\Spec k(x').
\end{equation*}
Moreover, let $\sD_\infty' \defined \Div_{\Si' \times \frSi'}(X' \times \frSi')$ be the group of ($\Z$-valued) divisors on $X' \times \frSi'$ supported on $\Si' \times \frSi'$.
We continue to assume that the pair $(r, \Si)$ satisfy the parity condition \eqref{eq:parity-r}.

\sss{The moduli space of $T$-shtukas} 
We consider the $X$-tori $\tilde{T} = \Res_{F'/F}\G_{m}$ and $T = \tilde{T}/\G_{m}$ and their associated moduli space of $T$-shtukas.
We get $\Bun_T \defined \Pic_{X'}/\Pic_X$ as well as the Hecke stack $\Hk_T^{\un \mu} \defined \Hk_{1, X'}^{\un \mu}$.
Associated to divisors at infinity on $X'$ we have Atkin-Lehner operators over $X'$ defined similar to Section \ref{sect:al-involution-sht}.
Then we can consider $\Sht_{1, X'}^{\un \mu}(D'_\infty)$ as in \cite[§4.1.3]{Yun2019}, defined as in \thref{def:sht}. The discrete Picard stack $\Pic_X(k)$ acts on $\Sht_{1, X'}^{\un \mu}(D'_\infty)$ and we define $\Sht_{T}^{\un \mu}(D'_\infty) = \Sht_{1, X'}^{\un \mu}(D'_\infty)/ \Pic_X(k)$ as the quotient. 

Below we are mainly interested in a particular choice of $D'_\infty$ defined as follows.
For each $\mi=(\mu_{x})_{x\in \Si}\in\{\pm1\}^{\Si}$ we define the divisor
\begin{equation*}
	\mi\cdot\Si' \defined \sum_{x\in\Si} \mu_{x}\bx'^{(1)}
\end{equation*}
on $X' \times \frSi$, which is supported on $\Si \times \frSi$.
Let us fix a divisor $D'_\infty = \sum_{x' \in \Si'} c_{x'} \bx'^{(1)} \in \sD_\infty'$ such that 
\begin{equation}
	\sum_{i = 1}^r \mu_i = \sum_{x' \in \Si'} c_{x'} \quad \text{and} \quad D_\infty' \equiv \mi \cdot \Si \mod \nu^* \sD_\infty.
	\label{eq:condition-D-prime}
\end{equation}

Let us explicate the definition of the corresponding moduli stack of shtukas.
\begin{definition}[{\cite[Definition 4.4]{Yun2019}}]
	The stack $\Sht_{\wt T}^{\un \mu}(D_\infty')$ parametrizes the data 
	\[
	\underline{\Lc} = \left( \Lc_0 \overset{f_1}{\underset{x'_1}{\dashrightarrow}} \Lc_{1} \overset{f_{2}}{\underset{x'_2}{\dashrightarrow}} \ldots \overset{f_{r}}{\underset{x'_r}{\dashrightarrow}} \Lc_r \xrightarrow[\theta]{\cong} ({}^\tau \Lc_0)(D'_\infty)\right),	
	\]
	where $x'_i \in X(S)$ for $1 \leq i \leq r$, $(x'^{(1)})_{x'\in \Sig'_\infty} \in \frSi'(S)$, the $\Lc_i$ are line bundles on $X' \times S$, $f'_i \colon \Lc_{i-1} \dashrightarrow \Lc_{i}$ is a modification of type $\mu_i$ along $\Gamma_{x_i'}$ and $\theta \colon \Lc_r \cong {}^\tau\Lc_0(D_\infty')$ is an isomorphims.
	
	The discrete Picard stack $\Pic_X(k)$ acts again on $\Sht_{\wt T}^{\underline{\mu}}(\mi \cdot \Si)$ by simultaneous twisting of the line bundles.
	The stack $\Sht_T^{\un \mu}(\mi \cdot \Si)$ is defined as the quotient $\Sht_{\wt T}^{\underline{\mu}}(\mi \cdot \Si) / \Pic_X(k)$. 
\end{definition}
As for $\Sht_G$, the first part of Condition \eqref{eq:condition-D-prime} guarantees the non-emptiness of $\Sht_{\wt T}^{\un \mu}(D_\infty')$. 
Recall by \cite[Lemma 4.1]{Yun2019} that $\Sht_T^{\un \mu}(\mi \cdot \Si)$ does not depend on the choice of $D_\infty'$ satisfying Condition \eqref{eq:condition-D-prime}. 
We recall the following geometric properties of $\Sht_T^{\un \mu}(\mi \cdot \Si)$. 
\begin{lemma}[{\cite[Corollary 4.3]{Yun2019}}]
	\thlabel{lem:shtT-proper}
	The moduli space of $T$-shtukas $\Sht_T^{\underline{\mu}}(\mi \cdot \Si)$ is a torsor under the discrete groupoid $\Bun_T(k)$ over $X'^r \times \frSi'$. 
	In particular, it is a proper and smooth DM-stack of dimension $r$ over $k$.
\end{lemma}

Moreover, by  \cite[Lemma 4.2]{Yun2019} a point in $\Sht$ is uniquely determined by the initial line bundle $\Lc_0$ together with its legs $(x'_i)_{1 \leq i \leq r}$. 
\sss{Heegner-Drinfeld cycles}
We generalize the construction of Heegner-Drinfeld cycles to our setting.
The special cycles (either Heegner-Drinfeld or Manin-Drinfeld cycles) depend on a further auxiliary choice of a section $\mu_f \colon \Sf \to \Sf'$ of $\nu$ (recall that $\nu|_{\Sf'}$ is 2-to-1).
As in \cite{Yun2019} let $\Sect(\Sf'/\Sf)$ denote the set of such sections.
For a choice $\mu_f$ and some $x \in \Sf$ let $x'_{\mu_f} = \mu_f(x) \in \Sf'$ and $\Sig_{\mu_f}' = \sum_{x \in \Sf} n_x x_{\mu_f}'$.

Given the pair $\mu_{\Sig} = (\mu_f, \mi)$ we generalize the construction of the map
$$\theta^{\mu_\Sig}_{\Bun_{\wt T}} \colon \Bun_{\wt T} \times \frSi' \to \Bun_2(\Sig)$$
of \cite[§4.2.1]{Yun2019}.
Let $(\Lc, (x'^{(1)})_{x' \in \Sig_\infty'}) \in (\Bun_{\wt T} \times \frSi')(S)$.
Then its image is defined as follows.
\begin{itemize}
	\item Let $\Ec = \nu_*\Lc$, a rank two vector bundle on $X \times S$.
	\item The $\Gamma_0(\Sig)$-level structure at the finite places is given by $(\Ec \supseteq \nu_*(\Lc(-\Sig'_{\mu_f})))$.
	\item For $x \in \Si$ the level structure at $x$ is (as in \cite{Yun2019}) defined as  
	\begin{equation*}
		\cE(-\ha x)=\begin{cases}\nu_{S,*}(\cL(-\Gamma_{x'^{(1)}}-\Gamma_{x'^{(2)}}-\cdots-\Gamma_{x'^{(d_{x})}})) & \mu_{x}=1;\\
			\nu_{S,*}(\cL(-\Gamma_{x'^{(d_{x}+1)}}-\Gamma_{x'^{(d_{x}+2)}}-\cdots-\Gamma_{x'^{(2d_{x})}})) & \mu_{x}=-1.
		\end{cases}
	\end{equation*}	
\end{itemize}
In other words,
\[
	\theta^{\mu_\Sig}_{\Bun_{\wt T}}(\Lc, (x'^{(1)})_{x' \in \Sig_\infty'})) \defined \left( \nu_* \Lc \supseteq \nu_*(\Lc(-\frD'(\{x'^{(1)}\}))) \right)
\]
where $\frD'(\{x'^{(1)}\}) \se X' \times S$ is the degree $N$ relative effective Cartier divisor 
$$ \frD'(\{x'^{(1)}\}) \defined \sum_{x \in \Sf} n_x[\mu_f(x)] + \sum_{x \in \Si} \begin{cases} \Gamma_{x'^{(1)}}+\Gamma_{x'^{(2)}}+\cdots+\Gamma_{x'^{(d_{x})}} & \mu_{x}=1;\\
	\Gamma_{x'^{(d_{x}+1)}}+\Gamma_{x'^{(d_{x}+2)}}+\cdots+\Gamma_{x'^{(2d_{x})}} & \mu_{x}=-1.
\end{cases}.$$

After taking the quotient by $\Pic_X$ we obtain a map $\theta_{\Bun_T}^{\mu_{\Sig}} \colon \Bun_T\times \frSi' \to \Bun_G(\Sig)$.
By applying $\theta_{\Bun_T}^{\mu_{\Sig}}$ to every modification step we also have analogous maps on the corresponding Hecke stacks $\th_{\Hk_T}^{\mu_\Sig} \colon \Hk_T^r \to \Hk'^r_G(\Sig)$. 

Let $\frT_{r,\Sig}:=\{\pm1\}^{r}\times \Sect(\Sf'/\Sf)\times \{\pm1\}^{\Si}$.
For $\mu = (\un \mu, \mu_f, \mi)$ we get a map
$$ \theta^\mu_T \colon \Sht_T^{\un \mu}(\mi \cdot \Si) \to \Sht'^r_G(\Sii) $$
by applying $\theta_{\Bun_T}^{\mu_\Sig}$ to each $\Lc_i$ in the chain of modifications (one checks that this is indeed well-defined as in \cite[§4.2.2]{Yun2019}).
Using the properness results of \thref{lem:shtT-proper}, we define the following classes in the Chow groups of proper cycles.
\begin{definition}
	The \emph{Heegner-Drinfeld cycle} of type $\mu \in \frT_{r,\Sig}$ is the class $$\Zc^{\mu} = \theta_{T, *}^{\mu}[\Sht_T^{\underline{\mu}}(\mi \cdot \Si')] \in \Ch_{c,r}(\Sht'^r_G(\Sii))_\Q.$$
\end{definition}
The symmetries among the Heegner-Drinfeld cycles by the Atkin-Lehner operators discussed in \cite[§4.3]{Yun2019} carry over verbatim to our setting.

\begin{definition}
	For $\mu \in \frT_{r,\Sig}$ the linear functional $\II^{\mu} \colon \sH^{\Sig}_G \to \Q$ is defined as the intersection number
	\[
		\II^{\mu}(f) = \left( \prod_{x' \in \Sig'_\infty} d_{x'} \right)^{-1} \left\langle \Zc^{\mu}, f * \Zc^{\mu} \right\rangle_{\Sht'^r_G(\Sii)}.
	\] 
\end{definition}

As it will be convenient below to work over an algebraically closed base field we compare this intersection number with a similar one over an algebraic closure $\overline k$ of $k$. 
As in \cite[§4.3.5]{Yun2019} let us fix $\xi \in \frSi'(\overline{k})$, and let us denote its composition with the projection to $\frSi$ by a slight abuse of notation also by $\xi$. 
Then 
\[
	\Sht'^r_G(\Sii) = \coprod_{\xi \in \frSi'} \Sht'^r_G(\Si; \xi), 
\]
where $\Sht'^r_G(\Si; \xi) \defined \Sht'^r_G(\Sii) \times_{\frSi'} \xi$. 
Similarly, we define the base change to $\overline{k}$ of 
\[
	\Sht_T^{\un T}(\mu_\infty \cdot \xi) \defined \Sht_T^{\un T}(\mu_\infty \cdot \Sig_\infty') \times_{\frSi'} \xi.
\]
Moreover, let $\theta^{\mu}_{T, \xi}$ be the base change of $\theta^{\mu}_{T}$ to $\xi$. 
As over $k$ we define Heegner-Drinfeld cycles of type $\mu$ over $\xi$ as 
\[
	\Zc^{\mu}(\xi) \defined \theta'^\mu_{T, \xi, *} [\Sht^{\un \mu}_T(\mu_\infty \cdot \xi)] \in \Ch_{c,r}(\Sht'^r_G(\Si; \xi))
\]
Then we can compute $\II^{\mu}$ also over $\overline{k}$ as follows.
\begin{lemma}[{\cite[Lemma 4.14]{Yun2019}}]
	Let $\xi \in \frSi'$ and $\mu \in \frT_{r, \Sig}$. Then for every $f \in \sH^{\Sig}_G$ 
	\[
		\II^{\mu}(f) = \left\langle \Zc^{\mu}(\xi), f*\Zc^{\mu}(\xi) \right\rangle.
	\]
\end{lemma}
As the argument only involves the infinite places in an essential way, the argument of \emph{loc.~cit.} carries over verbatim.

\subsection{Intersection numbers as traces on a Hitchin fibration $\cM_d$}
We explain how to express the intersection numbers $\II^\mu(h_D)$ for the Hecke functions $h_D$ as defined above as traces of Frobenius on a certain Hitchin-type fibration using the Lefschetz trace formula.
As the arguments carry over essentially verbatim from \cite[§5]{Yun2019}, we only sketch the strategy.

\subsubsection{The Hitchin-type fibration $\cM_d$}

We generalize the definition of $\cM_d$ in \cite[Definition 5.1]{Yun2019} to deeper level. As noted before we consider only the self-intersection of the Heegner-Drinfeld cycles. This means that in the notation of \emph{loc.~cit.} we have $\Sig_- = \emptyset$ and $\Sig_+ = \Sig$.
\begin{definition}
	Let $\cM_d$ be the stack over $\Fq$ whose $S$-points parametrize the data
	$ (\Ic, \Jc, \alpha, \beta, \j)$,
	where
	\begin{itemize}
		\item $\Ic$ and  $\Jc$ are two line bundles on $X' \times S$ of fiberwise degree $d+\r$ and $d + \r - N$, respectively, 
		\item together with sections $\alpha$ of $\Ic$ and $\beta$ of $\Jc$,
		\item and an isomorphism $\j = (\j_{\mathrm{Nm}}, (\j_x)_{x \in R}) \colon \Nm^{\sqR}_{X'/X}(\Ic) \cong \Nm_{X'/X}^{\sqR}(\Jc)  \otimes \Oc_{X}(\Sig)^\nat$ as points of $\Pic^{\sqR}_{X}(S)$,
	\end{itemize}
	such that 
	\begin{enumerate}
		\item $\alpha|_{\nu^{-1}(\Sig)}$ is nowhere vanishing,
		\item for each $x \in R$ we have $\j_x(\a|_{x' \times S}) =  \b|_{x' \times S}$ and the section $\Nm_{X'/X}(\a) - \Nm_{X'/X}(\b)$ vanishes precisely to the first order along $R \times S$, and
		\item $\Nm_{X'/X}(\alpha) - \Nm_{X'/X}(\beta)$ is fiberwise on $S$ not identically zero.
	\end{enumerate}
\end{definition}

We get a map $f_d \colon \cM_d \to \cA_d$ defined by 
\[
	(\Ic, \Jc, \alpha, \beta, \j) \mapsto (\Nm_{X'/X}^{\sqR}(\Ic), \Nm_{X'/X}(\alpha), \Nm_{X'/X}(\beta), \alpha|_{R \times S} = \beta|_{R \times S})
\]
We denote $f_d^\flat = \Omega \circ f_d \colon \cM_d \to \cA_d^\flat$. 
We set $\cM_d^\Diamond \defined \cM_d \times_{\cA_d} \cA_d^\Diamond$. 
Let moreover $f_d^\Diamond$ and $f_d^{\flat, \Diamond}$ be the restrictions of $f_d$ and $f_d^\flat$ to $\cM_d^\Diamond$.  
We summarize the relevant geometric properties of the Hitchin-type fibration $f_d \colon \cM_d \to \cA_d$. 
\begin{prop}[{cf. \cite[Proposition 5.5]{Yun2019}}]
	\thlabel{prop:geometry-Md}
	\begin{enumerate}
		\item\label{M smooth} Assume that $d \geq 4g - 3 + \rho + N$. Then the stack $\cM_{d}$ is a smooth DM stack pure of dimension $2d+\r -N-g+1$. 
		\item\label{f proper} The morphisms $f_{d}$ and $f^{\fl}_{d}$ are proper, its restrictions $f_d^\Diamond$ and $f^{\fl, \Diamond}_d$ are finite.
		\item\label{f small} When $d\ge 3g-2+N$, the morphism $f_{d}$ is small, i.e., it is generically finite and for any $n>0$, $\{a\in\cA_{d}|\dim f_{d}^{-1}(a)\ge n\}$ has codimension $\ge 2n+1$ in $\cA_{d}$.
	\end{enumerate}
\end{prop}

Then $\cM_d$ carries a natural Hecke correspondence defined by the substack $\Hc \se \cM_d \times X'$ consisting of those points $(\Ic, \Jc, \alpha, \beta, \j, x')$ such that $\b$ vanishes along $\Gamma_{x'}$. Then $\Hc$ has two projections to $\cM_d$ given by $\oll{\g}(\Ic, \Jc, \alpha, \beta, \j, x') = (\Ic, \Jc, \alpha, \beta, \j)$ and  $\orr{\g}(\Ic, \Jc, \alpha, \beta, \j, x') = (\Ic, \Jc(\Gamma_{\sigma x'} - \Gamma_{x'}), \alpha, \beta, \j)$, exhibiting $\Hc$ as a self-correspondence of $\cM_d$:
\begin{equation}
	\begin{tikzcd}
		& \cH \arrow[dl, "\oll{\g}", swap] \arrow[dr, "\orr{\g}"]  & \\
		\cM_d \arrow[dr] & & \cM_d \arrow[dl] \\
		& \cA_d&
	\end{tikzcd}
\end{equation}

Then we define $\Hc^r \defined \Hc \times_{\cM_d} \ldots \times_{\cM_d} \Hc$ to be the $r$-fold composition of $\Hc$, and denote by $\overline\Hc^\Diamond$ the closure of its restriction to $\cA^\Diamond_d$.  

Let us fix $\mu=(\un\mu, \mu_{f},\mu_{\infty}) \in \frT_{r, \Sig}$ for the rest of this section.
We can now state the main theorem of this section relating the intersection numbers to traces of Frobenius-Hecke operators on the Hitchin-type fibration generalising \cite[Theorem 6.5]{Yun2017} and \cite[Theorem 5.6]{Yun2019}.

\begin{theorem}\thlabel{th:Ir} Suppose $D$ is an effective divisor on $U$ of degree $d\ge \max\{g''+N, 2g\}$. Then
		\begin{equation} \label{Ir hDp}
			\II^{\mu}_T(h_{D})=\sum_{a\in \cA^{\fl}_{D}(k)}\Tr\left((f^{\fl}_{d,!}[\ov\cH^{\dm}])^{r}_{a}\circ \Fr_{a}, (\bR f^{\fl}_{d,!}\Ql)_{a}\right).
		\end{equation}
	Here $\Fr_{a}$ is the geometric Frobenius at $a$.
\end{theorem}

The proof of Theorem \ref{th:Ir} works very similar to the corresponding argument in \cite{Yun2019} by relating both sides of the formula to terms appearing in a certain ``master diagram'' (that will be introduced below) and apply the octahedron lemma of \cite[Theorem A.10]{Yun2017} together with the Lefschetz trace formula.
We only explain how to set up the master diagram in \eqref{TianX} and verify its geometric properties needed to apply the Octahedron Lemma in Section \ref{sec:octahedronlemma} below. 
For the rest of the proof the argument of \cite[Theorem 5.6]{Yun2019} goes through verbatim so we refer to \emph{loc.~cit.} for details.

\subsubsection{The master diagram}

We have the following straightforward generalization of the stacks $H_d(\Sigma)$ together with its Hecke correspondence of \cite[Definition 5.7]{Yun2019} to deeper level.
\begin{defn}
	Let $\wt H_{d}(\Sig)$ be the moduli stack whose $S$-points consist of triples $(\cE^{\da}, \cE'^{\da}, \ph)$ where
	\begin{itemize}
		\item $\cE^{\da}=(\cE, \cE(-\ha \Sig))$ and $\cE'^{\da}=(\cE', \cE'(-\ha \Sig ))$ are $S$-points of $\Bun_{2}(\Sig)$ such that $\deg(\cE'|_{X\times s})-\deg(\cE|_{X\times s})=d$ for all geometric points $s\in S$.
		\item $\ph: \cE\to \cE'$ is a map of coherent sheaves which is injective when restricted to $X\times s$ for all geometric points $s\in S$ and compatible with the level structure in the sense that $\cE(-\ha \Sig)$ is mapped to $\cE'(-\ha \Sig)$.
		\item The restriction $\ph|_{(\Sig\sqcup R)\times S}$ is an isomorphism.
	\end{itemize}
	Moreover, let $H_{d}(\Sig)=\wt H_{d}(\Sig)/\Pic_{X}$ where $\Pic_{X}$ acts by tensoring on $\cE^{\da}$ and $\cE'^{\da}$ simultaneously. 
\end{defn}

As in the Iwahori case we denote the projections to $\cE^{\da}$ and $\cE'^{\da}$ respectively by
\begin{equation*}
	\olr{p_{H}}=(\oll{p_{H}},\orr{p_{H}}): H_{d}(\Sig)\to\Bun_{G}(\Sig)^{2}.
\end{equation*}
We have an Atkin--Lehner operator
\begin{equation}\label{ALH inf}
	\AL_{H,\infty}: H_{d}(\Sig)\times\frSi\to H_{d}(\Sig)
\end{equation}
defined by applying $\AL_{\infty}$ (as defined in Section \ref{sect:al-involution-sht} above) to both $\cE$ and $\cE'$.

Let us also recall the definition of the Hecke correspondence $\Hk^{\un\mu}_{H,d}(\Sig)$  for $H_d(\Sigma)$, which carries over verbatim from the Iwahori case:
\begin{defn} 
	Let $\un\mu\in\{\pm1\}^{r}$. Let $\wt\Hk^{\un\mu}_{H,d}(\Sig)$ 
	whose  $S$-points classify tuples $(\underline{\cE}^\da, \underline{\cE}'^\da, (\ph_i)_{0 \leq i \leq r})$ where
	\begin{itemize}
		\item the two points $\un{\cE}^{\da} = ((\cE^{\da}_{i})_i, (f_{i})_i, (x_i)_i)$ and $\un{\cE}'^{\da} = (({\cE'}^{\da}_{i})_i, (f'_{i})_i, (x_i)_i)$ of $\Hk^{\un\mu}_{2}(\Sig)$ map to the same point $(x_{1},\dotsc, x_{r}) \in X^{r}(S)$, and
		\item for each $i = 0,1,\dotsc, r$ the map $(\cE_{i}, \cE'_{i}, \ph_i) \in \wt H_{d}(\Sig) $
\end{itemize}such that the diagram
\begin{equation}\label{diag EE'}
\xymatrix{\cE_{0}\ar@{-->}[r]^{f_{1}}\ar[d]^{\ph_{0}} & \cE_{1}\ar@{-->}[r]^{f_{1}}\ar[d]^{\ph_{1}}  & \cdots \ar@{-->}[r]^{f_{r}} &  \cE_{r}\ar[d]^{\ph_{r}}\\
	\cE_{0}'\ar@{-->}[r]^{f'_{1}} & \cE_{1}'\ar@{-->}[r]^{f'_{2}} & \cdots \ar@{-->}[r]^{f'_{r}}& \cE_{r}'
}
\end{equation}
is commutative.
	Let
	$	\Hk^{r}_{H,d}(\Sig):=\wt\Hk^{\un\mu}_{H,d}(\Sig)/\Pic_{X}$
	where $\Pic_{X}$ acts on $\wt\Hk^{\un\mu}_{H,d}(\Sig)$ by tensoring on all $\cE^{\da}_{i}$ and $\cE'^{\da}_{i}$. 
\end{defn}

The notation for $\Hk^{r}_{H,d}(\Sig)$ is justified because one can check, as in the Iwahori case, that $\wt\Hk^{\un\mu}_{H,d}(\Sig)/\Pic_{X}$ is canonically independent of $\un\mu$.
We have projections
\begin{equation*}
	p_{H,i}: \Hk^{r}_{H,d}(\Sig)\to H_{d}(\Sig), \quad \text{for }  i=0,\dotsc, r, \quad \text{ and } \quad \olr{q}=(\oll{q}, \orr{q}): \Hk^{r}_{H,d}(\Sig)\to \Hk^{r}_{G}(\Sig)^{2}
\end{equation*}
recording the $i$-th column of the diagram \eqref{diag EE'} and the upper and lower rows of the diagram \eqref{diag EE'}, respectively.
We also consider the base changes 
\begin{equation*}
	\Hk'^{r}_{H,d}(\Sig) \defined \Hk^{r}_{H,d}(\Sig)\times_{X^{r}}X'^{r} \qquad \text{and} \qquad
	\Hk'^{r}_{G}(\Sig) \defined \Hk^{r}_{G}(\Sig)\times_{X^{r}}X'^{r}.
\end{equation*}

In the following we write put the subscript $\frSi'$ to denote the product (over $\Spec k$) with $\frSi'$.
The \emph{master diagram} is then the following commutative diagram:
\begin{equation}
	\label{TianX}
		\begin{tikzcd}
			(\Hk^{\un\mu}_{T}\times\Hk^{\un\mu}_{T})_{\frSi'} 
			\arrow[d, "{(({p^{\un\mu}_{T,0}})^2 \times {\id}_{\frSi'}, \a_{T})}", swap] 
			\arrow[rr, "\th^{\mu}_{\Hk'}\times \th^{\mu}_{\Hk'} \times \id_{\frSi'}"] 
			&&
			\Hk'^{r}_{G}(\Sig)^{2}_{\frSi'} 
			\arrow[d, "{(p'^{2}_{G,0}, \a'_{G})}", swap] 
			&&
			\Hk'^{r}_{H,d}(\Sig)_{ \frSi'}
			\arrow[d, "{(p'_{H,0}, \a_{H})}"]
			\arrow[ll, "\olr{q}' \times\id_{\frSi'}", swap]   
			\\ 
			((\Bun_{T} \times \Bun_{T})_{\frSi'})^2 
			\arrow[rr,  "(\th^{\mu}_{\Bun} \times \th^{\mu}_{\Bun})^2"] 
			&& 
			\Bun_{G}(\Sig)^{2} \times \Bun_{G}(\Sig)^{2} 
			&& 
			H_{d}(\Sig) \times H_{d}(\Sig) 
			\arrow[ll, "\olr{p_{H}} \times \olr{p_{H}}", swap]  
			\\ 
			(\Bun_{T} \times \Bun_{T})_{\frSi'} 
			\arrow[u, "{({\id}, {\Fr})}"] 
			\arrow[rr, "{\th^{\mu}_{\Bun} \times \th^{\mu}_\Bun}"] 
			&& 
			\Bun_{G}(\Sig)^{2} 
			\arrow[u, "{({\id}, {\Fr})}"] 
			&& 
			H_{d}(\Sig) 
			\arrow[ll, "\olr{p_{H}}", swap] 
			\arrow[u, "{({\id}, {\Fr})}", swap]
		\end{tikzcd}
\end{equation}

The map $\a_{T}$ is the composition
	\begin{equation*}
		\Hk^{\un\mu}_{T}\times\Hk^{\un\mu}_{T}\times \frSi' \xr{p^{\un\mu}_{T,r}\times p^{\un\mu}_{T,r}\times\id_{ \frSi'}}\Bun_{T} \times \Bun_{T} \times \frSi' \xr{\AL_{T,\mu_{\infty}}^{(2)}}\Bun_{T} \times \Bun_{T} \times \frSi'
	\end{equation*}
	with $\AL_{T,\mu_{\infty}}^{(2)}$ defined as
	\begin{equation}\label{ALT mu mu'}
		\AL_{T, \mu_{\infty}}^{(2)} \left(\cL_{1},\cL_{2}, (x'^{(1)})_{x \in \Sig_{\infty}} \right)=\left(\cL_{1} \left(-\sum_{x\in\Si}\mu_{1, x}x_1'^{(1)} \right), \cL_{2}\left(-\sum_{x\in\Si}\mu_{2,x}x_2'^{(1)}\right), (x'^{(2)})_{x \in \Sig_{\infty}}\right).
	\end{equation}
	In other words, on the $\frSi'$-factor the map $\a_{T}$ acts as the Frobenius morphism.
	The map $\a'_{G}$ is the composition
	\begin{equation*}
		\Hk'^{r}_{G}(\Sig)^{2}_{\frSi'} \xr{p'^{2}_{G,r}\times \nu_{\infty}}\Bun_{G}(\Sig)^{2}\times\frSi\xr{\AL_{\infty} \times \AL_{\infty}}\Bun_{G}(\Sig)^{2},
	\end{equation*}
	where both copies of $\AL_\infty$ use the same $\frSi$-factor, and $ \nu_\infty \colon \frSi' \to \frSi$ is the projection.
 Similarly, $\a_{H}$ is the composition $\Hk'^{r}_{H,d}(\Sig)_{\frSi'} \xr{p'_{H,r}\times  \nu_{\infty}}H_{d}(\Sig)\times\frSi\xr{\AL_{H,\infty}}H_{d}(\Sig).$

We now consider the fiber products of the columns and rows of the diagram.
We denote by $\Sht^{r}_{H,d}(\Sii)$ the fiber product of the last column of \eqref{TianX}. In other words, the diagram
\begin{equation}\label{defn ShtH}
	\xymatrix{\Sht^{r}_{H,d}(\Sii)\ar[r]\ar[d] & \Hk^{r}_{H,d}(\Sig)_{\frSi'}\ar[d]^{(p_{H,0}, \a_{H})}\\
		H_{d}(\Sig)\ar[r]^-{(\id,\Fr)} & H_{d}(\Sig)\times H_{d}(\Sig)
	}
\end{equation}
is Cartesian by definition. Then the fiber products of the three columns are given by
\begin{equation}\label{diag:hor-maps}
	\xymatrix{\Sht^{\un\mu}_{T}(\mu_{\infty}\cdot\Sig_{\infty}') \times_{\frSi'}\Sht^{\un\mu}_{T}(\mu_{\infty}\cdot\Sig_{\infty}')\ar[r]^-{(\th^{\mu})^2} 
		& \Sht'^{r}_{G}(\Sii)\times_{\frSi'}\Sht'^{r}_{G}(\Sii)
		& \Sht'^{r}_{H,d}(\Sii)\ar[l]}
\end{equation}

We define $\Sht'^{r}_{\cM,d}$ to be the fiber product of this diagram.
In other words, we have by definition a Cartesian diagram
\begin{equation}\label{diag:defn-ShtM}
	\xymatrix{\Sht^{\mu}_{\cM,d}\ar[d]\ar[r] & \Sht'^{r}_{H,d}(\Sii)\ar[d]\\
		\Sht^{\un\mu}_{T}(\mu_{\infty}\cdot\frS_{\infty}')\times_{\frSi'}\Sht^{\un\mu}_{T}(\mu_{\infty}\cdot\frS_{\infty}')\ar[r]^-{(\th'^{\mu})^2} & \Sht'^{r}_{G}(\Sii)\times_{\frSi'}\Sht'^{r}_{G}(\Sii)}.
\end{equation}
As the Hecke correspondence $\Hk^{r}_{H,d}(\Sig)$ preserves the map $s \colon H_{d}(\Sig)\to U_{d}$ that associates to a point $(\Ec^\dagger, \Ec'^\dagger, \ph)$ the divisor of $\det(\ph)$ considered as a section of $\det(\Ec)^{-1} \otimes \det(\Ec')$ while the Frobenius of $H_d(\Sig)$ lifts the Frobenius on $U_d$, we obtain a canonical decomposition indexed by effective divisors of degree $d$ on $U$ as
\begin{equation}\label{decomp ShtH'}
	\Sht'^{r}_{H,d}(\Sii)=\coprod_{D\in U_{d}(k)}\Sht'^{r}_{G}(\Sii;h_{D}),
\end{equation}
as in \cite[(5.22)]{Yun2019} and \cite[Lemma 6.12]{Yun2017}.
We get an induced decomposition
\begin{equation}\label{decomp ShtM D'}
	\Sht^{\mu}_{\cM,d}=\coprod_{D\in U_{d}(k)}\Sht^{\mu}_{\cM,D}.
\end{equation}

Now we consider the fiber product of the three rows of the master diagram \eqref{TianX}.
We define $\cM(\mu_\Sig)$ to be the fiber product of the bottom row of \eqref{TianX}, i.e. it is defined by the Cartesian diagram
	\begin{equation*}
		\xymatrix{\cM_{d}(\mu_{\Sig})\ar[r]\ar[d] &  H_{d}(\Sig)\ar[d]^{\olr{\wt p_{H}}}\\
			\Bun_{T}\times\Bun_{T}\times\frSi'\ar[r]^{(\th^{\mu}_{\Bun})^2} & \Bun_{G}(\Sig)\times\Bun_{G}(\Sig).}
	\end{equation*}
The Atkin-Lehner operators $\AL$ on each of the constituents of the diagram give rise to an Atkin-Lehner operator $\AL_{\cM, \infty}$ on $\cM_d(\mu_\Sig)$. 
In a similar fashion we define its Hecke correspondence $\Hk^{\mu}_{\cM,d}$ via the Cartesian diagram	
\begin{equation}
	\begin{tikzcd}
		\Hk^{\mu}_{\cM,d} \arrow[r] \arrow[d] & \Hk'^r_{H,d}(\Sig) \arrow[d] \\
		(\Hk^{\un \mu}_T)^2_{\frSi'} \arrow[r] & \Hk'^r_G(\Sig)^2
	\end{tikzcd}
	\label{diag:defn-HkM}
\end{equation}
Then we can identify the fiber product of each of the rows of \eqref{TianX} as 
\begin{equation}
	\label{diag:vert-maps}
	\begin{tikzcd}
		\Hk^{\mu}_{\cM,d} \arrow[d, "{p_{\cM, 0}, \alpha_{\cM}}"] \\
		\cM_d(\mu_\Sig) \times \cM_d(\mu_\Sig) \\
		\cM_d(\mu_\Sig) \arrow[u, "{(\id, \Fr)}", swap]
	\end{tikzcd}
\end{equation}

As in \cite[§5.3]{Yun2019} we can identify $\cM_d \times \frSi' \cong \cM_d(\mu_\Sig)$ compatibly with their Hecke correspondences $\Hc^r \times \frSi' \cong \Hk_{\cM, d}^{\mu}$, compare \cite[Propositions 5.15 and 5.16]{Yun2019}.
Then $\Sht'$ is also the fiber product of this diagram of fiber products of the rows in the master diagram, that means that the diagram 
\begin{equation}
	\label{diag:var-ShtM}
	\begin{tikzcd}
		\Sht^{\mu}_{\cM,d} \arrow[r] \arrow[d] & \Hk^{\mu}_{\cM,d} \arrow[d]\\
		\cM_d(\mu_\Sig) \arrow[r, "{(\id,\Fr)}"] & \cM_d(\mu_\Sig) \times \cM_d(\mu_\Sig)
	\end{tikzcd}
\end{equation}
is Cartesian.

%
%
%
%
%
	
\subsubsection{Some geometric facts}
\label{sec:octahedronlemma}
Suppose $D$ is an effective divisor on $U$ of degree $d\ge \max\{2g'-1+N, 2g\}$. 
We collect some geometric facts about the stacks involved in the constructions in the master diagram \eqref{TianX}.
In particular, we verify the assumptions of the Octahedron lemma \cite[Theorem A.10]{Yun2017}, compare \emph{loc.~cit.} for the precise formulation.
In the Iwahori case the corresponding statements are verified in \cite[Proposition 5.18]{Yun2019}. 
We adapt their arguments to the deeper level setting or give references to the relevant statements.
	
	For \textbf{Condition (1)} of \emph{loc.~cit.} we need to verify that all stacks appearing in \eqref{TianX} except for $\Hk'^r_{H,d}(\Sig)_{\frSi'}$ are smooth of the correct dimension.
	For $\Bun_T$ as well as $\Hk_T$ this is well-known, compare for example \cite[Proposition 5.17]{Yun2019}. 
	For $\Bun_G(\Sig)$ and $\Hk_G'(\Sig)$ follows from \thref{prop:geometry-Hecke}.
	The assertion for $H_d(\Sig)$ follows as in \cite[Proposition 5.17 (5)]{Yun2019} by comparison with the situation without level in \cite[Lemma 6.14(1)]{Yun2017}.
	
	For \textbf{Condition (2)} we need to check that some of the fiber products of rows and columns have the expected dimensions. 
	More precisely, we need the assertions for $\Sht'^r_G(\Sii)$, in which case compare \thref{p:bc smooth}, $\Sht'^r_T$, compare \cite[Corollary 4.3]{Yun2019}, as well as for $\cM_d(\mu_\Sig)$ in which case the statement follows from the comparison with $\cM_d$ together with \thref{prop:geometry-Md}.

	For \textbf{Condition (3)} we first note that the diagram \eqref{defn ShtH} satisfies the conditions in \cite[\S A.2.10]{Yun2017} by an argument as in \cite[Lemma 6.14(1)]{Yun2017}, compare also \cite[Proposition 5.18]{Yun2019}.
	
	To check that the diagram \eqref{diag:defn-HkM} satisfies the conditions in \cite[\S A.2.8]{Yun2017} we first note that $\Hk^{\mu}_{\cM,d}$ is a DM stack that admits a finite flat presentation following the argument in the proof of \cite[Proposition 5.18]{Yun2019} and \cite[Proposition 5.5(5)]{Yun2019}, \cite[Proposition 6.1 (1)]{Yun2017}. Namely, the argument in \emph{loc.~cit.} that $\cM_d$ is a Deligne-Mumford stack that admits a finite flat presentation carries over to our setting and the claim follows from the fact that the projection $p_{\cM, 0} \colon \Hk^{\mu}_{\cM, d} \to \cM_d(\mu_\Sig) \to \cM_d$ is schematic. 
	
	For the remaining claim it suffices to show that $\th_{\Hk}^{\mu} \colon  \Hk_T^{\un \mu} \times \frSi' \to \Hk'^r_G(\Sig)$ factors as a regular local immersion and a smooth relative Delgine-Mumford map. 
	We may enlarge $\Sig$ to $\wt\Sig \subset X-R$ such that $\deg\wt\Sig>\r/2$, e.g. by adding a point $y \in |X - \Sig - R|$ with big enough multiplicity (if $\Sig \neq \emptyset$ one could also increase the multiplicity of one of the points in $\Sig$). By enlarging the base field $k$, we may assume that all points in $\nu^{-1}(\wt\Sig)$ are defined over $k$. Choose a section of $\nu^{-1}(\wt\Sig)\to \wt\Sig$ extending the existing section $\mu_{f}$, and call this section $\wt\mu_f$.
	Then map $\th^{\mu}_{\Hk}$
	factors as 
	\[
	\Hk_T^{\un \mu} \times \frSi' \xrightarrow{\th^{\wt \mu}_{\Hk}} \Hk'^r_G(\Sig) \times_{\Bun_G(\Sig)} \Bun_G(\wt \Sig) \to \Hk'^r_G(\Sig)
	\]
	into a regular local immersion and a smooth relative DM map, where the first map is defined as the composition of $\th^{\wt \mu}_{\Hk} \colon \Hk_T^{\un \mu} \times \frSi' \to \Hk'^r_G(\wt \Sig)$ with the forgetful map $\Hk'^r_G(\wt \Sig) \to \Hk'^r_G(\Sig)$.
	The latter asseriontion follows from \thref{lem:BunG-level-smooth}.
	To check that the first map is a regular local immersion we compute the induced map on tangent complexes.
	The proof in \cite[Proposition 5.18]{Yun2019} (respectively \cite[Lemma 6.11(1)]{Yun2017}) goes through after noting that the relative (over $X'^r$) tangent complex of the (smooth DM-)stack $\Hk^{\un \mu}_T$ at a point $b =(\Lc, (x_i)_{i \in I})$ is given by $H^1(X, \Oc_{X'}/\Oc_X)$, and the relative tangent complex of $\Hk'^{r}_G(\Sig;\wt \Sig)$ at the image of $b$ is given by $\cohog{*}{X, \Ad^{\un{x'}, \wt\Sig}(\nu_{*}\cL)}[1]$ and we can embed
		\begin{equation*}
			\Ad^{\un{x'}, \wt\Sig}(\nu_{*}\cL)\subset (\nu_{*}(\cO_{X'}(R'))/\cO_{X})\oplus_{R'}\nu_{*}(\s^{*}\cL^{-1}\ot\cL(R'-\wt\Sig'')).
		\end{equation*} 
%

	For \textbf{Condition (4)} of \cite[Theorem A.10]{Yun2017} we show that the two diagrams \eqref{diag:defn-ShtM} and \eqref{diag:var-ShtM} defining $\Sht'^{\mu}_{\cM,d}$ both satisfy the conditions in \cite[\S A.2.8]{Yun2017}. As in the Iwahori case treated in \cite[Proposition 5.18]{Yun2019}, the argument is analogous to the hyperspecial case in the proof of \cite[Theorem 6.6]{Yun2017}.

	This finishes the verification of the assumptions of the Octahedron lemma in our setting.
	To apply the Octahedron lemma we also need to compute the dimension of $\Hk'^r_{H,d}(\Sig)$, but an argument following the lines of \cite[Proposition 5.17]{Yun2019} as well as \cite[Lemma 6.10 (2)]{Yun2017} shows that it has dimension $2d+2r +3(g-1) + N$.

\section{Spectral decomposition of cohomology and the key identity}
\label{sect:cohomology}

We generalize finiteness results for the cohomology of $\Sht^r_G(\Sii)$ to deeper level.
We closely follow the discussion of \cite[§3.4, §3.5]{Yun2019} in the Iwahori case and explain how to modify their arguments to our more general setting.

\subsection{Horocycles}
\label{sec:truncation}
All the moduli stacks defined in the previous subsection are usually only locally of finite type over $\Fq$. In order to control their cohomology we cut out certain open substacks of finite type.
The key new arguments in this section are the generalizations in \thref{lem:unstable-divisor} of \cite[Lemma 3.19]{Yun2019}  and 
\thref{l:kappa mod} of {\cite[Lemma 3.21]{Yun2019}}. 
Building on these two lemmas, the remaining parts of \cite[§3.4, §3.5]{Yun2019} carry over in a straightforward manner. 

Using the fractional twists as defined above we translate the notion of index of instability of bundles with level structure of \cite{Yun2017, Yun2019} to our setting.
Let us recall the definition.
For an algebraically closed extension $K/k$ and $\Ec \in \Bun_2(K)$ the index of instability of $\Ec$ is $$\inst(\Ec) = \max\{2\deg(\Lc)- \deg(\Ec)\}$$
where the maximum is taken over all line subbundles $\Lc$ of $\Ec$.

Recall that above we constructed fractional twists $\Ec(\frac{D}{2})$ for any divisor $D$ on $X_K$ supported on $\Sig(K)$, where we interpret the fractional twists $\frac{D}{2}$ as in Section \ref{sss:frac-twist}.
In particular, $\inst(\Ec(\frac{D}{2}))$ only depends on the class of $D$ in $\oplus_{x \in \Sig(K)} (\Z/2n_x \Z) [x]$. 

Moreover, for $\Ec^\dagger \in \BunGLnaive(K)$ for an algebraically closed field $K/k$ set 
$$ \inst(\Ec^\dagger) = \min\left\{ \inst\left(\Ec\left(-\frac{D}{2}\right)\right)\colon  D \in \bigoplus_{x \in |\Sig_{\red}(K)|} {\Z}/2n_x\Z [x]\right\}.$$
Then $\Ec$ (or $\Ec^\dagger$) is called (purely) unstable if $\inst(\Ec) > 0$ (respectively $\inst(\Ec^\dagger)>0$).
The notions only depend on the image of $\Ec$ (or $\Ec^\dagger$) in $\Bun_G$ (respectively $\BunGnaive$).

For an unstable $\Fc \in \Bun_2$ there is a unique line subbundle $\Lc$ with $\deg(\Lc)> \frac{\deg(\Ec)}{2}$ called the maximal line subbundle of $\Fc$, let $\cM = \Fc/\Lc$ be its quotient. 
For an effective divisor $D'$ we denote by $\Fc \lrcorner_{D'}$ (respectively $\ulcorner_{D'} \Fc$) the pushout (respectively pullback) of the exact sequence
$$ 0 \to \Lc \to \Fc \to \cM \to 0$$
along $\Lc \to \Lc(D')$ (respectively along $\cM(-D') \to \cM$).

We have the following generalization of \cite[Lemma 3.19]{Yun2019}.
\begin{lemma}
	\thlabel{lem:unstable-divisor}
	Let $K$ be an algebraically closed field containing $k$ and let $\cE^\da \in \BunGnaive(K)$ be purely unstable. 
	\begin{enumerate}
		\item There is a unique $D \in \bigoplus_x (n_x{\Z})/(2n_x\Z) [x]$ with $\inst(\Ec^\dagger) = \inst\left( \Ec \left( \frac{D}{2} \right) \right),$ and the coefficient of $[x]$ in $D$ is given either by $0$ or $n_x$.
		\label{lem:unstable-divisor-unique}
		\item The point $\Ec^\dagger \in \BunGnaive(K)$ is uniquely determined by $\Ec \left( \frac{D}{2} \right) $ for $D$ as in \eqref{lem:unstable-divisor-unique} in the sense that
		$ \Ec(\frac{D+D'}{2}) =\Ec(\frac{D}{2})_{(\frac{D'}{2})}$ for every $D' \in \bigoplus_x \Z/(2n_x)\Z$.
		\label{lem:unstable-divisor-reconstruct}
		\item For any $D' = \sum_x c'_x[x]$ the degree of instability of the fractional twist is given by
		$$ \inst\left(\Ec\left( \frac{D'}{2} \right) \right) = \inst\left(\Ec\left( \frac{D}{2} \right) \right) + |D' - D|$$
		where $|D' - D| = \sum_x \min( |c'_x - c_x|, |2n_x + c'_x - c_x|)$ and $D = \sum_x c_x[x]$ with $0 \leq c_x < 2n_x$.
		\label{lem:unstable-divisor-formula}
	\end{enumerate}
\end{lemma}
\begin{proof}
	We follow the proof of \cite[Lemma 3.19]{Yun2019}, we only sketch the case that $|\Sig_{\red}(K)| = \{x\}$ contains a single point, the general case follows as in \emph{loc.~cit.}.
	Let us pick some $D = c[x]$ such that $\inst(\cE^\da) = \inst(\cE(\frac{D}{2}))$ and set $\cF = \cE(\frac{D}{2})$.
	Then $\inst(\cF(\pm \frac{x}{2})) = \inst(\cF) + 1$ using the minimality of $\cF$. Hence, $\cF(-\frac{x}{2}) = \ulcorner_x\cF$ and $\cF(\frac{x}{2}) = \Fc \lrcorner_x$. 
	If $c \neq 0, n$, this contradicts the assumption that their quotient is free of rank 1 as $K[\varpi_x]/(\varpi_x^2)$-module from our definition of level structure.
	The argument also shows that $\Fc(\frac{c x}{2}) = \Fc \lrcorner_{cx}$ and $\Fc(-\frac{c x}{2}) = \ulcorner_{cx} \Fc$ and hence $\inst(\Fc(\pm \frac{cx}{2})) = \inst(\Fc) + c$ for all $1 \leq c \leq n$ .
\end{proof}

For a purely unstable $\cE^{\da}\in \Bun_{G}(\Sig)(K)$, we may thus define 
$$
	\k(\cE^{\da})=(D, \inst(\cE^{\da}))\in \bigoplus_{x \in |\Sig_{\red}(K)|} {\Z}/2n_x\Z [x]\times\ZZ_{>0},
$$
where $D\in\bigoplus_x {\Z}/2n_x\Z [x]$ is the unique element such that $\inst(\cE^{\da})=\inst(\cE(\ha D))$. 
 
For a positive integer $N>0$ let  ${}^{N}\Bun_{G}$ be the locally closed substack of $\Bun_{G}$ whose geometric points are exactly those $\cE$ with $\inst(\cE)=N$. 
Similarly, for $\k \in \bigoplus_{x \in |\Sig_{\red}(\kbar)|} {\Z}/2n_x\Z [x]\times\ZZ_{>0}$ there is a locally closed substack ${}^{\k}\Bun_{G}(\Sig)\subset \Bun_{G}(\Sig)\ot\kbar$ whose geometric points are exactly those $\cE^{\da}$ with $\k(\cE^{\da})=\k$ (under the identification $\Sig(\kbar)\isom\Sig(K)$).

As in the Iwahori-case, we consider the partial order on $\bigoplus_x {\Z}/2n_x\Z [x]\times\ZZ$ defined by 
\begin{equation*}
	\k=(D,N)\le \k'=(D',N') \quad \text{ if and only if } \quad  N'-N\ge |D-D'|.
\end{equation*}
For $\k=(D,N)\in\bigoplus_x {\Z}/2n_x\Z [x]\times\ZZ_{>0}$, let ${}^{\le\k}\Bun_{G}(\Sig)\subset \Bun_{G}(\Sig)\ot\kbar$ be the open substack consisting of $\cE^{\da}$ such that for any $D'\in\bigoplus_x {\Z}/2n_x\Z [x]$, $\inst(\cE(\ha D'))\le N+|D'-D|$. We see that ${}^{\k}\Bun_{G}(\Sig)\subset{}^{\le\k'}\Bun_{G}(\Sig)$ if and only if $\k\le \k'$. Moreover,  ${}^{\k}\Bun_{G}(\Sig)$ is closed in ${}^{\le\k}\Bun_{G}(\Sig)$, and we denote its open complement by ${}^{<\k}\Bun_{G}(\Sig)$.

For two elements $\k=(D,N),\k'=(D',N')\in \bigoplus_x {\Z}/2n_x\Z [x]\times\ZZ_{>0}$ we define
\begin{equation*}
	|\k-\k'|:=|D-D'|+|N-N'|\in\ZZ_{\ge0}.
\end{equation*}

\begin{cor}[{cf. \cite[Corollary 3.20]{Yun2019} for $\Sig$ reduced}] \label{c:kBunSig} 
For $\k=(D, N)\in \bigoplus_x n_x{\Z}/2n_x\Z [x]\times\ZZ_{>0}$, the map $\cE^{\da}\mapsto \cE(\ha D)$ gives an isomorphism of $\kbar$-stacks
	\begin{equation*}
		{}^{\k}\Bun_{G}(\Sig)\isom {}^{N}\Bun_{G}\ot\kbar.
	\end{equation*}
\end{cor}

We recall how $\kappa$ can change under elementary modifications.
\begin{lem}[{\cite[Lemma 3.21]{Yun2019}}]
	\thlabel{l:kappa mod}
	 Suppose $(\cE^{\da},\cF^{\da},y, \ph)\in \Hk^{1}_{G}(\Sig)(K)$ (where $K$ is an algebraically closed field, $\cE^{\da},\cF^{\da}$ are lifted to $\Bun_{2}(\Sig)(K)$, $\ph: \cE\incl \cF$ and $y$ is the support of $\coker(\ph)$), and  $\cE^{\da}$ and $\cF^{\da}$ are both purely unstable. Write $\k(\cE^{\da})=(D,N)$, $\k(\cF^{\da})=(D',N')$. 
		\begin{enumerate}
			\item $|\k(\cE^{\da})-\k(\cF^{\da})|=1$.
			\item If  $N=N'$, then $D$ and $D'$ differ at a unique point $x\in \Sig_{\Iw}(K)$, and we have $y=x$. The points $\cE^{\da}$ and $\cF^{\da}$ are uniquely determined by the triple $(\cE(\ha D), \cF(\ha D'), \a)$ where $\a$ is an isomorphism of $G$-bundles
			\begin{equation*}
				\a: \cE(\ha D)\lrcorner_{x}\cong \cF(\ha D')\lrcorner_{x}. 
			\end{equation*}
			
			\item If $N=N'-1$, then $D=D'$, and $\cE^{\da}$ and $\cF^{\da}$ are determined by the single bundle $\cE(\ha D)$ in the following way: $\cE^{\da}$ is determined by $\cE(\ha D)$ as in Lemma \ref{lem:unstable-divisor}(2); $\cF(\ha D)=\cE(\ha D)\lrcorner_{y}$ and $\cF^{\da}$ is determined by $\cF(\ha D)$ again by Lemma \ref{lem:unstable-divisor}(2).
			\item If $N=N'+1$, then $D=D'$, and $\cE^{\da}$ and $\cF^{\da}$ are determined by the single bundle $\cF(\ha D)$ in the following way: $\cF^{\da}$ is determined by $\cF(\ha D)$ as in Lemma \ref{lem:unstable-divisor}(2); $\cE(\ha D)=\ulcorner_{y}(\cF(\ha D))$ and $\cE^{\da}$ is determined by $\cE(\ha D)$ again by Lemma \ref{lem:unstable-divisor}(2). 
		\end{enumerate}
\end{lem}
\begin{proof}
	We sketch the argument following \cite[Lemma 3.21]{Yun2019}. Namely we first note that $N - N' \in \{0, \pm 1\}$ by construction.
	The cases $N = N' \pm 1$ are treated as in \emph{loc.~cit.}. 
	Moreover, if $N = N'$, then $D$ and $D'$ differ at a unique point $x = y$ we have since the coefficients of $x$ in both $D$ and $D'$ is either $0$ or $n_x$ together with $|\kappa - \kappa'| = |D - D'| = 1$ that $n_x = 1$. 
	In other words, $x \in \Sig_\Iw(K)$ is a reduced point of $\Sig$.
	Now we can continue as in \emph{loc.~cit.}  
\end{proof}
As a consequence we get a description of the Hecke stack as in \cite[Corollary 3.22]{Yun2019}.

Let us now generalize the definition of horocycles.
We abbreviate notation for the rest of this section and write $\Sht \defined \Sht^r_G(\Sii) \times_k \Spec(\kbar)$ for an algebraic closure $\kbar$ of $k$.

\begin{defn}[{\cite[Definition 3.23]{Yun2019}}]\label{def: horo}
	Let $\un\k=(\k_{0},\k_{1},\dotsc, \k_{r})$ be a sequence of elements in 
	$$ \bigoplus_{x \in |\Sig_{\red}(\kbar)|} {\Z}/2n_x\Z [x]\times \ZZ_{>0}.$$ 
	\begin{enumerate}
		\item The {\em horocycle of type $\un\k$} of $\Sht$ is the locally closed substack ${}^{\un\k}\Sht\subset \Sht$ whose geometric points are exactly those $(\cE^{\da}_{i};\dotsc)\in \Sht$ such that each $\cE^{\da}_{i}$ is purely unstable with $\k(\cE_{i}^{\da})=\k_{i}$, for $i=0,1,\dotsc, r$.
		\item The {\em truncation up to $\un\k$} of $\Sht$ is the open substack of $\Sht$ consisting of $(\cE^{\da}_{i};\dotsc)$ such that $\cE^{\da}_{i}\in {}^{\le \k_{i}}\Bun_{G}(\Sig)$ for all $0\le i\le r$. 
	\end{enumerate}
\end{defn}
Then the horocycle ${}^{\un\k}\Sht$ is closed in ${}^{\le\un\k}\Sht$ and we denote its open complement by ${}^{<\un\k}\Sht$. 
We have the following necessary criterium for horocycles to be non-empty.

\begin{lemma}[{\cite[Lemma 3.24]{Yun2019}}]
	\thlabel{l:nonempty horo} 
	Let $\un\k=(\k_{0},\k_{1},\dotsc, \k_{r})$ with $\k_{i}=(D_{i}, N_{i})$ be a sequence of elements in 
	$$ \bigoplus_{x \in |\Sig_{\red}(\kbar)|} {\Z}/2n_x\Z [x]\times \ZZ_{>0}.$$ 
	If ${}^{\un\k}\Sht$ is non-empty, then
	\begin{enumerate}
		\item $D_i \in \bigoplus_{x \in |\Sig_{\red}(\kbar)|} (n_x{\Z}/2n_x\Z) [x]\times \ZZ_{>0}$ and $N_{0}=N_{r}$,
		\item for each $i=1,\dotsc, r$, $|\k_{i-1}-\k_{i}|=1$,
		\item and $\Fr(D_{0})$ (applying the arithmetic Frobenius to each point appearing $D_{0}$) and $D_{r}$ differ at exactly one $\kbar$-point above each place of $\Si$ and nowhere else.  
	\end{enumerate}
\end{lemma}
\begin{proof}
	As in \emph{loc.~cit.} this follows from \thref{lem:unstable-divisor} and \thref{l:kappa mod}.
\end{proof}

\begin{defn}\label{defn:Kr} Let $\frK_{r}$ be the set of tuples $\un\k=(\k_{0},\k_{1},\dotsc, \k_{r})$ satisfying the conditions in Lemma \ref{l:nonempty horo}, where each $\k_{i}\in\bigoplus_{x} ({\Z}/2n_x\Z) [x]\times \ZZ$.
	The partial order on $ \bigoplus_{x} {\Z}/2n_x\Z [x]\times\ZZ_{>0}$ extends to one on $\frK_{r}$ by stipulating that $(\k_{0},\dotsc, \k_{r})\le (\k'_{0},\dotsc, \k'_{r})$ if and only if $\k_{i}\le\k'_{i}$ for all $0\le i\le r$.
\end{defn}

As in the Iwahori case we have for two tuples  $\un\k,\un\k'\in \frK_{r}$ that ${}^{\un\k}\Sht\subset {}^{\le\un\k'}\Sht$ if and only if $\un\k\le \un\k'$.
Then it follows from the previous lemma that 
\[
	\Sht=\bigcup_{\un\k\in \frK_{r}}{}^{\le\un\k}\Sht.
\] 
For $\un\k\in\frK_{r}$ and $N\in \ZZ$, we write $\un\k>N$ if $N_{i}(\un\k)>N$  for all $0\le i\le r$. 
Moreover, we set $\frK^{\sh}_{r} = \{\un\k\in \frK_{r}|\un\k>\max\{2g-2,0\}\}$ and
\begin{equation*}
	{}^{\sh}\Sht = \cup_{\un\k\in\frK^{\sh}_{r}}{}^{\un\k}\Sht.
\end{equation*}
Then ${}^{\sh}\Sht$ consists of $(\cE^{\da}_{i};\dotsc)$  where all $\inst(\cE^{\da}_{i})> \max\{2g-2,0\}$, therefore it is a closed substack of $\Sht$.  Let ${}^{\fl}\Sht=\Sht-{}^{\sh}\Sht$ be its open complement. 
The following results can be shown as in \emph{loc.~cit.}.
\begin{lemma}[{\cite[Lemma 3.30]{Yun2019}}]\label{l:ft}
	The substack ${}^{\fl}\Sht$ is of finite type over $\kbar$. 
\end{lemma}

For $\un\k=(\k_{0},\dotsc, \k_{r})\in \frK_{r}$ with $\un\k_{i}=(D_{i}, N_{i})$, we define an associated index set $I(\un\k)\subset \{1,2,\dotsc, r\}$ by 
\begin{equation*}
	I(\un\k)=\{1\le i\le r| N_{i-1}\ne N_{i}\}.
\end{equation*}
By definition for $i\in\{1,2,\dotsc, r\} \setminus I(\un\k)$ there is a unique reduced point $x_{i}(\un\k) \in \Sig_\Iw(\kbar)$ such that $D_{i-1}$ and $D_{i}$ differ at $x_{i}(\un\k)$. 
Moreover, we denote the unique $\kbar$-point over each $x\in\Si$ where $D_{r}$ and $\Fr(D_{0})$ differ (following the last condition in \thref{l:nonempty horo}) by  $x^{(1)}(\un\k)$.
As for every $i\in I(\un\k)$ we have $N_{i}=N_{i-1}\pm1$ and $N_{r}=N_{0}$, we see that $\#I(\un\k)$ is even.
We define $X(\un\k)\subset (X^{r}\times\frSi)\ot\kbar$ to be the corresponding subspace
\begin{equation*}
	X(\un\k) \defined \{(x_{1},\dotsc, x_{r}, (x^{(1)})_{x\in\Si}) \colon  x_{i}=x_{i}(\un\k) \text{ for all } i \notin I(\un\k); \, x^{(1)}=x^{(1)}(\un\k)\text{ for all } x \in \Si\}.
\end{equation*}
The projection to the $I(\un\k)$-coordinates gives an isomorphism
\begin{equation*}
	X(\un\k)\isom X^{I(\un\k)}\ot\kbar.
\end{equation*}

%

\begin{lem}[{\cite[Corollary 3.27]{Yun2019}}] 
	For $\un\k\in\frK_{r}$ and $\un\k>0$, the restriction of the map $\Pi^{r}_{G}: \Sht\to (X^{r}\times \frSi) \otimes_k \kbar$ to ${}^{\un\k}\Sht$ has image in $X(\un\k)$.
\end{lem}
 We denote the resulting map by $\pi_{\un\k}: {}^{\un\k}\Sht\to X(\un\k).$
Using \thref{l:kappa mod} we can construct a map
\begin{equation*}
	q_{\un\k}: {}^{\un\k}\Sht\to \Sht^{N(\un\k)}_{1}\ot\kbar
\end{equation*}
such that $\pi_{\un \k}$ factors through $q_{\un \k}$.
As we do not need the concrete description of this map we refer to \cite[§3.4.6]{Yun2019} for the details.
The proofs of its following properties carry over verbatim from \emph{loc.~cit.}.
\begin{lemma}[{\cite[Lemma 3.28]{Yun2019}}]\label{l:horo fib} Suppose $\un\k\in \frK_{r}$ and $\un\k>\max\{2g-2,0\}$. Then the map $q_{\un\k}$ is smooth of relative dimension $r-\#I(\un\k)/2$. The geometric fibers of $q_{\un\k}$ are isomorphic to $[\Ga^{r-\#I(\un\k)/2}/Z]$ for some finite \'etale group scheme $Z$ acting on $\Ga^{r-\#I(\un\k)/2}$ via a homomorphism $Z\to \Ga^{r-\#I(\un\k)/2}$.
\end{lemma}

\subsection{Finiteness results for (perverse) cohomology sheaves under the Hecke action}

We discuss finiteness properties of the Hecke action on perverse cohomology sheaves of $\Sht$ following \cite[§3.5]{Yun2019}.
As all their arguments go through in our setting as well (using the generalized results on horocycles as in the previous subsection), we do not give proofs here but rather only refer to  \emph{loc.~cit.}

\begin{cor}[{\cite[Corollary 3.29]{Yun2019}}]\label{c:Pk perv} Suppose  $\un\k\in \frK_{r}$ and $\un\k>\max\{2g-2,0\}$. Let $\pi_{1}^{N(\un\k)}: \Sht^{N(\un\k)}_{1}\ot\kbar \to X(\un\k)$ be the projection. Then we have a canonical isomorphism
	\begin{equation*}
		\bR \pi_{\un\k,!}\Ql\cong \bR\pi_{1,!}^{N(\un\k)}\Ql[-2r+\#I(\un\k)](-r+\#I(\un\k)/2).
	\end{equation*}
	In particular, $\bR \pi_{\un\k,!}\Ql$ is a local system shifted in degree $2r-\#I(\un\k)$, and
	\begin{equation}\label{defn Pk}
		P_{\un\k}:=\bR \pi_{\un\k,!}\Ql[2r](r)\in D^{b}(X(\un\k),\Ql)
	\end{equation}
	is a perverse sheaf on $X(\un\k)$ with full support and pure of weight 0.
\end{cor}
We consider the cohomology of the moduli stack of shtukas 
\begin{equation*}
	V=\cohoc{2r}{\Sht,\Ql}(r) = \varinjlim_{\un\k\in\frK_{r},\un\k>0}\cohoc{2r}{{}^{\le \un\k}\Sht,\Ql}(r).
\end{equation*}
We consider the perverse cohomology sheaves, and similarly we also consider the perverse cohomology sheaves 
\[
	{}^p\coho{i}{K} \defined \varinjlim_{\un\k\in\frK_{r},\un\k>0} {}^p\coho{i}{\bR \pi_{\leq \un \kappa, !} \Ql[2r](r)} \in \mathrm{indPerv}((X^r \times \frSi)_{\kbar}, \Ql)
\]
The perverse chomology sheaves carry a Hecke action induced from the Hecke correspondences analogous to the construction in Section \ref{sss:Hk Chow}.
Have a cohomological constant term operator 
\[
	\g \colon {}^p\coho{0}{K} \to {}^p\coho{0}{K_\sh} \to \bigoplus_{\un \k \in \frK_{r}^{\sh}} P_{\un \k}
\]
defined as \cite[§3.5.3]{Yun2019}. In particular, by \cite[Lemma 3.34]{Yun2019} the first map is an isomorphism modulo constructibles (in the sense of \cite[Definition 3.32]{Yun2019}) while the second map is an isomorphism.
As in \cite[§3.5.4]{Yun2019} we obtain an mc-action of $\Pic_X(k)$ on $\bigoplus_{\un \k \in \frK_{r}^\sh} P_{\un \k}$ such that the cohomological constant term intertwines the Hecke action with the $\Pic_{X}(k)$-action under the map $\wt a^{\Sig}_{\mathrm{Eis}}$ as in \cite[Lemma 3.36]{Yun2019}.

We define $\ov\sH^{\Sig}_{\ell}$ to be the image of the ring homomorphism
\begin{equation*}
	\sH^{|\Sig|}_{G}\ot\Ql\to \End_{\Ql}(V)\times\Ql\left [\Pic_{X}^{\Sigha}(k) \right]^{\io_{\Pic}}
\end{equation*}
given by the product of the action on $V$ and $\wt a^{\Sig}_{\Eis}$.
Using finiteness properties of the Hecke action on the perverse cohomology sheaves we obtain finiteness properties of the Hecke action on $V$.
\begin{proposition}[{\cite[Corollary 3.38, Corollary 3.40]{Yun2019}}]
	\thlabel{prop:finiteness-coh}
	\begin{enumerate}
		\item 
		If $f\in \cI_{\Eis}$, then the image of the Hecke action $f: V\to V$ is finite-dimensional. 
		\item $\ov\sH^{\Sig}_{\ell}$ is a finitely generated $\Ql$-algebra of Krull dimension one.
		\item $V$ is finitely generated as a $\ov\sH^{\Sig}_{\ell}$-module.
	\end{enumerate}
\end{proposition}

\subsection{Cohomological spectral decomposition} 
\label{ss:coho spec decomp}
We explain how to generalize the cohomological spectral decomposition to deeper level. 
As before, the arguments of \cite[§3.5]{Yun2019} carry over verbatim using the results of the previous subsections.

\begin{theorem}[Cohomological spectral decomposition, {cf. \cite[Theorem 3.41]{Yun2019}}]
	\thlabel{th:spec decomp}
	\begin{enumerate}
		\item There is a decomposition of the reduced scheme of $\Spec \ov\sH^{\Sig}_{\ell}$ into a disjoint union
		\begin{equation*}
			\Spec \left(\ov\sH^{\Sig}_{\ell}\right)^{\red}=Z_{\Eis,\Ql}\coprod Z^{\Sig}_{0,\ell}
		\end{equation*}
		where $Z_{\Eis,\Ql}=\Spec \Ql\left [\Pic_{X}^{\Sigha}(k) \right]^{\io_{\Pic}}$ and $Z^{\Sig}_{0,\ell}$ consists of a finite set of closed points. 
		\item There is a unique decomposition
		\begin{equation*}
			V=V_{0}\oplus V_{\Eis}
		\end{equation*}
		into $\sH^{|\Sig|}_{G}\ot\Ql$-submodules, such that $\Supp(V_{\Eis}) \subset Z_{\Eis,\Ql}$, and $\Supp(V_{0}) = Z^{\Sig}_{0,\ell}$.
		\item The subspace $V_{0}$ is finite dimensional over $\Ql$.
	\end{enumerate}
\end{theorem}

The analogous statements also hold for the base change $V' \defined \cohoc{2r}{\Sht',\Ql}(r)$, and also for $V'(\xi)$. 
In particular, this proves \thref{thm:spectral-dec}.
We can apply the previous theorem in particular to the case $r =  0$. 
In this situation, we can identify the stack $\Sht_G^0(\Sig)$ with the discrete groupoid 
\[
	\Bun_G(k) \cong G(F)\backslash G(\A) / K_0(\Sig),
\]
where $K_0(\Sig) = \prod_{x \not \in \Sig} G(\Oc_x) \times \prod_{x \in \Sig} \Gamma_0(\pf_x^{n_x})$ as above.
Let $\Ac(K_0(\Sig))$ be the space of $\Q$-valued automorphic forms with level $K_0(\Sig)$, i.e. compactly supported ($\Q$-valued) functions on the double coset $G(F)\backslash G(\A) / K_0(\Sig)$. 
We obtain an isomorphism
\[
	\Ac(K_0(\Sig)) \otimes_\Q \Q_\ell \cong H^0_c(\Sht^0_G(\Sig), \Q_\ell).
\]
In this case, \thref{th:spec decomp} yields the decomposition
\[
	\Ac(K_0(\Sig)) \otimes_\Q \Qlbar = \Ac(K_0(\Sig))_{\mr{Eis}} \oplus \bigoplus_{\pi \in \Pi_\Sig(\Qlbar)} \Ac(K_0(\Sig))_\pi,
\]
where $\Pi_\Sig(\Qlbar)$ is the set of cuspidal automorphic representations of level $K_0(\Sig)$. 
By strong multiplicity one, we may identify $\Pi_\Sig(\Qlbar)$ with a subset of $\Spec(\sH_G^{\Sig})$ via the character $\l_\pi$ associated to $\pi \in \Pi_\Sig(\Qlbar)$.

\subsection{Proof of the main theorem}
%
We follow \cite[Section 9]{Yun2017} and \cite[Section 7]{Yun2019} to finish the proof of the main theorem.
Let 
\[
	\wt \sH^\Sig_\ell \defined \Im\left( \sH^{\Sig}_G \otimes_Q \Ql \longrightarrow \End_{\Ql}(V'(\xi)) \oplus \End_{\Ql}(\Ac(K_0(\Sig)))_{\Ql} \oplus \Ql[\Pic_{X}^{\Sigha}(k)]\right).
\]
As for  $\ov\sH^{\Sig}_{\ell}$ we also get that $\wt \sH^\Sig_\ell$ is a finitely generated $\Ql$-algebra of Krull dimension 1. 
By \thref{th:spec decomp} we get a decomposition
\begin{equation}
	\label{eq:dec-hecke-tilde}
	\Spec\left(\wt \sH^\Sig_\ell\right)^{\mathrm{red}} = Z_{\mr{Eis}} \amalg \Yc_0,
\end{equation}
where $\Yc_0$ is a finite set of closed points.

\subsubsection{The key identity} We comment on how to generalize the key identity \cite[(1.9), Theorem 9.2]{Yun2017} and \cite[Theorem 7.5]{Yun2019} to our setting. 
As noted before, the study of the two Hitchin-type fibrations in \cite{Yun2019} does not use in an essential way that the level is square-free, so the results carry over to our setting, using the appropriate modifications in their definitions as introduced above. 

\begin{lemma}[{cf. proof of \cite[Theorem 9.2]{Yun2017}}]
	\thlabel{lem:key-identity-finiteness}
	Let $\II$ and $\JJ$ be two linear functionals on $\wt \sH_G^{\Sig}$ and assume that $\II$ and $\JJ$ agree on the images of $h_D$ for all $D \se X - \Sig$  with $\deg D \geq n_0$ for some positive integer $n_0$. 
	Then $\II = \JJ$. 
\end{lemma}
\begin{proof}
	We follow the corresponding part of the proof of \cite[Theorem 9.2]{Yun2017}.
	Let $\sH^\dagger \se \sH^{\Sig}_G$ be the subspace spanned by the $h_D$ for $D$ as in the assertion. 
	It suffices to show that the map $a \colon \sH^\dagger \otimes_\Q \Ql \to \wt \sH^\Sig_\ell$ is surjective.
	As $\wt \sH_\ell^\Sig$ is finitely generated as $\Ql$-algebra, there is a finite set of places $S \se X - \Sig$ such that $\wt \sH^\Sig_\ell$ is generated by $a(h_s)$ for $s\in S$. 
	We may assume that $S$ contains all places $s$ of degree at most $n_0$. 
	As in \emph{loc.~cit.} we see that the image of $\sH^\dagger$ in $\wt \sH^\Sig_\ell$ contains the ideal $I$ generated by $a(h_y)$ for all $y \in |X| - (S \cup |\Sig|)$. 
	
	As a next step we show that the quotient $\wt \sH^\Sig_\ell/I$ is a finite dimensional $\Ql$-vector space. 
	Using \eqref{eq:dec-hecke-tilde} together with \thref{th:spec decomp} it suffices to show that $\Spec(\wt \sH^\Sig_\ell/I)$ is disjoint from $Z_{\mr{Eis}}$.
	Assume there is a $\Qlbar$-point $\sigma \in \Spec(\wt \sH^\Sig_\ell/I) \cap Z_{\mr{Eis}}$. 
	Then the composition $\sH^{\Sig}_G \to \wt \sH^\Sig_\ell/I \xrightarrow{\sigma} \Qlbar$ factors through a character $\chi \colon \Q[\Pic^{\Sigha}_X(k)] \to \Qlbar$.
	After renormalizing $\chi$ by $\chi' \defined q^{-\deg/2}\chi$ we get a contradiction with Chebotarev density.
	
	We conclude as in  \cite[Theorem 9.2]{Yun2017} by applying \cite[Lemma 9.1]{Yun2017}.
\end{proof}
\begin{theorem}[Key identity]
	\thlabel{key}
For every $f \in \sH^{|\Sig|}$
\begin{equation}
	\II^\mu(f) = (-\log q)^{-r} \JJ^r(f).
	\label{eq:key-id}
\end{equation}	
\end{theorem}
\begin{proof}
	Both $\II^{\mu}(f)$ and $\JJ^{\mu}(f)$ clearly depend only on the image of $h$ in $\wt \sH^{\Sig}_{\ell}$.
	We first consider the case $f = h_D$ for $D$ and effective divisor $D \in \Div(X - \Sig)$ with $\deg D \geq \max\{2g- 2 \}$. 
	In this case have identified both as certain traces of Frobenius. 
	In both cases we can interpret the resulting sheaves in terms of representations of a certain finite group as in \cite[§7.1]{Yun2019}.
	This can then be used to compute and compare the traces as in \cite[Theorem 7.4]{Yun2019}, which shows the key identity for all $f = h_D$ as above.
	The full statement can then be deduced from \thref{lem:key-identity-finiteness}.
\end{proof}

\subsubsection{Proof \thref{thm:int-period}}
Our main theorem now follows as in the proof of \cite[Theorem 1.2]{Yun2019}. We sketch the argument.
Let $\cY=\Spec \wt \sH^{\Sig}_{\ell}$. 
Under the decomposition \eqref{eq:dec-hecke-tilde}, the quotient $\wt \sH^{\Sig}_{\ell}$ of the Hecke algebra decomposes as 
\begin{equation}\label{prod decomp H}
	\wt \sH^{\Sig}_{\ell}=\wt \sH^{\Sig}_{\ell,\Eis}\times \wt \sH^{\Sig}_{\ell,0}
\end{equation}
such that $\Spec \wt \sH^{\Sig,\red}_{\ell,\Eis}=Z_{\Eis, \Ql}$ and $\Spec \wt \sH^{\Sig,\red}_{\ell,0}=\cY_{0}$. We have a decomposition
\begin{equation*}
	V'(\xi)\ot\Qlbar=V'(\xi)_{\Eis}\ot\Qlbar\op(\oplus_{\fkm\in \cY_{0}(\Qlbar)}V'(\xi)_{\fkm})
\end{equation*}
where $\Supp(V'(\xi)_{\Eis})\subset Z_{\Eis,\Ql}$ and $V'(\xi)_{\fkm}$ is the generalized eigenspace of $V'(\xi)\ot\Qlbar$ under the character $\fkm$ of $\wt\sH^{\Sig}_{\ell}$. Under this decomposition, let $Z^{\mu}_{\fkm}(\xi)$ be the projection of $Z^{\mu}(\xi)\in V'(\xi)$ to  $V'(\xi)_{\fkm}$.

Let $h\in \wt \sH^{\Sig}_{\ell,0}$, viewed as $(0,h)\in \wt \sH^{\Sig}_{\ell}$ under the decomposition \eqref{prod decomp H}. Since the $\sH^{|\Sig|}_{G}$-action on $V'(\xi)$ is self-adjoint with respect to the cup product pairing, we have the spectral decomposition of the intersection form
\begin{equation}\label{Ih}
	\II^{\mu}(h)=\sum_{\fkm\in \cY_{0}(\Qlbar)}(Z^{\mu}_{\fkm}(\xi), h*Z^{\mu'}_{\fkm}(\xi)).
\end{equation}

Let $\pi$ be as in the statement of \thref{thm:int-period} and let $e_{\pi}$ be the idempotent in $\wt \sH^{\Sig}_{\ell,0}\ot\Qlbar$ corresponding to $\pi\in \Pi_{\Sig}(\Qlbar)\subset \cY_{0}(\Qlbar)$. 
If we plug $h=e_{\pi}$ into \eqref{Ih} and the spectral decomposition of $\JJ$ of \thref{prop:spectral-dec-J} together with \eqref{eq:spectral-dec-J}, \thref{thm:int-period} follows from \thref{key} for $e_\pi$.

\bibliographystyle{alphaurl}
\bibliography{../literature}

\end{document}